\documentclass{siamart220329}
\usepackage[utf8]{inputenc}
\usepackage{eufrak}
\usepackage{color}

\usepackage{subfigure}
\usepackage{hyperref}

\usepackage{latexsym}
\usepackage{multicol}
\usepackage{graphicx}
\usepackage{color}
\usepackage{appendix}
\usepackage{hyperref}
\usepackage{bbold}
\usepackage{bm}

\def\R{\mathbb{R}}
\def\N{\mathbb{N}}

\def\P{\mathcal{P}}
\def\Pr{\mathcal P_2(\D)}
\def\D{\Omega}
\def\d{d}

\def\Asym{\mathcal A_{\rm sym}}
\def\A{\mathcal A}
\def\M{\mathcal M}
\def\GammaM{\Gamma_{\mathcal M}}
\def\GammaQM{\Gamma_{Q,{\mathcal M}}}

\def\Md{\delta_{\mathcal M,p}}
\def\Mdt{\widetilde{\delta}_{{\mathcal M},p}}
\def\Mdq{\delta_{Q,{\mathcal M},p}}
\def\AMdq{\delta_{Q,{\mathcal A},p}}
\def\AMdqtwo{\delta_{Q,{\mathcal A},2}}
\def\Mdtq{\widetilde{\delta}_{Q,{\mathcal M},p}}

\newtheorem{assumption}[theorem]{Assumption}
\newtheorem{remark}[theorem]{Remark}
\newtheorem{example}[theorem]{Example}

\title{A Wasserstein-type metric for generic mixture models, including 
location-scatter and 
group invariant measures}

\author{Geneviève Dusson, Virginie Ehrlacher, Nathalie Nouaime}

\author{ Genevi\`eve Dusson\thanks{Laboratoire de Math\'ematiques de Besan\c{c}on, UMR CNRS 6623,
  Universit\'e de Franche-Comt\'e
  (\email{genevieve.dusson@math.cnrs.fr})}
\and Virginie Ehrlacher\thanks{CERMICS, \'Ecole des Ponts \& Inria Paris
  (\email{virginie.ehrlacher@enpc.fr})}
\and Nathalie Nouaime\thanks{CEA Saclay, DEN/DM2S/STMF/LMSF (nathalie.nouaime@cea.fr)}
}

\date{\today}

\begin{document}

\maketitle

\begin{abstract}
In this article, we study Wasserstein-type metrics and corresponding barycenters for mixtures of a chosen subset of probability measures called atoms hereafter. 
In particular, this works extends what was proposed by Delon and Desolneux~\cite{Delon2020-wk} for mixtures of gaussian measures to other mixtures. 
We first prove in a general setting that for a set of atoms equipped with a metric that defines a geodesic space, the set of mixtures based on this set of atoms is also geodesic space for the defined modified Wasserstein metric. We then focus on two particular cases of sets of atoms: (i) the set of location-scatter atoms and (ii) the set of measures that are invariant with respect to some symmetry group. Both cases are particularly relevant for various applications among which electronic structure calculations. Along the way, we also prove some sparsity and symmetry properties of optimal transport plans between measures that are invariant under some well-chosen symmetries. 
\end{abstract}



\section{Introduction}

The original motivation of this work stems from electronic structure calculations in quantum chemistry, which are widely used for the simulation of molecular and material systems~\cite{Cances2003-jm,Helgaker2014-py}. Permutation-invariant probability measures naturally arise in this context, as the square modulus of wavefunctions and their marginals satisfy such invariance property, due to the indistinguishability of bosonic and fermionic particles, e.g. electrons.  In general, these objects are moreover parame\-trized by the positions of the nuclei in the molecular system, and efficient ways to interpolate between such objects when the positions of the nuclei are changing would be extremely beneficial to this field, where the computations at stake are numerically very expensive, as they involve the resolution of high-dimensional and/or nonlinear eigenvalue partial differential equations. 
Due to the localized nature of the considered probability measures, optimal transport seems a natural way to interpolate between them using Wasserstein barycenters developed by Carlier and Agueh~\cite{Agueh2011-uz}, as shown in~~\cite{DussonFries} for a toy model with one-dimensional particles. 
However due to the high-dimensionality of the considered measures for real systems (e.g. already $3$-dimensional for the electronic density, and $6$-dimensional for the electronic pair density), this does not seem feasible in practice using standard optimal transport algorithms, such as the now widely used Sinkhorn algorithm~\cite{Sinkhorn1967-rk}, see~\cite{Peyre2019-yf} for a monograph on the numerical aspects of optimal transport describing this algorithm.

Luckily, in~\cite{Delon2020-wk}, Delon and Desolneux proposed a Wasserstein-type distance and interpolation scheme based on the decomposition of probability measures as mixtures of gaussian distributions.
The modified optimal transport problem expressed in this context becomes independent of the underlying dimension of the distribution, or of its discrete spatial representation, so that the problem becomes extremely cheap to solve when the considered measures can be decomposed as convex combinations of a few gaussian distributions. 
However, we encounter two limitations when trying to apply this strategy in the context of electronic structure calculations. 
First, for some models, the wave-functions are not as regular as gaussians, and cusps are to be found around the nuclei positions in the molecular system. 
Therefore, the wave-function or its marginals would be better represented as a mixture of Slater-type functions, i.e. based on $\exp(-\alpha | x |)$, see~\cite[Figure 2]{Pham2017-bd} for an example on the $H_2^+$ molecule which is composed of two protons and one electron. 
Unfortunately these functions do not fit in the framework of~\cite{Delon2020-wk}. 
Second, when considering fermions, the wave-function is anti-symmetric, so that its modulus square as well as its marginals is not only permutation symmetric, but also contains constraints coming from the anti-symmetry property. For example, the modulus square of the wave-function is zero when two variables are equal. 
This feature cannot be satisfied as such with convex combinations of strictly positive measures.
Note that generic distributions as well as group symmetric mixtures appear in other contexts, such as portfolio theory in finance~\cite{Gupta2013-qk}, or in image analysis where the described objects are defined up to rigid movements~\cite{Mehr2018-zi}. 

In this work, we therefore extend the theory presented in~\cite{Delon2020-wk,Chen2015-gu} to generic mixture models. This allows us to consider both mixtures based on different distributions such as any elliptical distribution, but also any distribution transported with an affine map. These are the main cases where the Wasserstein barycenters between two different individual measures composing the mixtures (later called atoms) can be explicitely computed.
We then provide various numerical results involving different types of distributions, as well as group-invariant measures for permutation with or without underlying antisymmetry, as well as rotation-invariant measures.
Our main contributions in this article are listed below:
\begin{enumerate}
    \item We prove in a general setting that for a set of atoms equipped with a metric that defines a geodesic space, the set of mixtures based on this set of atoms is also geodesic space for the defined modified Wasserstein metric.
    \item We provide numerical results for the modified Wasserstein barycenters in the case of location-scatter atoms and group-invariant measures.
    \item We prove some sparsity and symmetry properties of optimal transport plans between measures that are invariant under some specific symmetries, namely the permutation of the variables, as well as a permutation-invariance arising from antisymmetry. 
\end{enumerate}
We leave the practical application of this theory to electronic structure calculations to a further work~\cite{PolackDusson}.

The rest of the article is organised as follows.  In Section~\ref{sec:3}, we provide a few preliminaries on Wasserstein distances and barycenters that will be used in the subsequent sections. 
In Section~\ref{sec:gen}, we give general conditions on a set of atoms for a modified Wasserstein distance defined on the associated set of mixtures to be computable by means of a sparse discrete optimal transport problem in the spirit of~\cite{Delon2020-wk}, and for the set of mixtures to be a geodesic space.
We then focus our attention to two particular cases of interest: the case of location-scatter atoms in Section~\ref{sec:sec4}, and the case of symmetry group invariant measures in Section~\ref{sec:5}. Along the way, in Section~\ref{sec:properties}, we gather some properties satisfied by the exact optimal transport plan relative to the Wasserstein metric, as well as Wasserstein barycenters.

\section{Preliminaries on Wasserstein distances and barycenters}
\label{sec:3}

To start with, we introduce in this section a few objects related to Wasserstein metric and barycenters, that will be useful in the subsequent sections of this article, see for instance~\cite{Villani2009-sy,Santambrogio2015-xa,Peyre2019-yf} for references. Let $\Omega \subset \R^d$ for $d\in\N$ be a Borelian set. We denote by $\P(\Omega)$ the set of probability measures on $\Omega$.

\subsection{Wasserstein space and distance}

Let $c:\D\times\D\rightarrow\R^+$ be a metric on $\D$. For $p>1$, we denote by $\P^c_p(\D)$ the set of probability measures on $\D$ with finite $p^{\rm th}$ moments, i.e. 
\[
   \P^c_p(\D) := \left\{ \mu \in \P(\D), \quad
      \exists x_0\in \D, \;
      \int_\D c(x,x_0)^p d\mu(x) < +\infty
      \right\}.
\]

The $p^{\rm th}$-Wasserstein distance between $\mu,\nu\in\P^c_p(\D)$ associated to $c$ is defined as
\begin{equation}
   \label{eq:Wpc}
   W^{c}_p (\mu_0, \mu_1):=\left( \inf_{\gamma \in \Pi (\mu_0, \mu_1)} \int_{\D \times \D} c(x, y)^p \, d\gamma (x, y) \right)^{1/p},
\end{equation}
where $\Pi (\mu_0, \mu_1)$ denotes the set of measures on $\D\times\D$ with marginals $\mu_0$ and $\mu_1$, also called the set of transport plans between $\mu_0$ and $\mu_1$.

An important particular example is the case where the metric $c$ is the euclidean distance on $\D$, i.e. when
\begin{equation}
    \label{eq:euclidean_distance}
    \forall x,y\in \D, \; c(x,y):=\|x-y\|. 
\end{equation}
In this case, we drop the superscript $^c$ in the notation for the ease of readability. 
Then a Wasserstein distance of particular interest is the $2$-Wasserstein distance with euclidean distance on $\D$ defined by
\begin{equation}\label{eq:OT2}
W_2(\mu_0,\mu_1) :=\mathop{\inf}_{\gamma \in \Pi(\mu_0,\mu_1)} \left( \int_{\D\times \D} \|x-y\|^2 \,d\gamma(x,y) \right)^{1/2}.
\end{equation}
The space $\P_2(\D)$ endowed with the distance $W_2$ is a metric space, usually called $L^2$-Wasserstein space (see~\cite{Villani2009-sy} for more details).
From~\cite[Theorem 1.17]{Santambrogio2015-xa}, there exists a unique optimal transport plan solution of the minimization problem~\eqref{eq:OT2} denoted by $\gamma$ in the sequel, provided that $\mu_0$ is absolutely continuous.
Also, the optimal transport plan $\gamma$ has the following form
\[
   \gamma = ({\rm Id}, T)\# \mu_0,
\]
where $T: \D \to \D$ is an application called the optimal transport map and satisfying $T\#\mu_0 = \mu_1$.  
Here, we denote by $T\#\mu$ the push-forward measure of a measure $\mu$ on~$\Omega$ by a map $T:\Omega \to \Omega$, that is the measure $\nu$ on $\D$ such that $\forall A \subset \D$, $T\#\mu (A) = \mu(T^{-1} (A))$. Similar results are available for more generic cost functions~$c$, in particular when for all $x,y\in\D$, $c(x,y) = h(x-y)$ with $h$ stricly convex and $\D$ compact~\cite{Santambrogio2015-xa}.

The path $(\mu_t)_{t\in [0,1]}$ given by 
\[
 \forall t\in[0,1], \quad \mu_t = P_t \# \gamma = ((1-t){\rm Id} + tT)\#\mu_0, \quad 
   \forall (x,y)\in \D^2, \; P_t(x,y):=(1-t)x + t y,
\]
then defines a constant speed geodesic in $\P_2(\D)$ between $\mu_0$ and $\mu_1$. The path $(\mu_t)_{t\in[0,1]}$ is called the McCahn interpolation between $\mu_0$ and $\mu_1$~\cite{McCann1997-da}. 
It can be shown that for all $t\in[0,1]$, $\mu_t$ can be equivalently expressed as the unique solution of the following minimization problem
\begin{equation}
\label{eq:mut}
     \mu_t := \mathop{\rm argmin}_{\rho \in \Pr} \; (1-t) W_2^2(\rho, \mu_0) + t W^2_2(\rho, \mu_1). 
\end{equation}

\subsection{Wasserstein barycenters}

We next recall the notion of barycenters in the Wasserstein space which was introduced in~\cite{Agueh2011-uz} and can be seen as an extension of the McCahn interpolation to a family of more than two measures.
Let $n\in \N^*$ and let
$$
\mathcal L_n:=\left\{ \bm{t} :=(t_1,\cdots,t_n) \in [0,1]^n, \quad \sum_{i=1}^n t_i = 1\right\}
$$
be the probability simplex of dimension $n-1$. 
For any family of probability measures $\bm{\mu} = (\mu_i)_{1\leq i\leq n} \in (\P^c_p(\D))^n$ and barycentric weights 
$\bm{t} = (t_i)_{1\leq i\leq n} \in \mathcal L_n$,  if one of the measures $\mu_i$ has a density, there exists a unique minimizer to the problem
\begin{equation}
\label{eq:bary}
\inf_{\rho\in \P^c_p(\D)} \sum_{i=1}^n t_i W_p^c(\rho,\mu_i)^p,
\end{equation}
which is the barycenter of the family of measures $\bm{\mu}$ with barycentric weights $\bm{t}$. 
The solutions of the barycenter problem are moreover related to the
solutions of the following multi-marginal optimal transport problem~\cite{Gangbo1998-wi}
\begin{equation}\label{eq:multimarginal}
    W_p^c(\mu_1, \ldots, \mu_n)^p:= \mathop{{\inf}}_{\gamma \in \Pi(\mu_1, \ldots, \mu_n)}\int_{\D^n} \frac{1}{2}\sum_{i,j =1}^n t_i t_j c(x_i, x_j)^p \,d\gamma(x_1, \ldots, x_n),
\end{equation}
where $\Pi(\mu_1, \ldots, \mu_n)$ denotes the set of probability measures on $\D^n$ having $\mu_1, \ldots, \mu_n$ as marginals. 
In particular, at least in the case of the euclidean distance~\eqref{eq:euclidean_distance} if $\Omega$ is a convex set and if (\ref{eq:multimarginal}) has a unique solution $\gamma^*$, there exists a unique solution $\rho^*$ to~\eqref{eq:bary} given by $\rho^* = B\#\gamma^*$, with $B: \D^n \to \D$ defined by $B(x_1, \ldots, x_n):= \sum_{i=1}^n t_i x_i$ for all $(x_1,\ldots,x_n)\in \Omega^n$ and the infima of~\eqref{eq:bary} and~\eqref{eq:multimarginal} are equal. 

\subsection{Case of gaussian distributions}

To conclude this section, we recall some well-known results about gaussians distributions for which many quantities introduced above can be made explicit. Here, $\D = \R^d$ and we denote by $\mathcal S_d$ the set of symmetric positive definite matrices of $\R^{d\times d}$.
 For any $m \in \mathbb{R}^d$ and $\Sigma \in \mathcal S_d$, the (normalized) Gaussian distribution $g_{m, \Sigma}\in \mathcal P_2(\R^d)$ is defined by 
\[ 
 g_{m,\Sigma}(\,dx) = G_{m, \Sigma}(x)\,dx,
\]
 where
\[
\forall x\in \R^d, \quad G_{m, \Sigma}(x):= \frac{1}{(2\pi)^{d/2} \sqrt{{\rm det} \Sigma}} \exp\left(\- (x-m)^T \Sigma^{-1} (x-m)\right). 
\]
For any $m_0,m_1\in \R^d$ and $\Sigma_0,\Sigma_1\in \mathcal S_d$, denoting by $g_0 := g_{m_0,\Sigma_0}$ and $g_1 := g_{m_1,\Sigma_1}$, there holds
\[
W_2^2(g_0, g_1) = \|m_0 - m_1\|^2 +  {\rm Tr}\left(
\Sigma_0 + \Sigma_1 - 2 \left(\Sigma_0^{1/2}\Sigma_1 \Sigma_0^{1/2}\right)^{1/2} 
\right).
\]
Finally, the Wasserstein barycenter between two and more generally $n$ Gaussian measures is also a Gaussian measure. Precisely, for~$\bm{t} = (t_i)_{1\leq i\leq n} \in \mathcal L_n$, the unique minimizer of~\eqref{eq:bary} is given by \begin{equation}\label{eq:Gaussbarmulti}
g_{\bm{t}}(\,dx) = G_{m_{\bm{t}}, \Sigma_{\bm{t}}}(x)\,dx, 
\end{equation}
where $m_{\bm{t}}$ and $\Sigma_{\bm{t}}$ satisfy
\begin{equation}\label{eq:Meanbarmulti}
m_{\bm{t}} =  \sum_{i=1}^n t_i m_i, \quad 
\Sigma_{\bm{t}} =  \sum_{i=1}^n t_i \left( 
\Sigma_{\bm{t}}^{1/2} \Sigma_i \Sigma_{\bm{t}}^{1/2}
\right)^{1/2}.
\end{equation}

\section{Wasserstein-type distance and barycenters for mixtures}\label{sec:gen}

In this section, we follow the works~\cite{Chen2018-vn,Delon2020-wk} which proposed a modified Wasserstein distance for gaussian mixtures, and we generalize the theory to a larger setting. 
Namely, we exhibit sufficient conditions on the densities constituting the mixtures and the definition of the modified Wasserstein distance on mixtures to define a metric and geodesic space. We then prove a similar result as in~\cite{Delon2020-wk}, namely the equivalence between a discrete optimisation problem and a continuous one in the case where the considered mixtures are identifiable.

\subsection{Distance and barycenters between generic mixtures}

We start by defining a set of particular densities, which we call atoms, and mixtures thereof.

\begin{definition}[$\A$-mixture]
Let $\A$ be a subset of $\P(\D)$, called dictionnary of atoms hereafter. We denote by $\mathcal M(\A)$ the set of finite mixtures of atoms $\A$, i.e. 
the set of probability measures $\mu$ of $\P(\D)$ such that there exists $K\in \N^*,$ $\bm{a}:=(a^1, \cdots, a^K)\in \A^K$ and $\bm{\lambda}:=(\lambda^1, \cdots, \lambda^K)\in \mathcal L_K$ such that
\[
\mu = \sum_{k=1}^K \lambda^k a^k. 
\]
\end{definition}
Typical examples of dictionaries of atoms are
\begin{itemize}
    \item [(i)] the set of non-degenerate gaussian measures i.e.
$$
\A^d_g:=\left\{ g_{m,\Sigma}, \; m\in \R^d, \; \Sigma \in \mathcal S_d\right\};
$$
    \item [(ii)] the set of possibly degenerate gaussian measures i.e.
$$
\overline{\A}^d_g:=\left\{ g_{m,\Sigma}, \; m\in \R^d, \; \Sigma \in \overline{\mathcal S}_d\right\};
$$
\end{itemize}
where $\overline{\mathcal S}_d$ is the set of symmetric positive (but not necessarily definite) matrices of $\R^{d\times d}$.
The set $\mathcal M(\A^d_g)$ (respectively $\mathcal M\left(\overline{\A}^d_g\right)$) is then the set of finite gaussian mixtures (respectively possibly degenerate gaussian mixtures) of dimension $d$. Note that both $\A^d_g$ and $\overline{\A}^d_g$ are geodesic spaces when endowed with the metric $W_2$. 

In order to define a distance on the set of mixtures as well as establish many properties on the space of mixtures which we study in the sequel, we introduce a map $\delta: \A \times \A \to \mathbb{R}_+$, which will need to satisfy the following assumption.

\begin{assumption} \label{as:metric_atom}
    The application $\delta:\A\times \A \rightarrow \R^+$ defines a metric on $\A$ and $(\A,\delta)$ is a geodesic space.
\end{assumption}

\begin{definition}[Mixture distance]
Let $\A \subset \P(\D)$ be a dictionary of atoms, and let $\delta:  \A \times \A \to \R_+$ be a metric on $\A$. 
For all $p>1$, we define the application  $\Md: \mathcal M(\A) \times \mathcal M(\A) \to \R_+$ as follows:
for all $\mu_0 = \sum_{j=1}^{J} \lambda_0^j a_0^j \in \M(\A)$ and $\mu_1 = \sum_{k=1}^{K} \lambda_1^k a_1^k \in \M(\A)$, 
\begin{equation}\label{eq:discreteatoms}
\Md(\mu_0, \mu_1)^p:= \mathop{\min}_{w := (w_{jk})_{1\leq j \leq J, 1\leq k \leq K} \in \Pi(\Lambda_0, \Lambda_1)} \sum_{j=1}^{J} \sum_{k=1}^{K} w_{jk}\delta^p(a_0^j, a_1^k),
\end{equation}
\begin{align*}
    \text{with} \quad \Pi(\Lambda_0, \Lambda_1):= \Bigg\{ & w := (w_{jk})_{1\leq j \leq J, 1\leq k \leq K}\in \R_+^{J \times K}, \\ & \quad \forall 1\leq j \leq J, \; \sum_{k=1}^{K} w_{jk}= \lambda_0^j, \quad \forall 1\leq k \leq K, \; \sum_{j=1}^{J}w_{jk} = \lambda_1^k \Bigg\}. 
\end{align*}
\end{definition}

\begin{remark}
   It is easily seen that there exists at least one minimizer $w^*:=(w^*_{jk})_{1\leq j \leq n_0, 1\leq k \leq n_1}$ to problem~(\ref{eq:discreteatoms}). However, let us point out here that uniqueness is not guaranteed. 
\end{remark}

\begin{proposition}
    For all $p>1$, if $\delta:\A\times \A \rightarrow \R^+$ defines a metric, the application $\Md$ defines a metric on the set of $\A$-mixtures $\mathcal M(\A)$. 
\end{proposition}

The proof of this result follows from~\cite[Theorem 1]{Chen2018-vn} that was derived for the particular case of gaussian mixtures. We present it here for sake of completeness as well as to precisely specify the assumptions that guarantee that the result is indeed valid in this general framework.

\begin{proof}
    First,  $\Md$ is clearly symmetric, $\Md(\mu_0,\mu_1) \ge 0$ and $\Md(\mu_0,\mu_1) = 0$ if and only if $\mu_0 = \mu_1$. 
    Next we prove the triangle inequality.
    Let $\mu_0 = \sum_{j=1}^{J} \lambda_0^{j} a_0^{j}\in \M(\A)$, $\mu_1 = \sum_{k=1}^{K} \lambda_1^{k} a_1^{k}\in \M(\A)$, and 
    $\mu_2 = \sum_{l=1}^{L} \lambda_2^l a_2^l\in \M(\A)$. We can assume without loss of generality that $\lambda_0^j >0$ for all $1\leq j \leq J$, $\lambda_1^k >0$ for all $1\leq k \leq K$ and $\lambda_2^l >0$ for all $1\leq l \leq L$. We want to show that
    \[
        \Md(\mu_0,\mu_2) \le \Md(\mu_0,\mu_1) + \Md(\mu_1,\mu_2).
    \]
    Let $w^{01} \in \R^{J\times K}$ be a solution to~\eqref{eq:discreteatoms} for $(\mu_0,\mu_1)$
    and $w^{12} \in \R^{K\times L}$ be a solution to~\eqref{eq:discreteatoms} for $(\mu_1,\mu_2)$. Let us define for $1\le j\le J, 1\le l\le L$,
    \[
        w^{02}_{jl} = \sum_{k=1}^{K} \frac{w^{01}_{jk} w^{12}_{kl}}{\lambda_1^k}.
    \]
    A simple calculation leads to show that $w^{02} \in \Pi(\Lambda_0,\Lambda_2)$. 
    Then 
\begin{align*}
    \Md(\mu_0,\mu_2)^p 
    & \le 
    \sum_{j=1}^{J} \sum_{l=1}^{L} w^{02}_{jl}\delta^p(a_0^j, a_2^l) = 
     \sum_{j=1}^{J} \sum_{l=1}^{L}
     \sum_{k=1}^{K} \frac{w^{01}_{jk} w^{12}_{kl}}{\lambda_1^k} \delta^p(a_0^j, a_2^l) \\
     & \le  
     \sum_{j=1}^{J}
    \sum_{k =1}^{K}
    \sum_{l=1}^{L}
 \frac{w^{01}_{jk} w^{12}_{kl}}{\lambda_1^k} \left(\delta(a_0^j, a_1^k ) + \delta(a_1^k,a_2^l) \right)^p,
\end{align*}
using the triangular inequality for $\delta$.
Using Minkowski inequality, we obtain
\begin{align*}
     \Md(\mu_0,\mu_2) 
    & \le \left(\sum_{j,k,l=1}^{J,L,K} \frac{w^{01}_{jk} w^{12}_{kl}}{\lambda_1^k} \delta(a_0^j, a_1^k )^p\right)^{1/p}
     + \left(\sum_{j,l,k=1}^{J,L,K}  \frac{w^{01}_{jk} w^{12}_{kl}}{\lambda_1^k}  \delta(a_1^k,a_2^l)^p \right)^{1/p} \\
     & = \left( \sum_{j,k=1}^{J,K} 
      w^{01}_{jk} \delta(a_0^j, a_1^k)^p \right)^{1/p}
     + \left( \sum_{l,k=1}^{L,K}
      w^{12}_{kl} \delta(a_1^k,a_2^l)^p \right)^{1/p} \\
     & = \Md(\mu_0,\mu_1)  + \Md(\mu_1,\mu_2).
\end{align*}
This concludes the proof.
\end{proof}

\begin{proposition}[Geodesic space]
   \label{eq:geodesic_space_mixture}
Under Assumption~\ref{as:metric_atom}, the space $\M(\A)$ equipped with the metric $\Md$ is a geodesic space.
\end{proposition}

\begin{proof}
    To show that $\M(\A)$ equipped with the metric $\Md$ is a geodesic space, we consider paths $\rho:=(\rho_t)_{t\in[0,1]}$ with $\rho_t \in \M(\A)$ for all $t\in [0,1]$ and define the length of the path relative to the $\Md$ metric as
    \begin{equation}
    \label{eq:path_length}
         {\rm Len}_{\Md}(\rho) = \sup_{\substack{N \in \N^*, \; (t_i)_{0\leq i \leq N}\in [0,1]^{N+1} \\ 0= t_0 < t_1 < \ldots < t_N = 1 }} \sum_{i=1}^N \Md(\rho_{t_{i-1}},\rho_{t_i}).
    \end{equation}
We want to show that given any two points $\mu_0,\mu_1 \in \M(\A),$ there exists a path between them the length of which equals the distance $\Md(\mu_0,\mu_1)$ between the points $\mu_0$ and $\mu_1$.
    
    First, for $\mu_0,\mu_1 \in \M(\A),$ for any path $\rho:=(\rho_t)_{t\in[0,1]}$, we always have, using the triangle inequality 
    \[
        \Md(\mu_0,\mu_1) \le \sum_{i=1}^N \Md(\rho_{t_{i-1},\rho_{t_i}}).
    \]
    Therefore, ${\rm Len}_{\Md}(\rho) \ge \Md(\mu_0,\mu_1)$. To show the equality, we need to exhibit one path which connects $\mu_0$ to $\mu_1$ whose length is $\Md(\mu_0,\mu_1)$. The proof here is adapted from~\cite[section III, B and C]{Chen2018-vn}.
    
    Let us assume that $\mu_0 = \sum_{j=1}^{J} \lambda_0^{j} a_0^{j}\in \M(\A)$ and $\mu_1 = \sum_{k=1}^{K} \lambda_1^{k} a_1^{k}\in \M(\A)$.
    We define for any $a_0^j,a_1^k\in \A$ a constant speed geodesic  $(a_t^{jk})_{t\in[0,1]}$ with  $a_t^{jk}\in \A$ for all $t\in [0,1]$ such that for all $0\le s,t \le 1$, $\delta(a_t^{jk}, a_s^{jk}) = |t-s| \delta(a_0^j, a_1^k)$.
     
     Let $w^*$ be a solution to~\eqref{eq:discreteatoms} for $\mu_0$ and $\mu_1$, we define
    $
        \mu_t = \sum_{j=1}^{J} \sum_{k=1}^{K} w^*_{jk} a_t^{jk}$. Then, for any $0\le s \le t \le 1$, choosing $w_{jk,j'k'} = w^*_{jk} \delta_{jk,j'k'} \in \Pi(w^*,w^*)$, we obtain
    \begin{align*}
        \Md^p(\mu_s,\mu_t) \; &
        \le  \sum_{jk}  \sum_{j'k'} w_{jk, j'k'} \delta^p(a_t^{jk}, a_s^{j'k'})
        \\& 
        \le (t-s)^p \sum_{jk} w^*_{jk} \delta^p(a_0^j, a_1^k) \\
        & \le (t-s)^p \Md^p(\mu_0,\mu_1),\\
    \end{align*}
    from which we deduce that 
    $
    {\rm Len}_{\Md}((\mu_t)_{t\in[0,1]}) = \Md(\mu_0,\mu_1)$. We conclude that $\M(\A)$ equipped with $\Md$ is a geodesic space.
\end{proof}

In the case where $(\M(\A),\Md)$ is a geodesic space, we can now express barycenters in terms of geodesics of atoms.

\begin{corollary}
Let us assume that Assumption~\ref{as:metric_atom} holds. Let $\mu_0 = \sum_{j=1}^{J} \lambda_0^j a_0^j$ and $\mu_1 = \sum_{k=1}^{K} \lambda_1^k a_1^k$ be two elements of $\M(\A)$. 
   The barycenters between $\mu_0$ and $\mu_1$ belong to $\M(\A)$  and can be written as 
   \[
      \forall t\in [0,1],\quad \mu_t = \sum_{j=1}^{J}\sum_{k=1}^{n_1} w_{jk}^* a_t^{j,k},
   \]
where $w^* = (w_{jk}^*)_{1\leq k \leq K, 1\leq l \leq n_1}$ is a solution to (\ref{eq:discreteatoms}) and $(a_t^{j,k})_{t\in[0,1]}$ are constant speed geodesic between respectively $a_0^j$ and $a_1^k$.
\end{corollary}

\subsection{Particular case of Wasserstein metric}

We focus here on the specific case where $\A \subset \mathcal P_p^c(\D)$ for some $p>1$ and metric $c:\Omega\times\Omega \rightarrow \R^+$, and where the atom metric $\delta$ is defined by the Wasserstein metric $W_p^c$, i.e.
\begin{equation}
    \label{eq:Wass_distance}
     \forall a_0, a_1\in \A, \quad \delta(a_0,a_1) = \left( 
      \inf_{\gamma \in \Pi(a_0,a_1)} 
    \int_{\Omega\times\Omega}(c(x,y))^p d\gamma(x,y)
    \right)^{1/p}.
\end{equation}
Let us point out that Assumption~\ref{as:metric_atom} is 
indeed satisfied in this setting.
We first investigate two-marginal problems, and then show that the theory can easily be extended to multi-marginal problems.

\subsubsection{Two-marginal problem}

Let us show that the discrete problem~\eqref{eq:discreteatoms} is in fact equivalent to a continuous problem, similar to the one presented in~\cite{Delon2020-wk}. To obtain this equivalence, we need to make two assumptions, which is the existence of an optimal transport plan between any atoms, as well as the identifiability of the mixtures of atoms.

\begin{assumption}[Existence of optimal transport plan]\label{as:existencetransport}
   The set of atoms $\A$ and the cost function $c$ are such that for any $a_0,a_1\in\A$, there exists at least one solution to the optimal transport problem
   \begin{equation}\label{eq:OT}
      \mathop{\rm inf}_{\gamma \in \Pi(a_0,a_1)} \int_{\Omega\times\Omega} c(x,y)^p d\gamma(x,y).
   \end{equation}
The set of minimizers of (\ref{eq:OT}) is then denoted by $\Gamma^*_{a_0,a_1}$.
\end{assumption}
This assumption is for instance satisfied for the cost function $c(x,y) = \|x-y\|$, $p=2$ and atoms that are absolutely continuous measures.

\begin{assumption}[Identifiability] \label{as:identifiability}
    The set of mixtures $\M(\A)$ is identifiable, that is given two mixtures $ \mu_0 = \sum_{j=1}^{J} \lambda_0^j a_0^j$, $ \mu_1 = \sum_{k=1}^{K} \lambda_1^k a_1^k$, where the atomic functions  $(a_0^j)_{j=1,\ldots,J}$ are pairwise distinct, and similarly for $(a_1^k)_{k=1,\ldots,K}$, then $\mu_0 = \mu_1$ if and only if $J = K$, and the indices in the sums can be reordered such that all $\lambda_0^k = \lambda_1^k$ and $a_0^k = a_1^k$ for all $k=1,\ldots,K$.
\end{assumption}
A classical example of identifiable mixtures is the set of gaussian mixtures~$\M(\A^d_g)$, see e.g.~\cite[Appendix]{Delon2020-wk}. 
We then define the set of optimal transport plans of the atoms as well as mixtures thereof.

\begin{definition}[Admissible set of atom transport plans]
\label{def:admsettransport_atom}
A set of transport plans $\Gamma(\A)\subset \mathcal P(\Omega \times \Omega)$ is said admissible for the set of atoms $\A$ if and only if
$$
\Gamma(\A) = \bigcup_{a_0,a_1\in \A} \Gamma_{a_0,a_1},
$$
where for all $a_0,a_1\in \A$, $\Gamma_{a_0,a_1}$ is a convex set such that
$$
\Gamma^*_{a_0,a_1} \subset \Gamma_{a_0, a_1} \subset \Pi(a_0, a_1).
$$
\end{definition}

\begin{definition}[Mixtures of admissible atom transport plans]
\label{def:admsettransport_mixtures}
Let $\Gamma(\A)$ be an admissible set of atom transport plans.
We define $\GammaM(\A):= \M\left( \Gamma(\A)\right)$ as the set of finite mixtures of optimal transport plans of atoms $\Gamma(\A)$, i.e. the set of probability measures $\gamma$ of $\P(\Omega\times\Omega)$ such that there exists $K\in \N^*$, $\bm{\gamma}:=(\gamma^1, \cdots, \gamma^K)\in \Gamma(\A)^K$ and $\bm{\lambda}:=(\lambda^1, \cdots, \lambda^K)\in \mathcal L_K$ such that
\[
    \gamma = \sum_{k=1}^K \lambda^k \gamma^k. 
\]
\end{definition}

\begin{proposition} \label{prop:equivalence_discrete_cont}
  Let $\mu_0 = \sum_{j=1}^{J} \lambda_0^{j} a_0^{j}\in \M(\A)$ and $\mu_1 = \sum_{k=1}^{K} \lambda_1^{k} a_1^{k}\in \M(\A)$. Let us define
    \begin{equation}\label{eq:MOT2}
        \Mdt(\mu_0, \mu_1):= 
        \left(
        \inf_{\gamma \in \Pi(\mu_0, \mu_1) \cap \GammaM(\A) } \int_{\Omega \times \Omega} c(x,y)^p \,d\gamma(x,y)
        \right)^{1/p}. 
    \end{equation}
    Then, under Assumptions~\ref{as:existencetransport} and~\ref{as:identifiability}, problem~\eqref{eq:MOT2} is equivalent to~\eqref{eq:discreteatoms}, i.e.
    $ {\Mdt}(\mu_0, \mu_1) = {\Md}(\mu_0, \mu_1)$.
    Moreover, the sets of minimizers of (\ref{eq:MOT2}) is equal to the set of measures $\gamma$ which can be written as
    \begin{equation}
    \label{eq:gammaopt}
           \gamma = \sum_{j=1}^{J}  \sum_{k=1}^{K} w_{jk}^* \gamma_{jk},
    \end{equation}
        with $w^*:=(w^*_{jk})_{1\leq j \leq J, 1\leq k \leq K}$ a minimizer of~\eqref{eq:discreteatoms}, and $\gamma_{jk}  \in \Gamma^*_{a_0^j, a_1^k}$ is an optimal transport plan between $a_0^j$ and $a_1^k$.
\end{proposition}

\begin{proof}
    This proof extends the arguments presented in~\cite[Proposition 4]{Delon2020-wk}.
    Let $\mu_0 = \sum_{j=1}^{J} \lambda_0^j a_0^j$, $\mu_1 = \sum_{k=1}^{K} \lambda_1^k a_1^k$ be two mixtures in $\M(\A)$, with $a_0^j$ (respectively $a_1^k$) all distinct. 
     First, let $w^*\in \R^{J\times K}$ be a solution to~\eqref{eq:discreteatoms}, and let for $j = 1,\ldots, J$ and $k = 1,\ldots, K$, $\gamma_{jk}:= \gamma^*_{a_0^j, a_1^k}$ 
    so that $\gamma_{jk}\in \Gamma_{a_0^j, a_1^k} \supset \Gamma_{a_0^j, a_1^k}^*$.  
    Then let us define
    $
    \gamma^* = \sum_{j=1}^{J}  \sum_{k=1}^{K} w^*_{jk} \gamma_{jk}$. There holds $\gamma^* \in \GammaM(\A) \cap \Pi(\mu_0,\mu_1)$. 
    Moreover, 
    \begin{align*}
    {\Mdt}(\mu_0, \mu_1)^p &\le 
         \sum_{j=1}^{J} \sum_{k=1}^{K} w^*_{jk} 
    \int_{\Omega\times\Omega} c(x,y)^p d\gamma_{jk}(x,y)  =  \sum_{j=1}^{J} \sum_{k=1}^{K} w^*_{jk} \delta^p(a_0^j,a_1^k) \\
    & = \Md(\mu_0, \mu_1)^p
        .
    \end{align*}
    Second, let $\gamma \in \GammaM(\A)$. Then, there exists $I\in \N^*$ such that $\gamma = \sum_{i=1}^I \lambda^i \gamma^i$ with $\gamma^i\in \Gamma(\A)$. Using that $P_0 \# \gamma = \mu_0$, we obtain 
    \[
        \sum_{i=1}^I \lambda^i P_0 \# \gamma^i = \sum_{j=1}^{J} \lambda_0^j a_0^j.
    \]
    Using the fact that for all $1\leq i \leq I$, $P_0 \# \gamma^i\in \A$ and using the identifiability assumption~\ref{as:identifiability}, 
    we obtain that for each $i =1,\ldots, I$, 
    there exists $j \in \{ 1, \ldots, J\} $ such that $P_0 \# \gamma^i = \mu_0^j$. 
    Similarly, for each $i =1,\ldots, I$, there exists $k \in \{ 1, \ldots, K\} $ 
    such that $P_1 \# \gamma^i = \mu_1^k$. Hence for all $1\leq i \leq I$, $\gamma^i \in \Pi(\mu_0^j,\mu_1^k)$ for some $j \in \{ 1, \ldots, J\}, k \in \{ 1, \ldots, J\} $. 
    For all $1\leq j \leq J$ and $1\leq k \leq K$, let $I_{jk}:= \left\{ 1\leq i \leq I, \; \gamma^i \in \Pi(\mu_0^j,\mu_1^k)\right\}$. The collection of sets $(I_{jk})_{1\leq j \leq J, 1\leq k \leq K}$ then forms a partition of the set $\{1,\ldots,I\}$ and for all $1\leq j \leq J$ and $1\leq k\leq K$, we denote by $\displaystyle \Lambda_{jk}:= \sum_{i\in I_{jk}} \lambda^i$. For all $1\leq j \leq J$ and $1\leq k \leq K$, we also denote by 
    $$
\overline{\gamma}_{jk}:= \left\{
    \begin{array}{ll}
    \frac{1}{\Lambda_{jk}}\sum_{i\in I_{jk}} \lambda^i \gamma^i & \mbox{ if } \Lambda_{jk} \neq 0,\\
    \gamma^*_{a_0^j, a_1^k} & \mbox{ if } \Lambda_{jk} = 0.\\
    \end{array}
    \right.
    $$
    Then, it necessarily holds that for all $1\leq j \leq J$, $1\leq k \leq K$, $\overline{\gamma}_{jk} \in \Gamma_{a_0^j, a_1^k}$ since $\Gamma_{a_0^j, a_1^k}$ is a convex set and $\gamma^i \in \Gamma_{a_0^j, a_1^k}$ for all $i\in I_{jk}$. Furthermore, we have
    $
    \gamma = \sum_{j=1}^{J}\sum_{k=1}^{K} \Lambda_{jk}\overline{\gamma}_{jk}$. Thus, 
    \begin{align*}
        \int_{\Omega\times\Omega} c(x,y)^p d\gamma(x,y) &=
        \sum_{j=1}^{J}  \sum_{k=1}^{K} w_{jk} \int_{\Omega\times\Omega} c(x,y)^p d\overline{\gamma}_{jk}(x,y) \\
        & \ge \sum_{j=1}^{J}  \sum_{k=1}^{K} w_{jk} \delta^p(a_0^j, a_1^k) \ge \Md^p(\mu_0,\mu_1),
    \end{align*}
    where we have used~\eqref{eq:Wass_distance}. This inequality being valid for any $\gamma \in \GammaM(\A)$, we conclude that ${\Mdt}(\mu_0,\mu_1) \ge \Md(\mu_0,\mu_1)$. 
    Equation~\eqref{eq:gammaopt} is then a straightforward consequence of this proof.
\end{proof}

\begin{remark}
   In the case where $\A$ is the set of atomic gaussian measures $\A^d_g$, 
the result of~\cite[Proposition 4]{Delon2020-wk} corresponds to Proposition~\ref{prop:equivalence_discrete_cont} for a set of atomic transport plans chosen as
   $$
   \Gamma(\A):= \bigcup_{a_0,a_1\in \A_g^d} \left( \overline{\A}_g^{2d} \cap \Pi(a_0,a_1)\right) = \overline{\A}_g^{2d}.
   $$
   Indeed  for all $a_0,a_1\in \A_g^d$, the set $\Gamma_{a_0,a_1}:=\left( \overline{\A}_g^{2d} \cap \Pi(a_0,a_1)\right)$ is convex and such that
   \[
   \{ \gamma^*_{a_0,a_1} \} \subset \Gamma_{a_0,a_1} \subset \Pi(a_0,a_1).
   \]
\end{remark}

\subsubsection{Multi-marginal problems}

The theory naturally extends to multi-mar\-ginal problems. For this, we need to assume the existence of multi-marginal optimal transport plans. We then define the extensions of Definitions~\ref{def:admsettransport_atom} and~\ref{def:admsettransport_mixtures} to the multi-marginal case.

\begin{assumption}[Existence of multi-marginal transport map]\label{as:multimarg}
   The set of atoms $\A$ and the metric $c$ are such that for all $Q\geq 2$, for any $\bm{t}:=(t_q)_{1\leq q \leq Q} \in \mathcal L_Q$ and for any $a_1,\cdots,a_Q\in\A$, there exists at least one solution to the multi-marginal optimal transport problem
\begin{equation}
   \label{eq:mm_atoms}
   \AMdq^{\bm{t}}
   (a_1,\ldots,a_Q) =
   \left( 
   \inf_{\gamma\in \Pi(a_1,\ldots,a_Q)}
   s(x_1,\ldots,x_Q) d\gamma(x_1,\ldots,x_Q) \right)^{1/p},
\end{equation}
with
\[
   s(x_1,\ldots,x_Q) = \frac{1}{2} \sum_{q=1}^Q
   \sum_{q'=1}^Q t_q t_{q'} c(x_q,x_{q'})^p.
   \]
The set of minimizers of (\ref{eq:mm_atoms}) is then denoted by $\Gamma^{*,\bm{t}}_{a_1,\ldots,a_Q}$.
\end{assumption}

\begin{definition}
   We say that $\Gamma_Q(\A) \subset \mathcal P(\Omega^Q)$ is an admissible set of atom multi-marginal transport plans if
       \[
       \Gamma_Q(\A) = \bigcup_{a_1,\cdots,a_Q \in \A} \Gamma_{a_1,\ldots,a_Q}, \]
        where $\Gamma_{a_1,\ldots,a_Q}$ is any convex set such that 
       $$
       \bigcup_{\bm{t} \in \mathcal L_Q} \Gamma^{*,\bm{t}}_{a_1,\ldots,a_Q} \subset \Gamma_{a_1,\ldots,a_Q} \subset \Pi(a_1,a_2,\ldots,a_Q).
       $$
\end{definition}

\begin{definition}
   We define $\GammaQM(\A)$ as the set of finite mixtures of admissible atom multi-marginal optimal plans $\Gamma_Q(\A)$, i.e. the set of probability measures of $\P(\Omega^Q)$ such that there exists $K\in \N^*$, $\bm{\gamma}:=(\gamma^1, \cdots, \gamma^K)\in \Gamma_Q(\A)^K$ and $\bm{\lambda}:=(\lambda^1, \cdots, \lambda^K)\in \mathcal L_K$ such that
\[
   \gamma = \sum_{k=1}^K \lambda^k \gamma^k. 
\]
\end{definition}

Now, let $Q\geq 2$ and for all $1\leq q \leq Q$, let $K_q\in \mathbb{N}^*$, $\boldsymbol{\lambda}_q:=(\lambda_q^{k_q})_{1\leq k_q \leq K_q} \in \mathcal L_{K_q}$ and $\mu_q:= \sum_{k_q=1}^{K_q} \lambda_q^{k_q} a_q^{k_q} \in \M(\A)$. For all $\bm{t}:=(t_q)_{1\leq q \leq Q} \in \mathcal L_Q$, we define the multi-marginal transport problem as
\begin{equation}
   \label{eq:mm}
   \Mdtq^{\bm{t}}(\mu_1,\ldots,\mu_Q) =
   \left( 
   \inf_{\gamma\in \Pi(\mu_1,\ldots,\mu_Q) \cap \Gamma_{Q,\M}(\A)}
   s(x_1,\ldots,x_Q) d\gamma(x_1,\ldots,x_Q) \right)^{1/p},
\end{equation}
with
\[
   s(x_1,\ldots,x_Q) = \frac{1}{2} \sum_{q=1}^Q
   \sum_{q'=1}^Q t_q t_{q'} c(x_i,x_j)^p.
   \]

Using a similar approach as in the previous section, we prove the following result.

\begin{proposition}
Under Assumption~\ref{as:multimarg}, problem~\eqref{eq:mm} is equivalent to the following discrete problem
   \begin{equation}
      \label{eq:mm2}
      \left( \Mdq^{\bm{t}}
      (\mu_1,\ldots,\mu_Q) \right)^p =
      \hspace{-.2cm}
   \inf_{w\in \Pi(\boldsymbol{\lambda}_1,\ldots,\boldsymbol{\lambda}_Q)}
   \sum_{k_1,k_2,\ldots,k_{Q}=1}^{K_1,K_2,\ldots,K_Q}  
   \hspace{-.6cm}
   w_{k_1,k_2,\ldots,k_Q}
   \hspace{-.1cm}
   \left(\AMdq^{\bm{t}}
   (a_1^{k_1},a_2^{k_2}, \ldots, a_Q^{k_Q})\right)^p \hspace{-.1cm},
   \end{equation}
      \begin{align*}
      \text{with} \;\;
          \Pi(\boldsymbol{\lambda}_1, \ldots, & \boldsymbol{\lambda}_Q)  := \Bigg\{  w \in \R_+^{K_1 \times \ldots \times K_Q}, \\
     & \forall 1\le q \le Q, \; \forall 1\le k_q \le K_q,
       \sum_{
     k_1,\ldots,k_{q-1},k_{q+1},\ldots, k_Q = 1}^{  K_1,\ldots,K_{q-1},K_{q+1},\ldots, K_Q}
        w_{k_1,\ldots,k_Q}= \lambda_q^{k_q} \Bigg\}. 
      \end{align*}
      Moreover, denoting by $\mathcal K:= \{1,\ldots, K_1\}\times \{1,\ldots, K_2\}\times \ldots \times \{1,\ldots, K_Q\}$, any solution to~\eqref{eq:mm} can be written as 
      \begin{equation}
         \label{eq:gammastar}
         \gamma^* = \sum_{{\bf k}\in \mathcal K } w^*_{\bf k}  \gamma^*_{\bf k},
      \end{equation}
      where $(w^*_{\bf k})_{{\bf k}\in \mathcal K}$ is a solution to~\eqref{eq:mm2} and $\gamma^*_{\bf k} \in \Gamma^{*,\bm{t}}_{a_1^{k_1}, \ldots, a_Q^{k_Q}}$ for all ${\bf k} = (k_1,\ldots,k_Q)\in \mathcal K$. 
      As a consequence, any barycenter of $(\mu_1,\cdots,\mu_Q)$ with barycentric weights $\bm{t}$ can be written as 
      \begin{equation}
         \label{eq:barycenter}
         {\rm bar}_{\bm{t}}(\mu_1,\ldots, \mu_Q) =
         \sum_{{\bf k}\in \mathcal K} w^*_{\bf k} \; {\rm bar}_{\bm{t}}(a_1^{k_1},\ldots, a_Q^{k_Q}),
      \end{equation}
      where ${\rm bar}_{\bm{t}}(a_1^{k_1},\ldots, a_Q^{k_Q})$ is the barycenter of $(a_1^{k_1},\ldots, a_Q^{k_Q})$ with barycentric weights $\bm{t}$.  Finally, any solution $(w^*_{\bf k})_{{\bf k}\in \mathcal K}$ to~\eqref{eq:mm2} contains at most $K_1+K_2+\ldots + K_Q-Q+1$ nonzeros components.
\end{proposition}

\begin{proof}
The proof for the equivalence between the discrete and continuous optimization problem is similar to the proof of Proposition~\ref{prop:equivalence_discrete_cont}. The structure of the barycenter is a direct consequence of equation~\eqref{eq:gammastar}. The maximum number of nonzeros components comes from the form of the discrete problem, which is a linear program with $K_1+K_2+\ldots+K_Q-Q+1$ affine constraints, therefore there exists a solution with at most $n_1+n_2+\ldots+n_Q-Q+1$ nonzero components, see~\cite[Theorem 2]{Anderes2016-kv} or~\cite[Appendix A]{Friesecke2022-od}. 
\end{proof}

In the end, as soon as Assumption~\ref{as:metric_atom} is satisfied for the atoms, we can define a geodesic space on the mixtures of these atoms. Moreover, if the metric on the atoms corresponds to a Wasserstein distance for which there is uniqueness of the transport plans and that the mixtures of atoms are identifiable, then there is an equivalence between the discrete problem~\eqref{eq:discreteatoms} and the continuous one~\eqref{eq:MOT2}, which in particular shows that the result is independent of the representation of the mixtures. 

\section{Mixtures of location-scatter atoms}\label{sec:sec4}

The sufficient conditions presented in Section~\ref{sec:3} to define a geodesic space on mixtures are pretty simple. We only need the set of atoms to be a geodesic space (Assumption~\ref{as:metric_atom}). However, all calculations on mixtures include distance calculations on the set of atoms, as well as multi-marginal calculations for the computation of barycenters. Therefore, the efficiency of the mixture calculations will highly depend on the cost of computing distances and barycenters between atoms. 
In the best case scenario, these should be explicit. 
This is for example the case of gaussian measures, which motivated the two contributions~\cite{Delon2020-wk,Chen2018-vn}. The aim of this section is to highlight how the general framework introduced in the previous section can be used in order to extend these practical considerations to more general sets of atoms, including location-scatter atoms~\cite{Alvarez-Esteban2016-us}.

\begin{remark}
   In the work by Delon and Desolneux~\cite{Delon2020-wk}, it is mentioned that other distributions than gaussians can be used, as long as they satisfy two conditions: the identifiability property (Assumption~\ref{as:identifiability}), and a marginal consistency, i.e. that transport plans are mixtures of 2d-dimensional atoms. In fact, what we have shown in the previous section is that we do not need the marginal consistency, as we use for transport plans mixtures of transport plans between atoms. Interestingly, for gaussians, the set of transport plans between atoms can be chosen as the set of (degenerate) gaussians in dimension $2\d$.
\end{remark}

In this section, we therefore start by stating results on location-scatter measures, which correspond to families of probability measures
generated from affine transformations. In some specific cases described more precisely below, transport maps and Wasserstein barycenters are explicitely computable. We then turn to a few practical examples, including elliptical distributions and affine-generated measures. 
For the sake of simplicity, we focus in this section on the 2-Wasserstein distance with quadratic cost (i.e. when $p=2$ and $c(x,y) = \|x-y\|$) and $\Omega = \mathbb{R}^d$.

\subsection{General location-scatter measures}

Location-scatter measures are generated from affine transformations of a given probability measure. We define the corresponding set of atoms as follows. 

\begin{definition}[Location-scatter atoms]\label{def:scatter}
    We define the set of location-scatter atoms generated from $a\in\P_2(\mathbb{R}^d)$ as 
   \begin{equation}\label{eq:scatter}
      \A := \left\{ 
         T\# a, \quad T : x \in \R^d \mapsto A x + b, \; A\in \mathcal S_d, \; b\in \R^d
         \right\}.
   \end{equation}
\end{definition}

We then formulate a generic result on location-scatter measures, which includes an explicit expression for the Wasserstein distance. This result is a rewriting in our context of~\cite[Theorem 2.3]{Alvarez-Esteban2016-us}, itself based on~\cite[Theorem 2.1]{Cuesta-Albertos1996-xn}.

\begin{theorem}
    \label{thm:location-scatter}
    Let $\A$ be a set of location-scatter atoms generated from $a\in \P_2(\mathbb{R}^d)$ in the sense of Definition~\ref{def:scatter}. Let us assume that the measure $a$  has mean $m_a \in \mathbb{R}^d$ and a covariance matrix $\Sigma_a \in \mathcal S_d$.
    Let 
    $a_0,a_1\in\A$ be such that there exist $T_0, T_1:\R^d \rightarrow \R^d$ with $T_0(x) = A_0 x + b_0$, $T_1(x) = A_1 x + b_1$, $A_0,A_1 \in \mathcal S_d$ and $b_0,b_1\in \R^d$
    such that 
   $
      a_0 = T_0\#a, 
   $
    $
      a_1 = T_1\#a.
   $
   Then, $a_0$ and $a_1$ have respectively means $m_i$ and covariance $\Sigma_i$ for $i=0,1$ defined by
      \begin{align*}
          m_i  = A_i m_a + b_i, \qquad
          \Sigma_i  =
          A_i \Sigma_a A_i^T.
      \end{align*}
   Moreover the Wasserstein distance squared between $a_0$ and $a_1$ satisfies
   \begin{equation}
      \label{eq:W2_location_scatter}
          W_2^2(a_0,a_1) \ge \| m_0 - m_1 \|^2 
         + {\rm Tr}\left( 
         \Sigma_0 + \Sigma_1 - 2 (\Sigma_0^{1/2} \Sigma_1 \Sigma_0^{1/2})^{1/2}
         \right),
   \end{equation}
      with equality if and only if the transport map
      \begin{equation}
         \label{eq:transportmaptheorem}
         Tx = A x + (m_0 - m_1), \quad \mbox{with } \; A = \Sigma_0^{-1/2} \left( 
            \Sigma_0^{1/2} \Sigma_1 \Sigma_0^{1/2}
            \right)^{1/2} \Sigma_0^{-1/2},
      \end{equation}
      is such that $T\# a_0 = a_1$.
 \end{theorem}

\begin{proof}
Let us first show that $a_0$ (respectively $a_1$) has mean $m_0$ and covariance $\Sigma_0$ (resp. $m_1$ and $\Sigma_1$).
    Since $a_0 = T_0 \# a,$ the density of $a_0$ satisfies
    \[
        a_0(x) = | \det A_0 |^{-1} a(A_0^{-1} (x-b_0) ).
    \]
    The mean can be computed as
    \[
    m_0 = \int_{\mathbb{R}^d} x a_0(x) dx = 
    \int_{\mathbb{R}^d} x | \det A_0 |^{-1} a(A_0^{-1} (x-b_0) ) dx.
    \]
    Using the change of variable $y= A_0^{-1} (x-b_0)$, we obtain
    \[
    m_0 =
    \int_{\mathbb{R}^d} (A_0 y + b_0) a(y) dy = A_0 m_a + b_0.
    \]
    For the covariance, we compute
    \[
    \Sigma_0 = \int_{\mathbb{R}^d}
    \hspace{-.1cm}
    (x-m_0) (x-m_0)^T a_0(x) dx = \int_{\mathbb{R}^d}
     \hspace{-.1cm}
    (x-m_0) (x-m_0)^T | \det A_0 |^{-1} a(A_0^{-1} (x-b_0) ) dx.
    \]
    Using the same change of variables, we obtain
    \[
    \Sigma_0 = \int_{\mathbb{R}^d}
    A_0(y- m_a) (y-m_a)^T A_0^T
    a(y) dy = A_0 \Sigma_a A_0^T.
    \]
The proof for computing the mean and covariance of $a_1$ is similar.
Equations~\eqref{eq:W2_location_scatter} and~\eqref{eq:transportmaptheorem} directly follow  from~\cite[Theorem 2.3]{Alvarez-Esteban2016-us}.
\end{proof}

\begin{example}
   A nice example of location-scatter atoms is the case of a generative atom $a$ that is an absolutely continuous radially symmetric with finite second-order moments probability measure on $\R^d$, i.e. with an associated density $\mu$ satisfying $\mu(x) = \xi(|x|)$ for some function $\xi\in L^1(\mathbb{R}^d)$~\cite{FrieseckeOT}.
   This encompasses the case of elliptical distributions defined by
   \[
      \forall x\in \R^d, \quad f_{m,\Sigma}(x) = \frac{1}{Z_{h,\Sigma,m}} h((x-m)^T \Sigma^{-1} (x-m)),
   \]
   for $m\in\R^d$, $\Sigma\in\mathcal S_d$,
   where the generative atom $a$ can be taken as $f_{0,I_d}$, denoting by $I_d$ the identity matrix of size $d$.
\end{example}

We can express a similar theorem for the multi-marginal optimal transport problem.

\begin{theorem}
   Let $\A$ be a set of location-scatter atoms generated from $a\in \P_2({\mathbb{R}^d})$ in the sense of Definition~\ref{def:scatter}. Let us assume that the measure $a$ has a mean $m_a\in \mathbb{R}^d$ and a covariance matrix $\Sigma_a \in \mathcal S_d$. Let $a_1,a_2,\ldots, a_Q\in\A$ be such that there exist $T_1, \ldots, T_Q:\R^d \rightarrow \R^d$ with $T_q(x) = A_q x + m_q$, $A_q \in  \mathcal S_d$ and $m_q\in \R^d$
    such that  for $q=1, \ldots, Q$, 
   $
      a_q = T_q\#a, 
   $
 and $a_q$ has mean $m_q$ and covariance matrix $\Sigma_q$.
    
The multi-marginal optimal transport problem  with parameters $\bm{t} = (t_1,\ldots,t_Q)$ has for minimal cost
   \[
      \AMdqtwo^{\bm t}
   (a_1,\ldots,a_Q) = \sum_{q=1}^Q t_q W_2^2(a_q, \bar a_{\bm t}), 
   \]
   where $\bar a_{\bm t} = T_{\bm t} \#a$ is the Wasserstein barycenter with
   $
        T_{\bm t} = S x + m, 
   $
   where $S$ is the only positive definite matrix satisfying 
   \[
        S = \sum_{q=1}^Q t_q (S^{1/2} \Sigma_q S^{1/2})^{1/2}, \quad 
        \text{and} \quad 
        m = \sum_{q=1}^Q t_q m_q.
   \]
\end{theorem}

\begin{proof}
   This theorem is a reformulation of~\cite[Corollary 4.5]{Alvarez-Esteban2016-us}.
\end{proof}

With this result, we immediately obtain that the set of location-scatter atoms is a geodesic space, using the geodesics defined by the barycenters, which  guarantees the applicability of the results of Section~\ref{sec:3} for the location-scatter atoms.

\subsection{Particular case of elliptical distributions}

Among the distributions satisfying Assumption~\ref{as:metric_atom},  we insist here on the case of elliptical distributions, which are widely used in practice.
The first natural example is the case of gaussian distributions, for which Wasserstein distance and barycenters can be explicitly computed.
This case was thoroughly presented in~\cite{Delon2020-wk,Chen2018-vn}, therefore we do not expand further on this example. 
However from Theorem~\ref{thm:location-scatter}, any location-scatter distribution can be considered provided that one knows its mean and covariance matrix.

We therefore make the explicit calculations of the means and covariance matrices of elliptical distributions which write for a given $h:\R^+\rightarrow\R^+$
\[
 g_{m,\Sigma} (x) = \frac{1}{Z_{m,\Sigma,h}} h((x-m)^T\Sigma^{-1} (x-m)),
\]
with $m\in\R^d, \Sigma\in\R^{d\times d},$  $Z_{m,\Sigma,h}$ being a normalization factor.

\begin{lemma}
    For a given $h:\R^+\rightarrow\R^+$, 
then the distribution defined on $\R^d$ via
\[
    \forall x\in\R^d, \quad g_{m,\Sigma} (x) = \frac{1}{Z_{m,\Sigma,h}} h((x-m)^T\Sigma^{-1} (x-m)), 
\]
with
\begin{equation}
    \label{eq:zm}
    Z_{m,\Sigma,h} = \frac{2 \pi^{d/2}}{\Gamma(d/2)} | \det\Sigma|^{1/2} \int_{\R^+} r^{d-1} h(r^2) dr,
\end{equation}
has mean $m$. Moreover if
\begin{equation}
    \label{eq:h-condition}
   \frac{\int_{\R^+} r^{d+1} h(r^2) dx}{\int_{\R^+} r^{d-1} h(r^2) dr}  = d, 
\end{equation}
then $g_{m,\Sigma}$ has covariance $\Sigma$.
\end{lemma}

\begin{proof}
    First, let us check that $g_{m,\Sigma}$ indeed defines a probability distribution. Using a first change of variable $y = \Sigma^{-1/2}(x-m)$, a second change of variables with spherical coordinates, and using that the surface of the $d-$sphere of radius $r$ is $\frac{2 \pi^{d/2}}{\Gamma(d/2)} r^{d-1}$ we obtain
    \begin{align*}
        \int_{\R^d} h((x-m)^T\Sigma^{-1} (x-m)) dx & = 
         \frac{2 \pi^{d/2}}{\Gamma(d/2)} | \det\Sigma|^{1/2} \int_{\R^+} r^{d-1} h(r^2) dr,
    \end{align*}
    from which we easily deduce using~\eqref{eq:zm} that $g_{m,\Sigma}$ is normalized to one.
    Second, the mean of $g_{m,\Sigma}$ satisfies
    \begin{align*}
        \int_{\R^d} x g_{m,\Sigma}(x) dx &= \int_{\R^d} (x+m) g_{0,\Sigma}(x) dx = m,
    \end{align*}
    where we have used a change of variable, and that $g_{0,\Sigma}(-x) = g_{0,\Sigma}(x)$ for all $x\in \R^d$.
    Third, the covariance matrix of $g_{m,\Sigma}$ can be computed as
     \begin{align*}
        \int_{\R^d} (x-m) (x-m)^T g_{m,\Sigma}(x) dx &=
        \int_{\R^d} x \Sigma x^T  g_{0,I}(x) dx
     =    \Sigma^{1/2} \hspace{-.1cm} \int_{\R^d} y  y^T  g_{0,I}(y) dy [\Sigma^{1/2}]^T.  
    \end{align*}
    Using a change of variable with spherical coordinates, we write 
    \[
    y = r \left( \cos(\varphi_1), \sin(\varphi_1) \cos(\varphi_2), \ldots, \sin(\varphi_1) \ldots \sin(\varphi_{d-1}) \right)
    \]
    for $\varphi_1,\ldots, \varphi_{d-2} \in (0,\pi)$ and $\varphi_{d-1} \in (0,2\pi)$, and the volume element is 
    \[
    r^{d-1} \sin^{d-2}(\varphi_1)\ldots \sin(\varphi_{d-2}) dr d\varphi_1 \ldots d\varphi_{d-1}.
    \]
    It is then easy to see that the off-diagonal terms of $\int_{\R^d} y  y^T  g_{0,I}(y) dy$ are zero, as they all contain at least one term writing $\int_0^\pi \cos(\varphi) \sin^k(\varphi) d\varphi$ for some integer $k$, which is zero. As for the diagonal terms, using the Wallis integral formula for the sine terms
    \[
        \int_{0}^\pi \sin^k(\varphi) d\varphi = \sqrt{\pi} \frac{\Gamma\left(\frac{k}{2}+\frac12 \right)}{\Gamma\left(\frac{k}{2}+1 \right)}, 
    \]
    and the beta function for the term containing a cosine function
    \[
        \int_{0}^\pi \cos^2(\varphi)\sin^k(\varphi) d\varphi = \frac{\Gamma\left(\frac{k}{2} + \frac{1}{2} \right) \Gamma\left(\frac32 \right)}{\Gamma\left(\frac{k}{2}+2 \right)} = \frac{\sqrt{\pi}}{2} \frac{\Gamma\left(\frac{k}{2} + \frac{1}{2} \right)}{\Gamma\left(\frac{k}{2}+2 \right)},
    \]
    we obtain that the angular part of the diagonal terms of $\int_{\R^d} y  y^T  g_{0,I}(y) dy$ are all equal to
    \begin{align*}
        \label{eq:beta_d}
        \beta_d &=         
        \int \sin^d(\varphi_1) \sin^{d-1}(\varphi_2) \ldots \sin^{d-k+2}(\varphi_{k-1}) \; \cos^2(\varphi_k) \sin^{d-k-1}(\varphi_k) \;  \nonumber  \\
        & 
        \quad \quad \sin^{d-k-2}(\varphi_{k+1}) \ldots \sin(\varphi_{d-2}) d\varphi_1 \ldots d\varphi_{d-1} 
         \nonumber
        \\ &= 2\pi \prod_{i=d-k+2}^{d} \left[ \sqrt{\pi} \frac{\Gamma\left(\frac{i}{2}+\frac12 \right)}{\Gamma\left(\frac{i}{2}+1 \right)} \right]
        \frac{\sqrt{\pi}}{2} \frac{\Gamma\left(\frac{d-k-1}{2} + \frac{1}{2} \right)}{\Gamma\left(\frac{d-k-1}{2}+2 \right)}
        \prod_{i=1}^{d-k-2} \left[ \sqrt{\pi} \frac{\Gamma\left(\frac{i}{2}+\frac12 \right)}{\Gamma\left(\frac{i}{2}+1 \right)} \right] \nonumber
        \\
        &= \frac{\pi^{d/2}}{\Gamma(d/2+1)}.
    \end{align*}
    We therefore obtain
    \begin{align*}
         \int_{\R^d} (x-m) (x-m)^T g_{m,\Sigma}(x) dx  & = \frac{\beta_d}{Z_{0,I,h}} \int_{\R^+} r^{d+1} h(r^2) dr \; \Sigma \\
         &= \frac{\Gamma(d/2)}{2 \Gamma(d/2+1)} \frac{\int_{\R^+} r^{d+1} h(r^2) dr}{\int_{\R^+} r^{d-1} h(r^2) dr} \; \Sigma \\
         & = \frac{1}{d} \frac{\int_{\R^+} r^{d+1} h(r^2) dr}{\int_{\R^+} r^{d-1} h(r^2) dr} \Sigma.
    \end{align*}
    Thus $g_{m,\Sigma}$ has a covariance $\Sigma$ if \eqref{eq:h-condition} is satisfied.
\end{proof}

The result of this lemma ensures that Theorem~\ref{thm:location-scatter} can be applied for such distributions. 
Note that gaussian distributions for which $h(x) = \exp(-x/2)$ indeed satisfy  equation~\eqref{eq:h-condition}.

\subsection{Numerical tests}
\label{seq:numerics}

In this section, we provide a few practical examples for three types of mixtures. Two are based on elliptical distributions, namely based on Slater functions and on the Wigner semicircle distribution. The third one is based on the gamma distribution, and illustrate that the framework developed in this article is not limited to elliptical distributions.

\subsubsection{Numerical setting}

The numerical tests presented below have been implemented with the Julia language~\cite{bezanson2017julia} and the Wasserstein barycenters for the $W_2$ metric have been computed with the Python optimal transport library POT~\cite{flamary2021pot}. 
All Wasserstein barycenters for the $W_2$ metric are computed using the log-sinkhorn algorithm to avoid numerical errors, with a regularization parameter of $10^{-4}$ and a maximum number of iterations of $10 000$.
The one-dimensional cases are computed on a grid containing 200 points, and the two-dimensional cases on a grid with 50 points per dimension.
The code used for generating all the figures can be downloaded at \texttt{\url{https://github.com/dussong/W2_mixtures.jl/}}. 

In terms of computational cost, it is difficult to provide proper timings to compare the cost of computing the $W_2$ barycenters and the mixture baycenters based on the mixture metric denoted by $W_{2,\mathcal M}$ 
as it highly depends on the choice of the grid, the smoothing parameter, and number of iterations for the Sinkhorn algorithm for the $W_2$ calculations. Note however that the cost of computing the mixture Wasserstein barycenters does not depend on any spatial grid, as the size of the problem is only related to the number of atoms in the considered mixtures. Therefore, with the provided parameters, the order of magnitude for computing one mixture barycenter is less than 1ms while the computation of the $W_2$ barycenters is of the order of 10s for the one-dimensional cases and about 8 minutes for the two-dimensional cases, hence several order of magnitude more expensive than the mixture-based calculations.

\subsubsection{Slater-type elliptical distributions}

We consider here Slater-type elliptical distributions, that is we consider $h(x) = \exp(-\alpha_d |x|^{1/2})$ where $\alpha_d$ is adapted with respect to $d$ in order to satisfy equation~\eqref{eq:h-condition}. Namely,~\eqref{eq:h-condition} is satisfied if
\[
    \alpha_d = \sqrt{d+1}, \quad \text{and} \quad Z_{\Sigma} = \frac{2\pi^{d/2}}{(d+1)^{d/2}} \frac{\Gamma(d)}{\Gamma(d/2)} 
    |\det \Sigma|^{1/2}.
\]

The identifiability of the mixtures of such atoms can be proved with similar arguments as in~\cite[Proposition 2]{Delon2020-wk}. 
Below we provide one-dimensional and two-dimensional examples for Slater-type elliptical distributions.
 
\paragraph{One-dimensional case} Using the previous argument, we consider the function $h(x) = \exp(-\sqrt{2}|x|^{1/2})$. 
In Figure~\ref{fig:slater1d}, we present the Wasserstein barycenters with respect to the $W_2$ and $W_{2,\mathcal M}$ metrics between two mixtures of two atoms each. We observe that the two barycenters look quite different. In particular the $W_2$ barycenter computed with a Sinkhorn algorithm is smoother than the $W_{2,\mathcal M}$ barycenter, which is a mixture of three Slater functions, hence inherits three cusps.

\begin{figure}
    \centering
    \subfigure[$t=0$]{\includegraphics[width=0.19\textwidth]{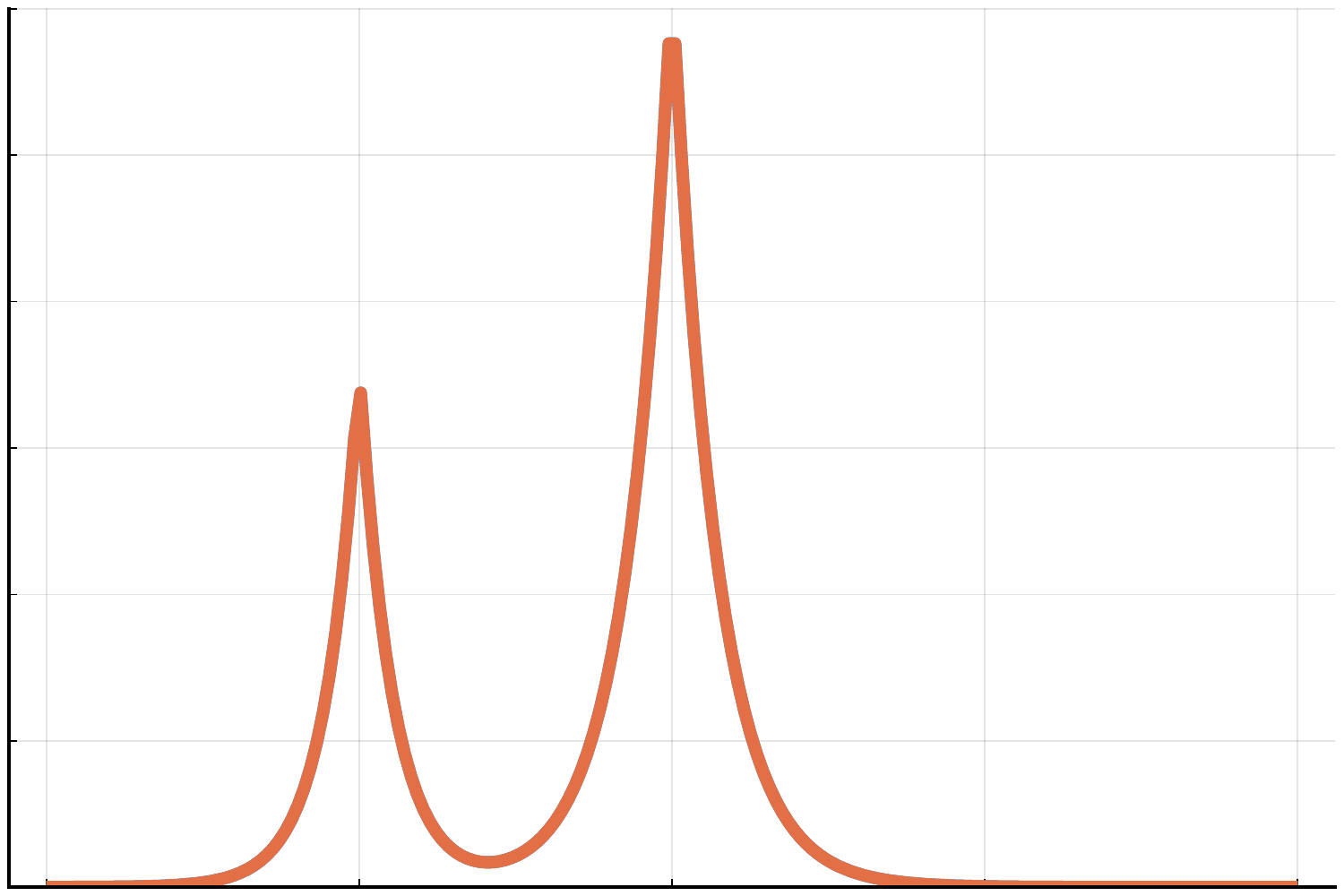}} 
    \subfigure[$t=0.25$]{\includegraphics[width=0.19\textwidth]{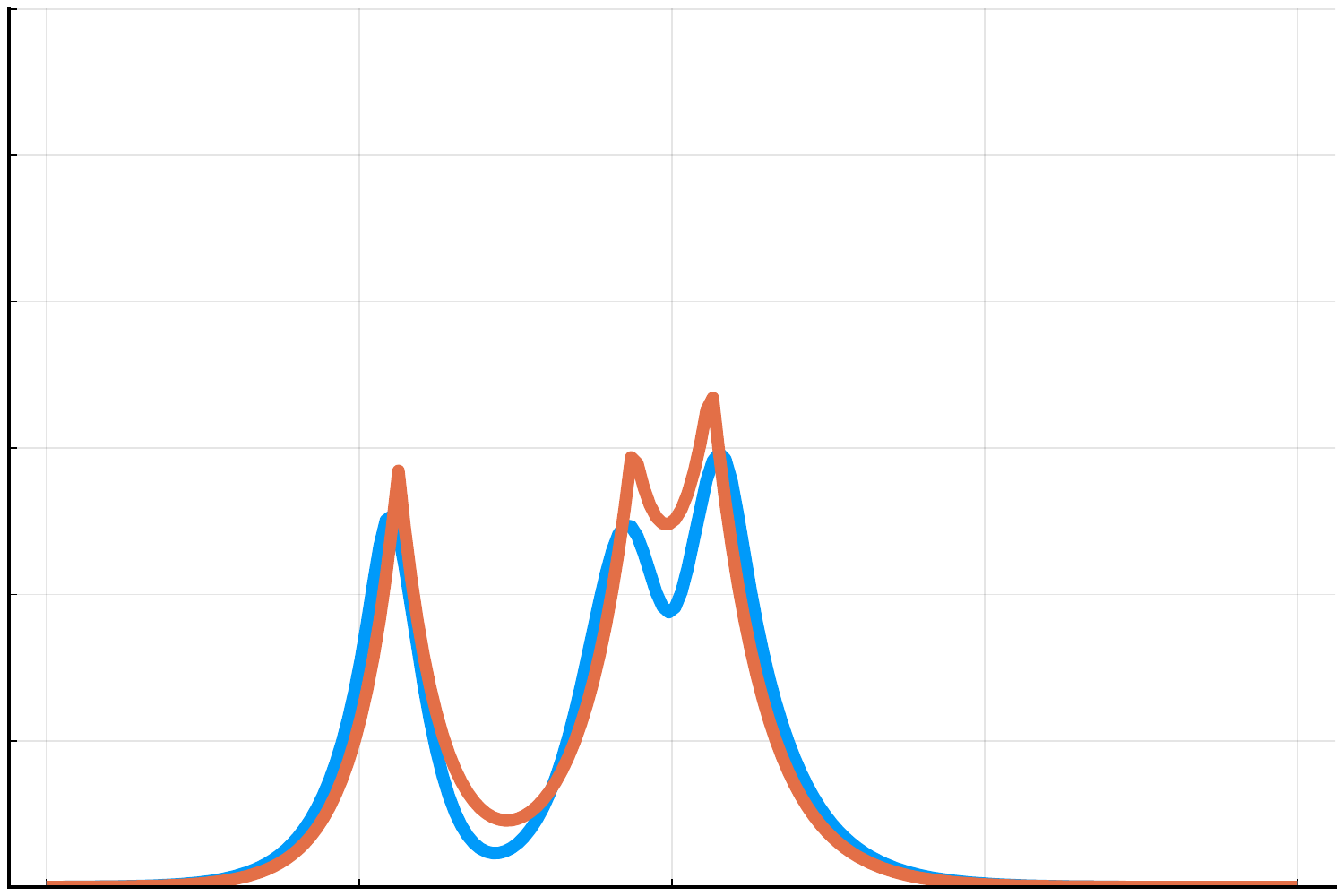}} 
    \subfigure[$t=0.5$]{\includegraphics[width=0.19\textwidth]{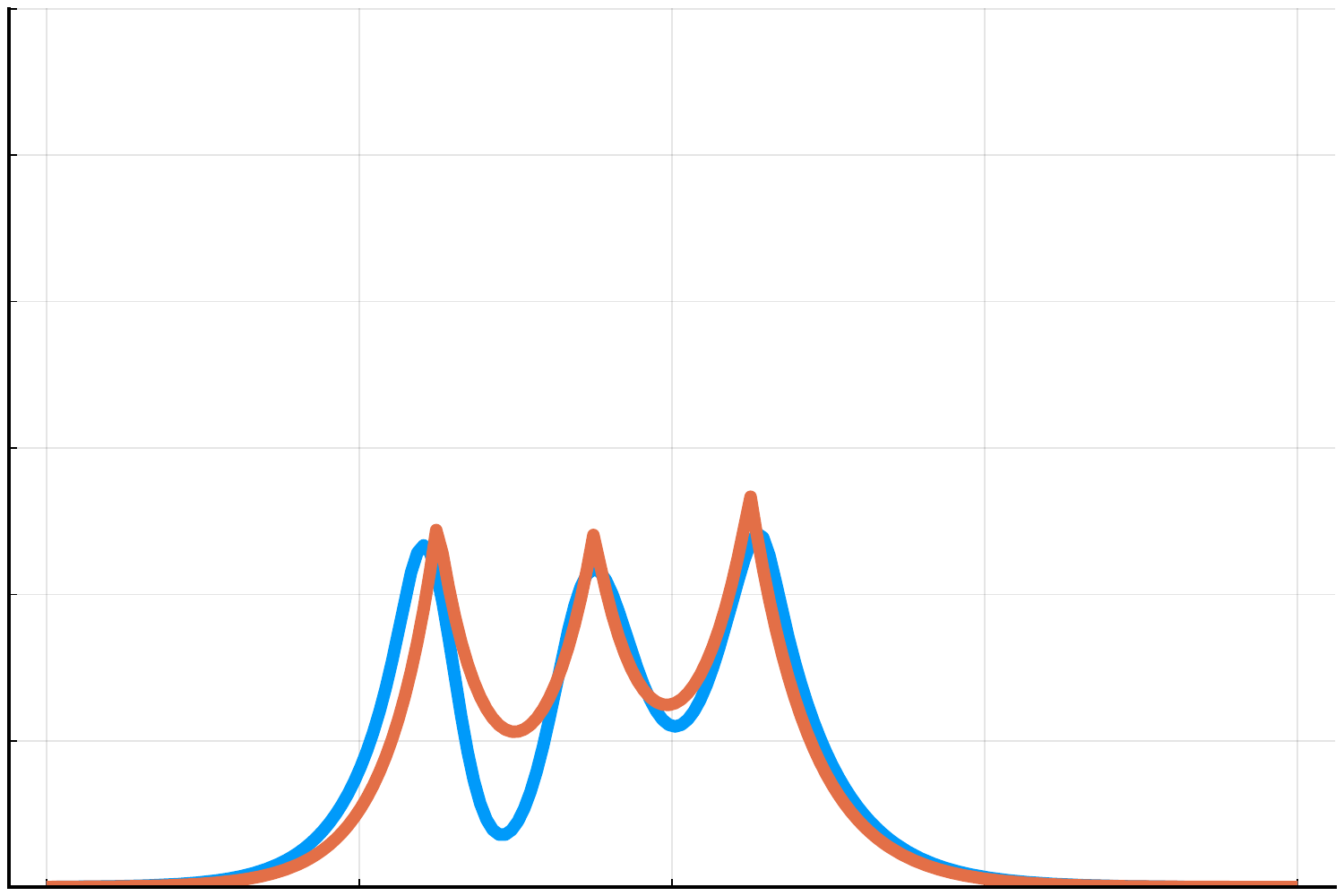}}
    \subfigure[$t=0.75$]{\includegraphics[width=0.19\textwidth]{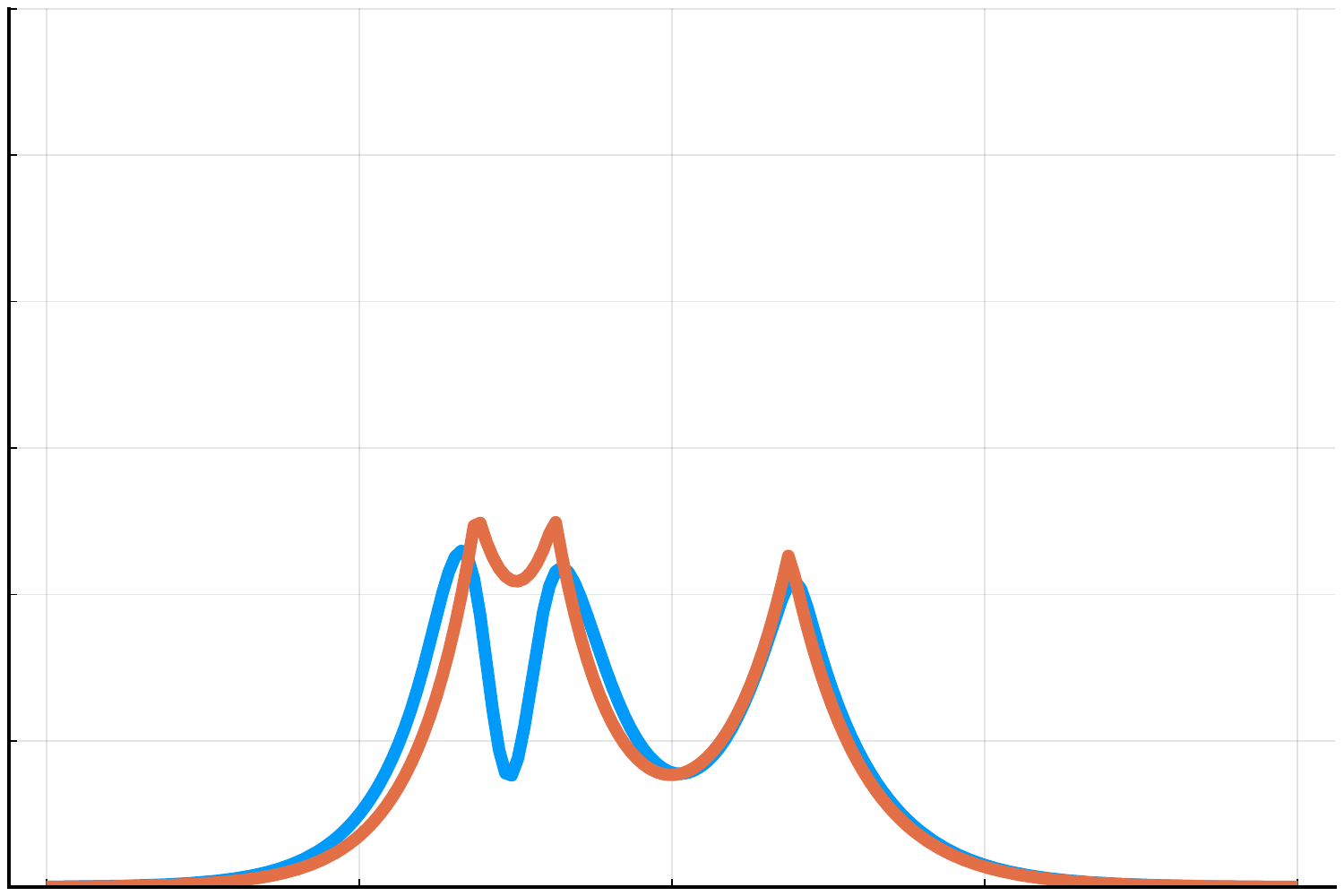}}
    \subfigure[$t=1$]{\includegraphics[width=0.19\textwidth]{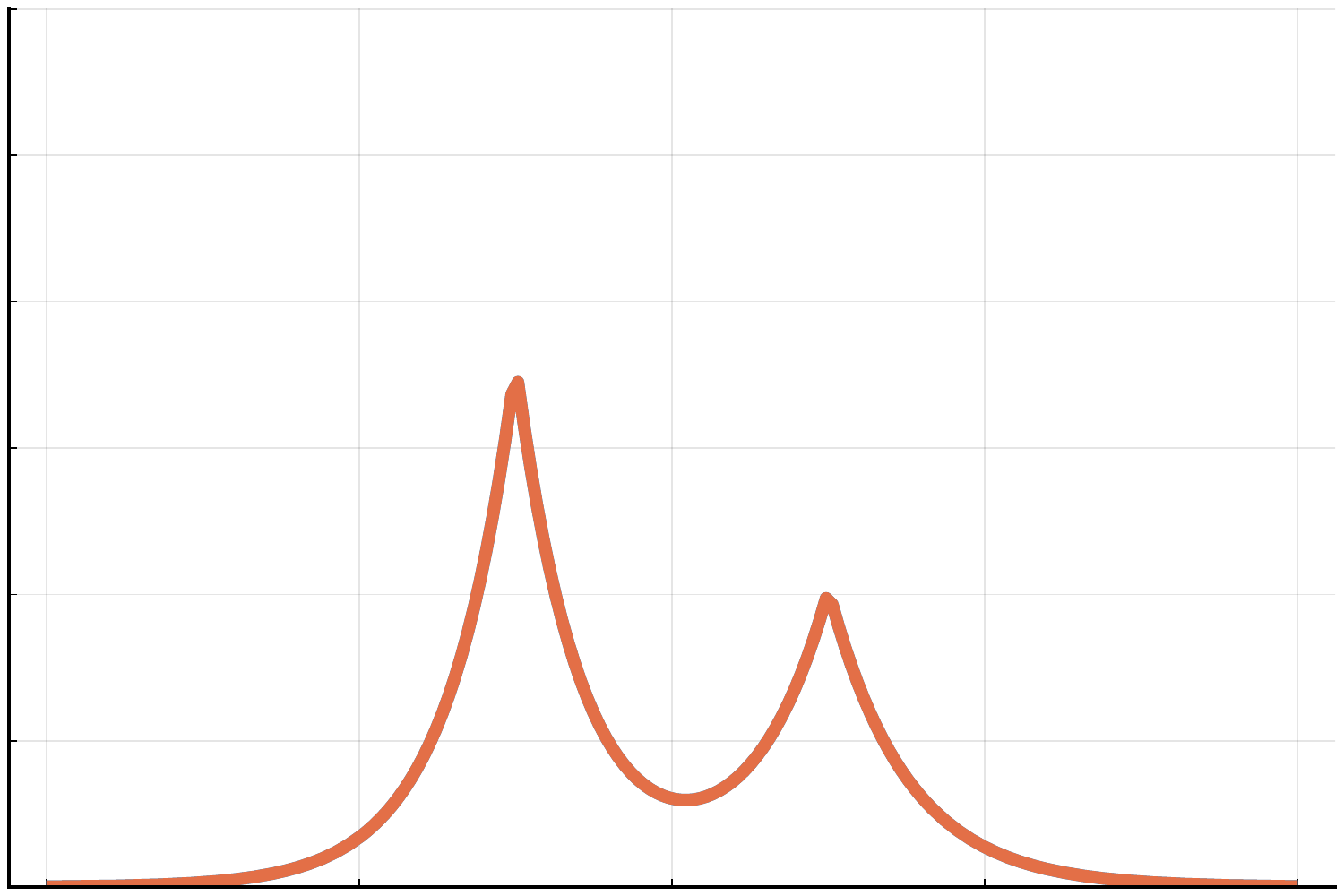}}
    \caption{Wasserstein barycenters between two mixtures of Slater-type elliptical distributions for the $W_2$ metric (blue) and the $W_{2,\mathcal M}$  metric (red).}
    \label{fig:slater1d}
\end{figure}

\paragraph{Two-dimensional case} In two dimensions, we consider $h(x) = \exp(-\sqrt{3}|x|^{1/2})$, for which we check that equation~\eqref{eq:h-condition} is satisfied. 
In Figures~\ref{fig:slater2dmw2_contour} and~\ref{fig:slater2dmw2}, we present the $W_{2,\mathcal M}$ and $W_2$ barycenters for mixtures of two-dimensional Slater-type elliptical distributions. As for the one-dimensional case, the $W_2$ barycenter is smoother than the $W_{2,\mathcal M}$ barycenter.

\begin{figure}
    \centering
    \subfigure[$t=0$]{
    \begin{tabular}{@{}c@{}}
         \includegraphics[width=0.18\textwidth]{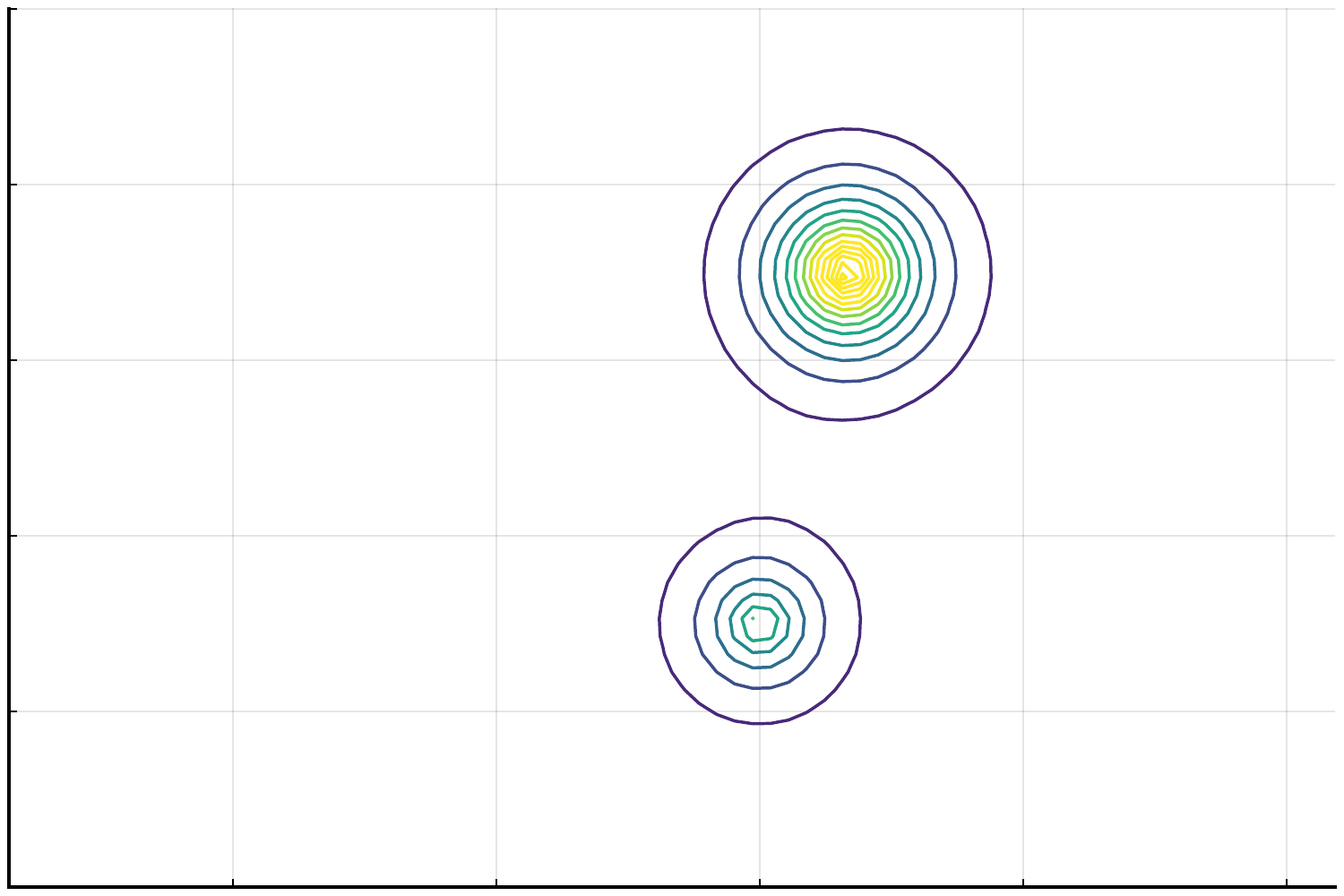} \\
         \includegraphics[width=0.18\textwidth]{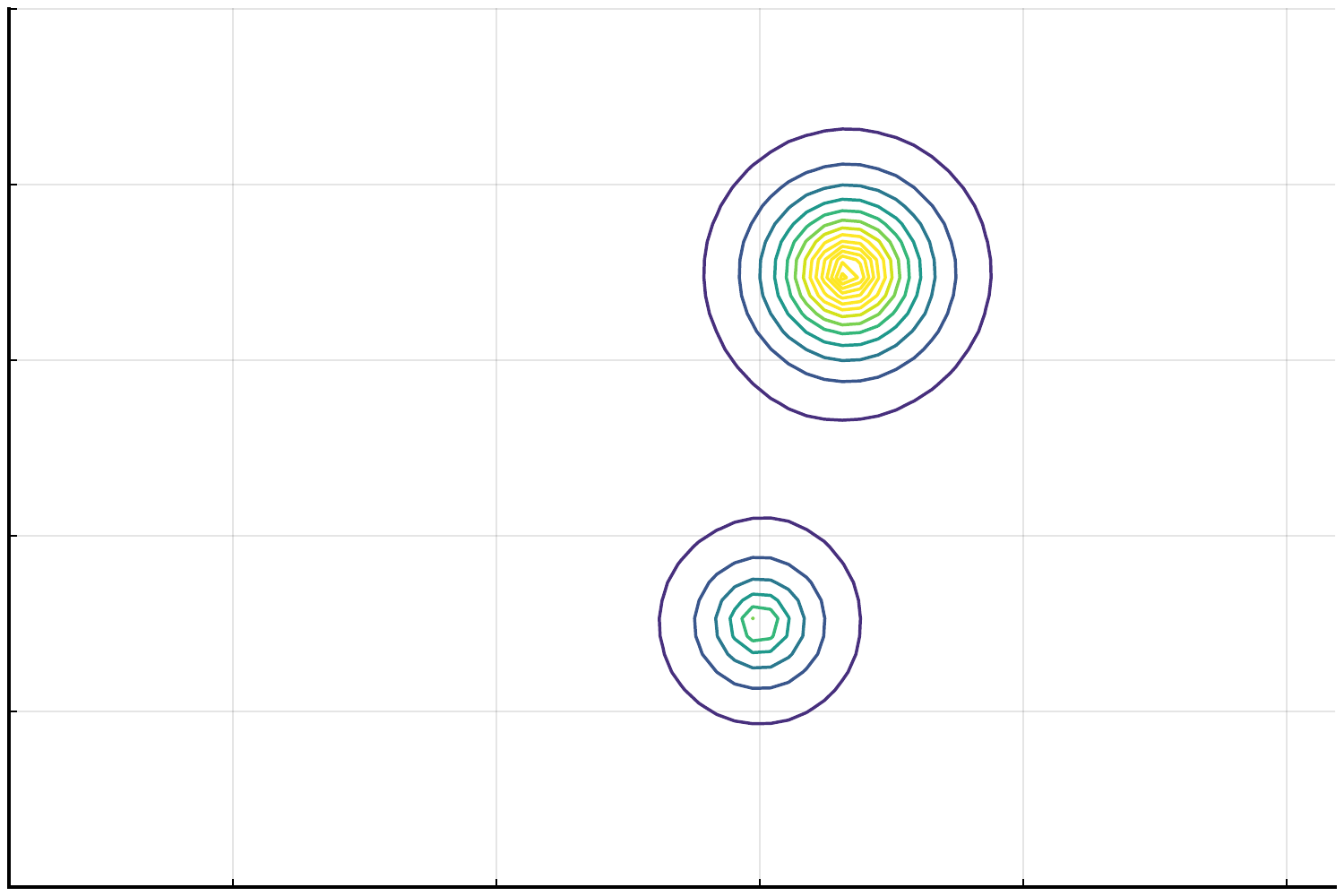}
    \end{tabular}
    } 
    \subfigure[$t=0.25$]{\begin{tabular}{@{}c@{}}
         \includegraphics[width=0.18\textwidth]{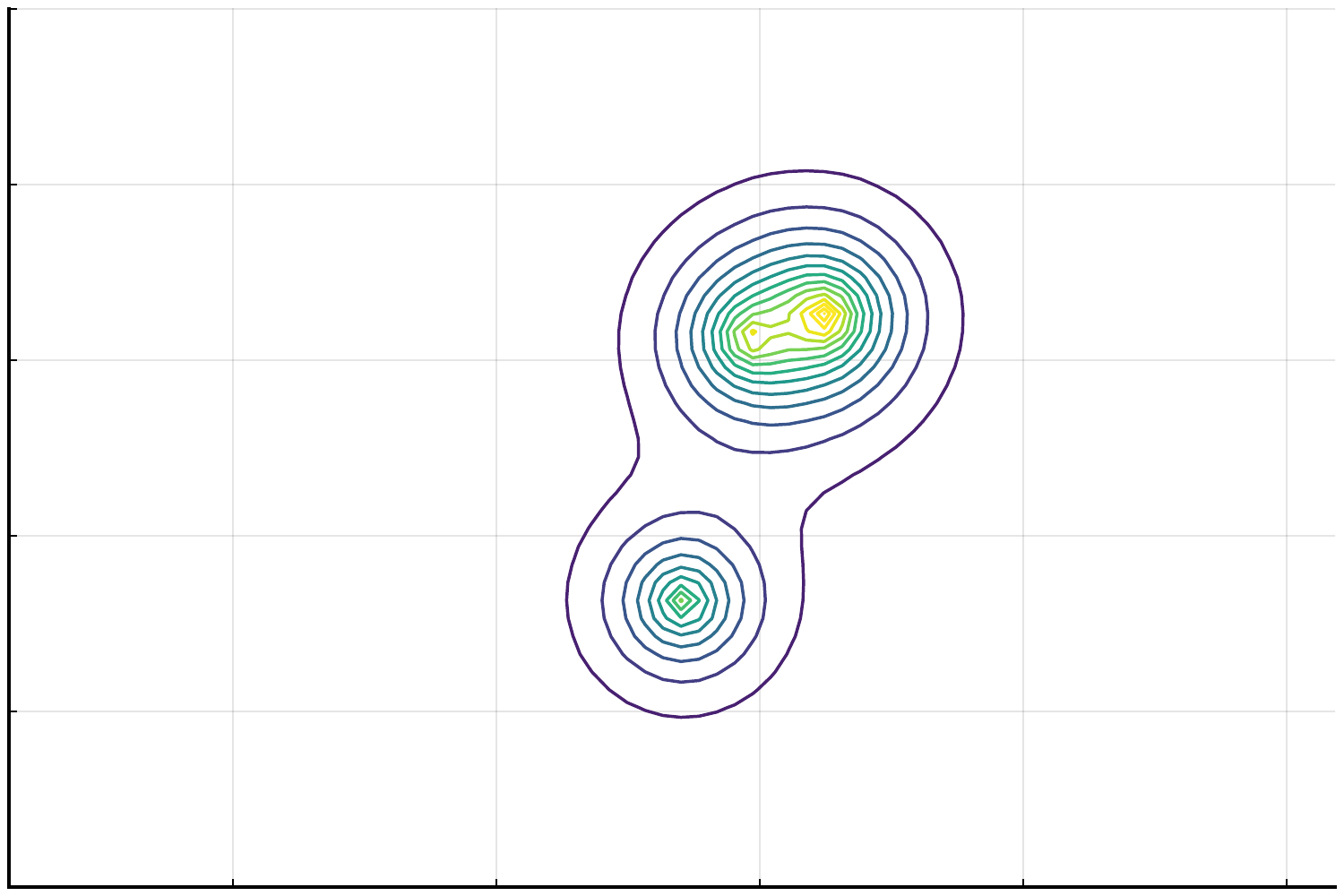} \\
         \includegraphics[width=0.18\textwidth]{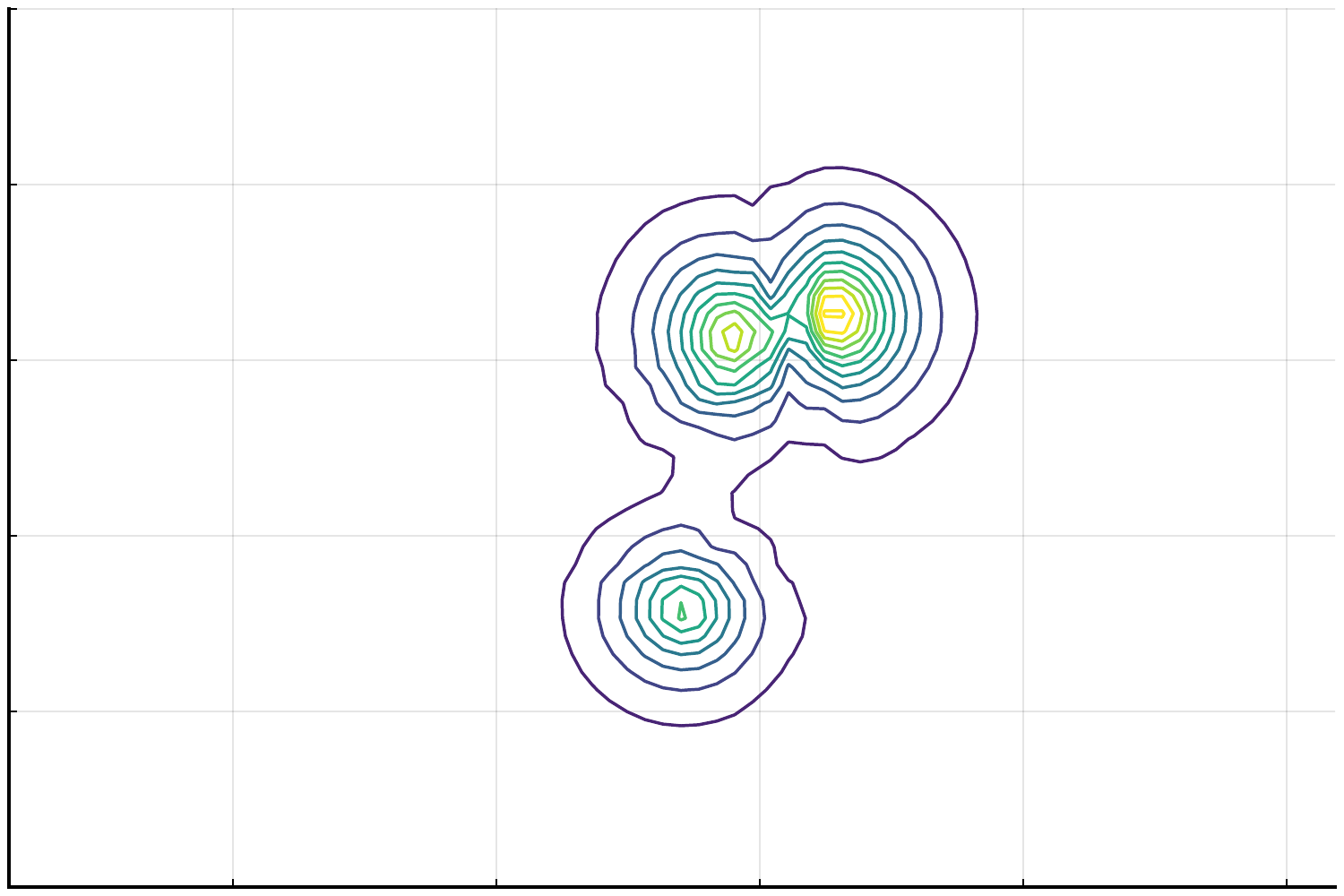}
    \end{tabular}} 
    \subfigure[$t=0.5$]{\begin{tabular}{@{}c@{}}
         \includegraphics[width=0.18\textwidth]{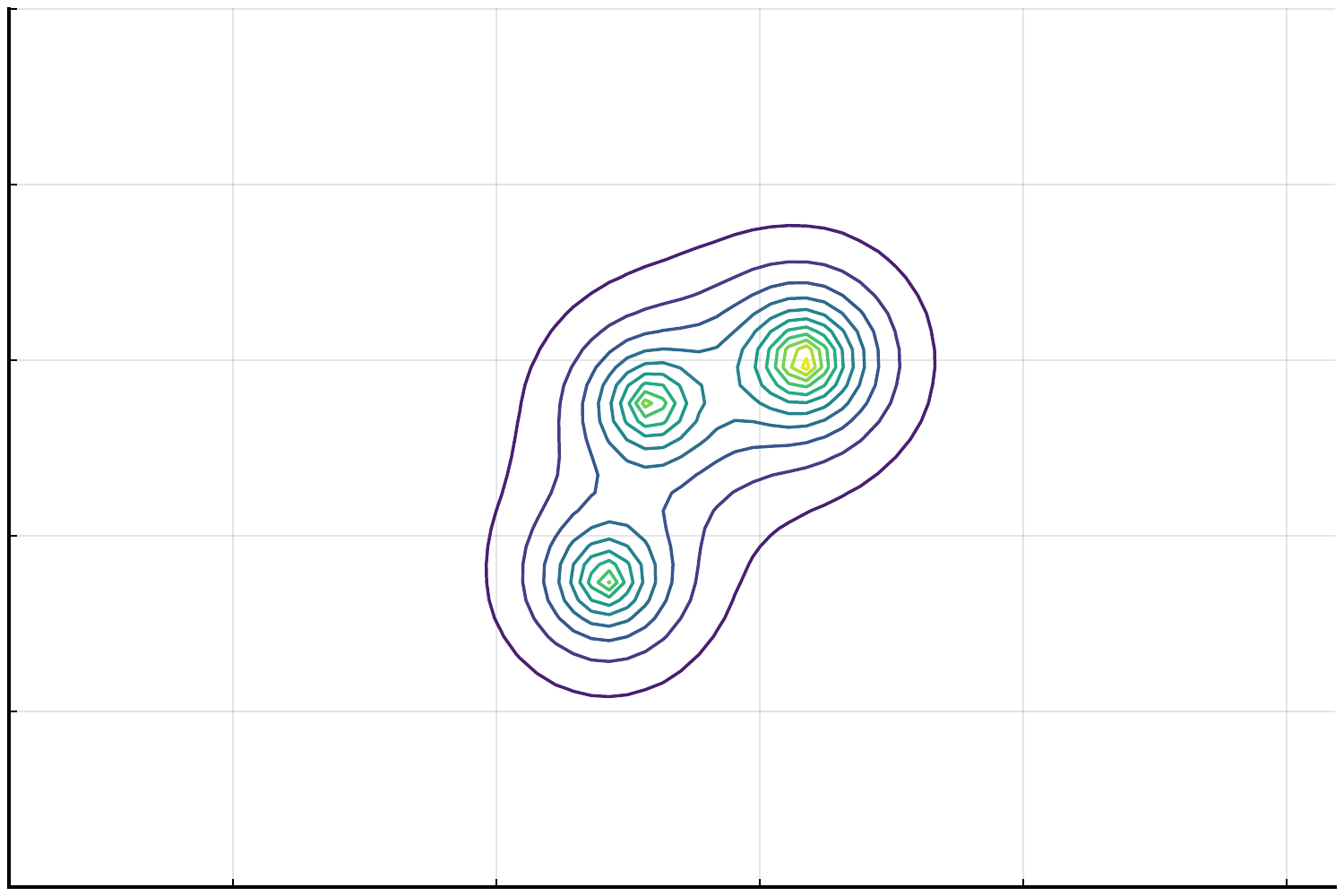} \\
         \includegraphics[width=0.18\textwidth]{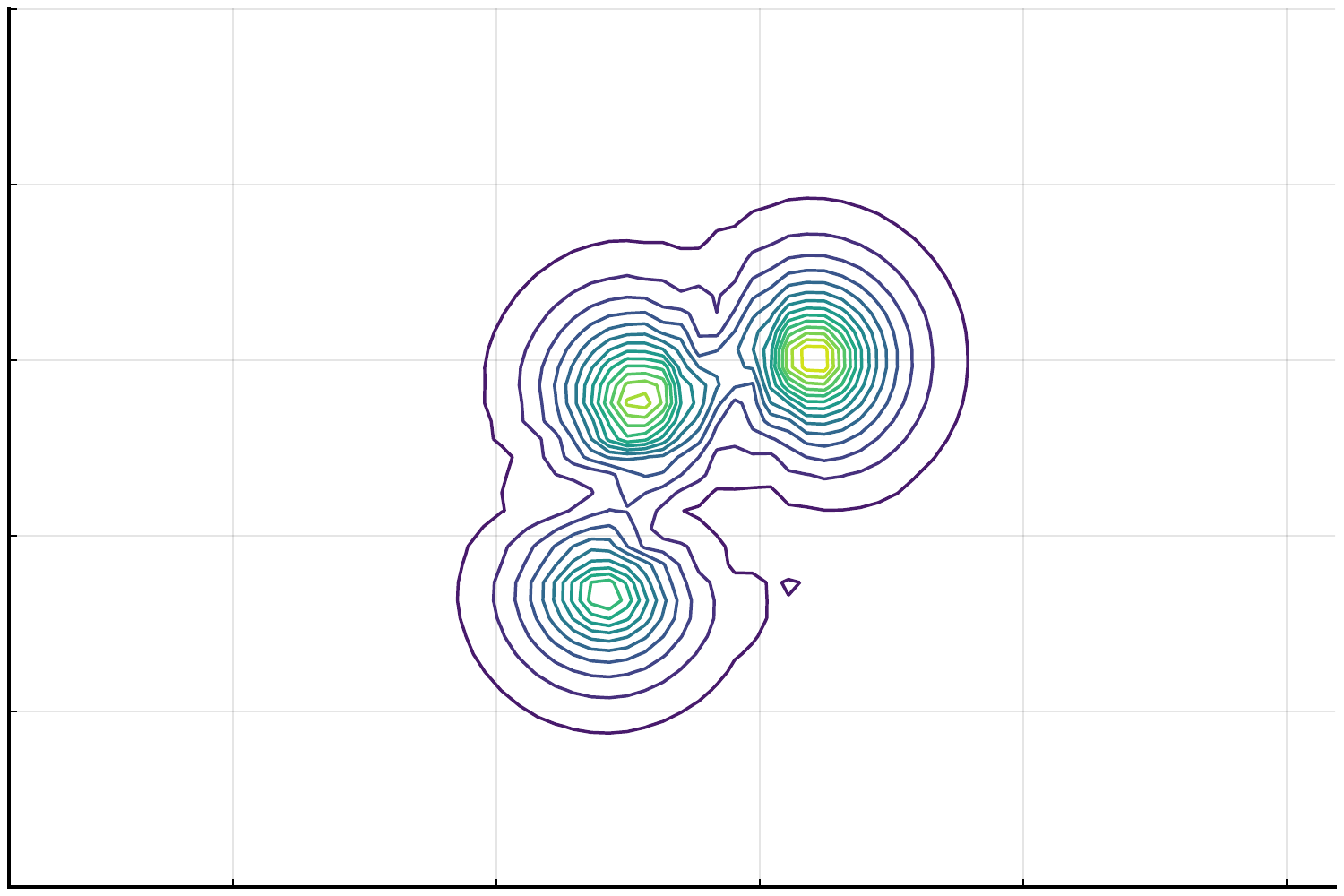}
    \end{tabular}} 
    \subfigure[$t=0.75$]{\begin{tabular}{@{}c@{}}
         \includegraphics[width=0.18\textwidth]{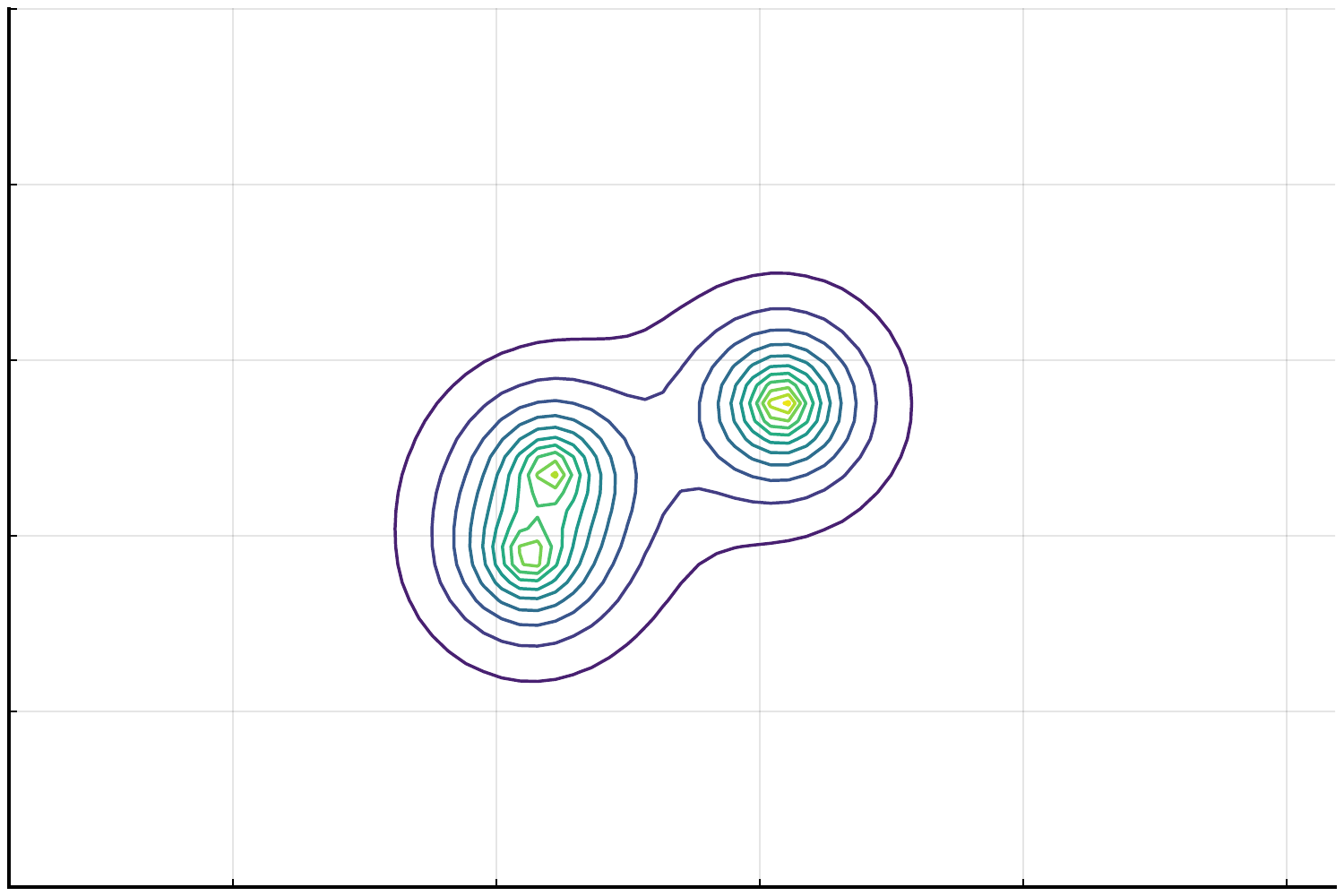} \\
         \includegraphics[width=0.18\textwidth]{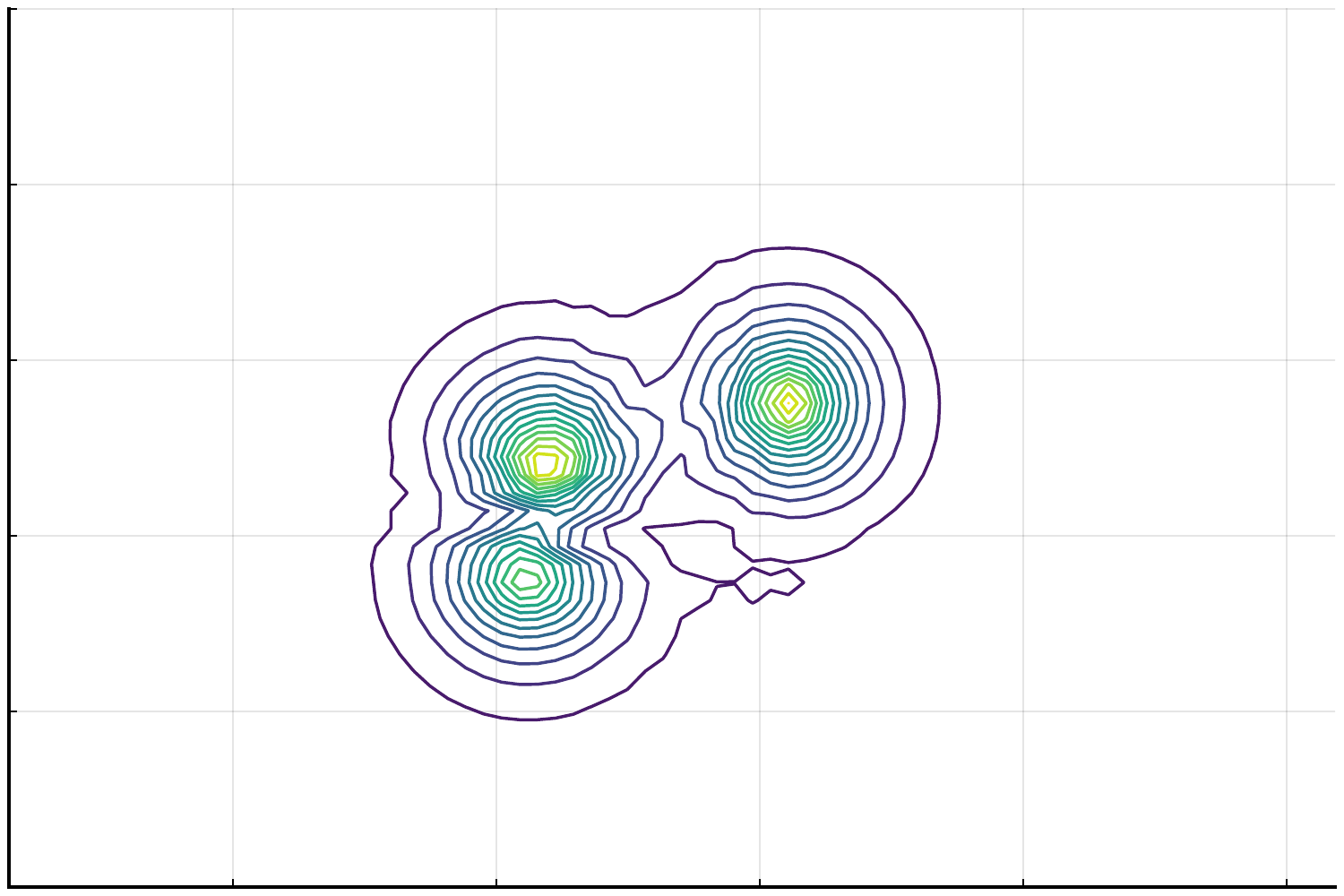}
    \end{tabular}} 
    \subfigure[$t=1$]{\begin{tabular}{@{}c@{}}
         \includegraphics[width=0.18\textwidth]{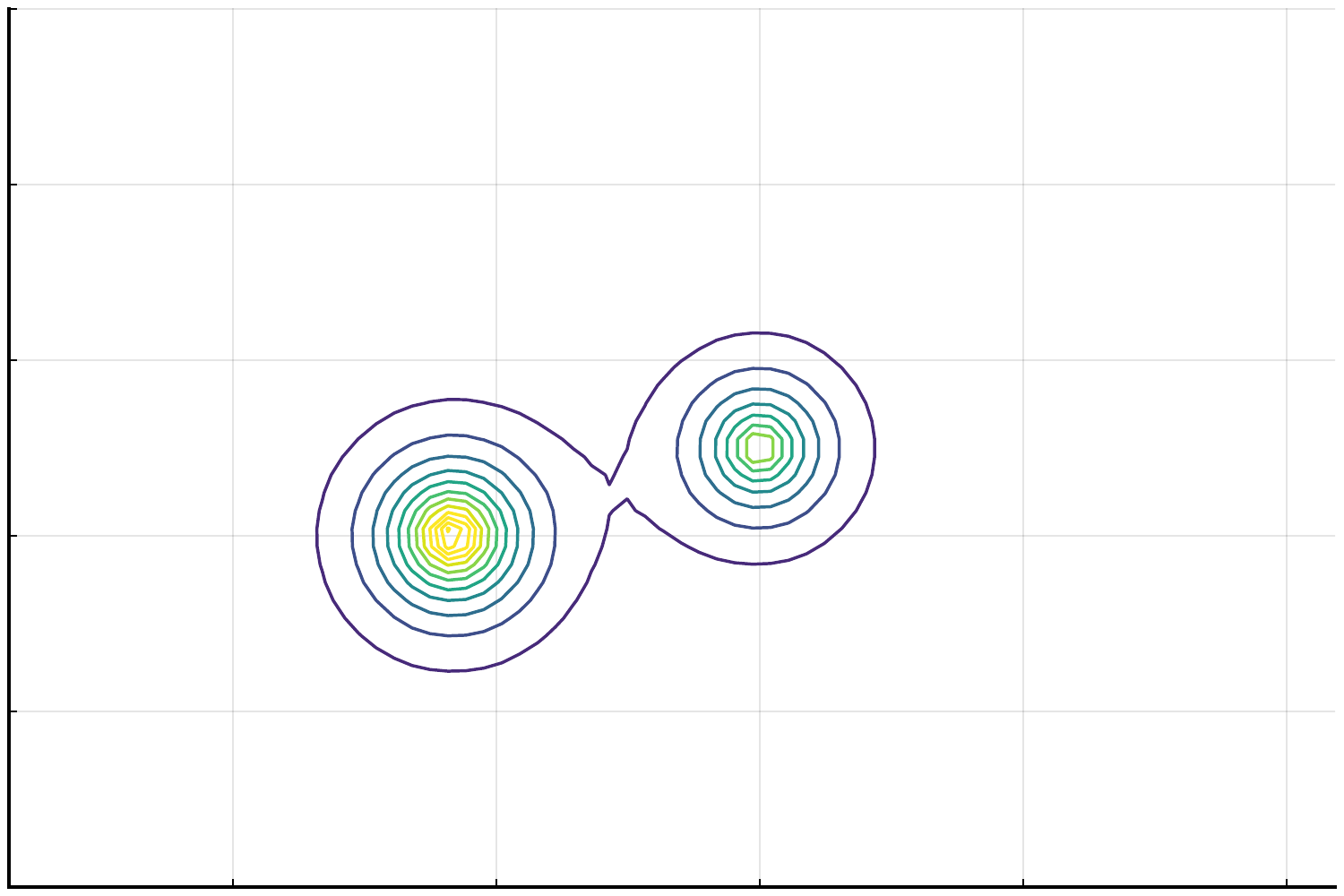} \\
         \includegraphics[width=0.18\textwidth]{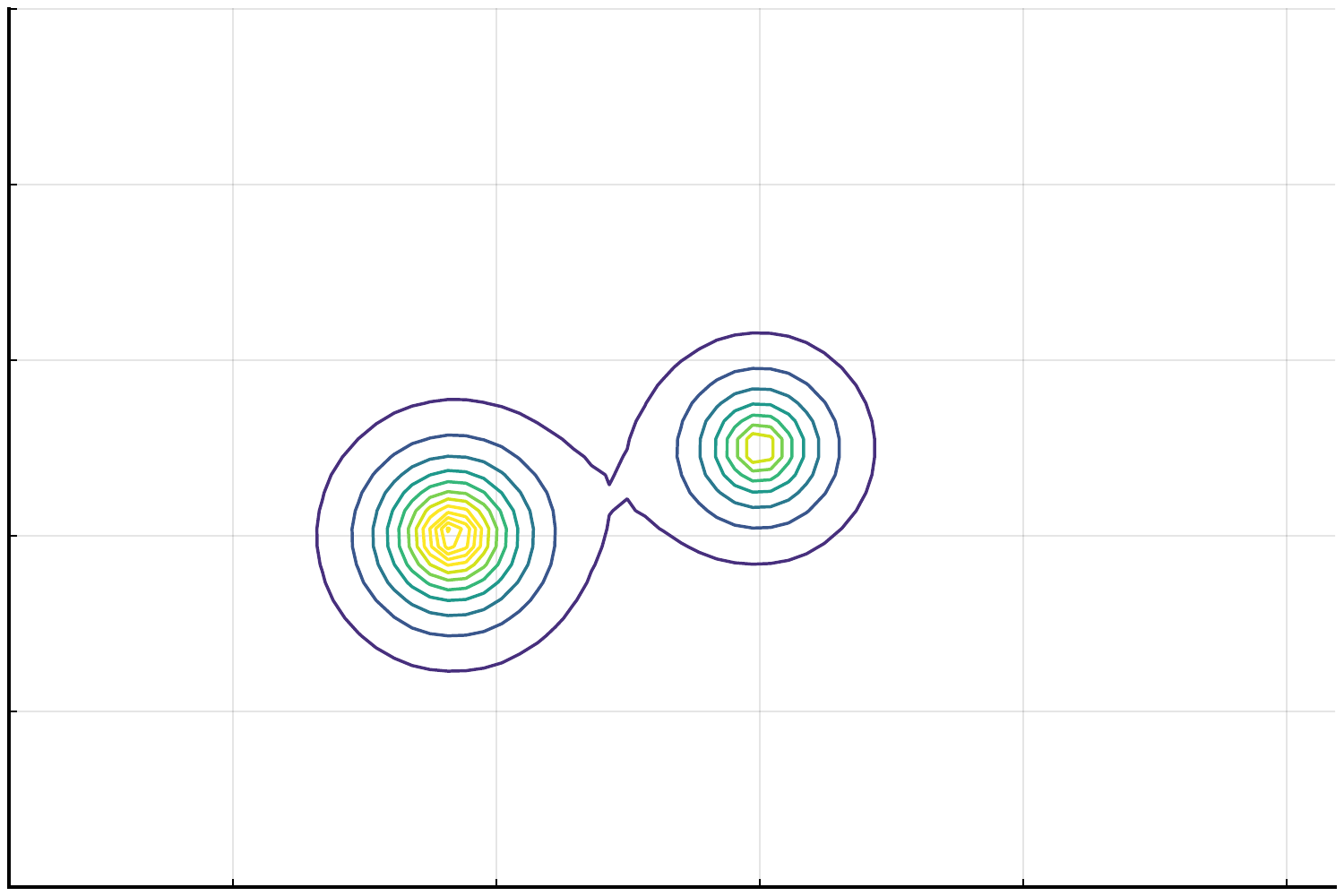}
    \end{tabular}} 
    \caption{Contour plots of $W_{2,\mathcal M}$ (top) and $W_2$ (bottom) barycenters between two mixtures of Slater-type elliptic distributions.}
    \label{fig:slater2dmw2_contour}
\end{figure}

\begin{figure}
    \centering
    \subfigure[$W_{2,\mathcal M}$-barycenter]{\includegraphics[width=0.49\textwidth]{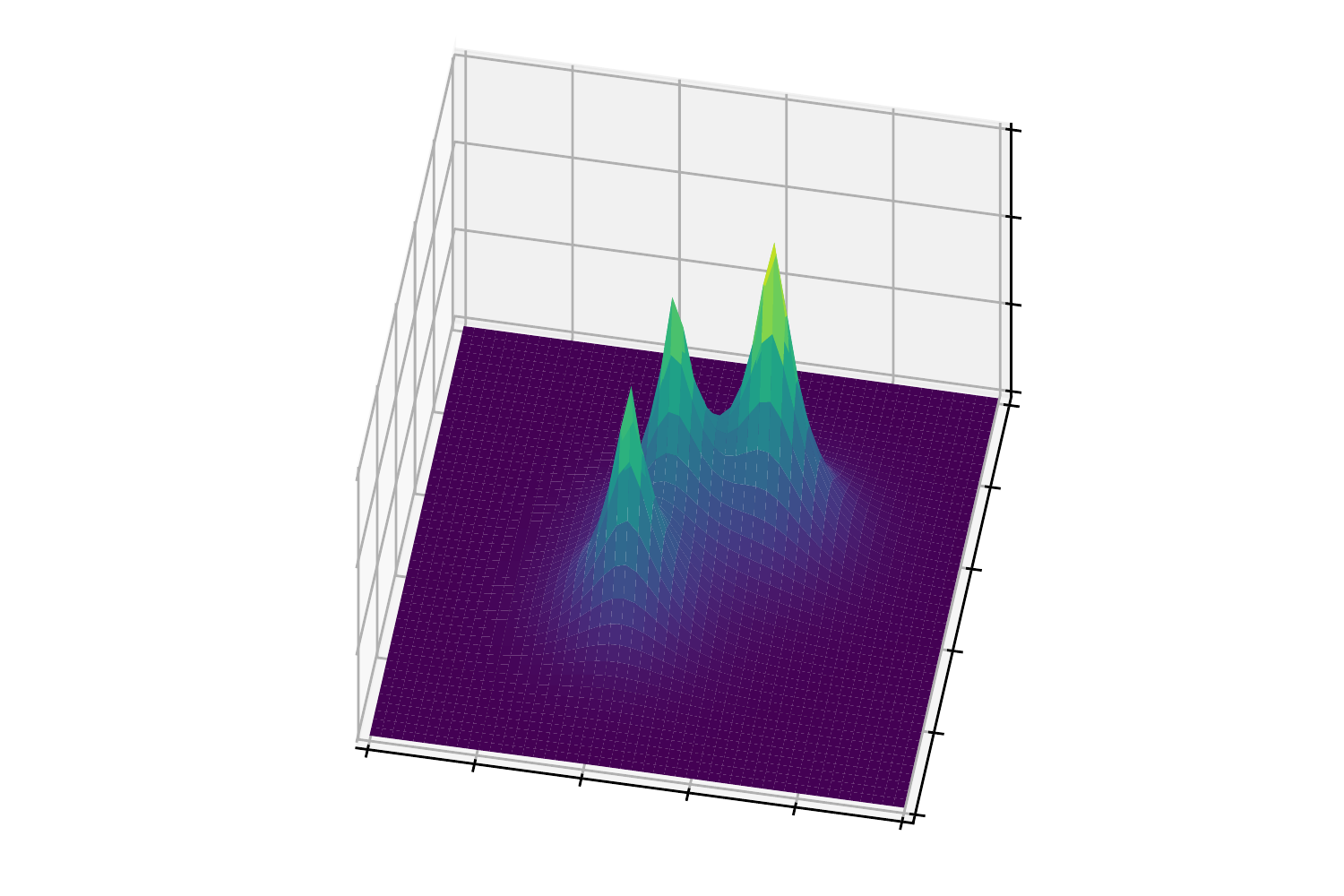}} 
    \subfigure[$W_2$-barycenter]{\includegraphics[width=0.49\textwidth]{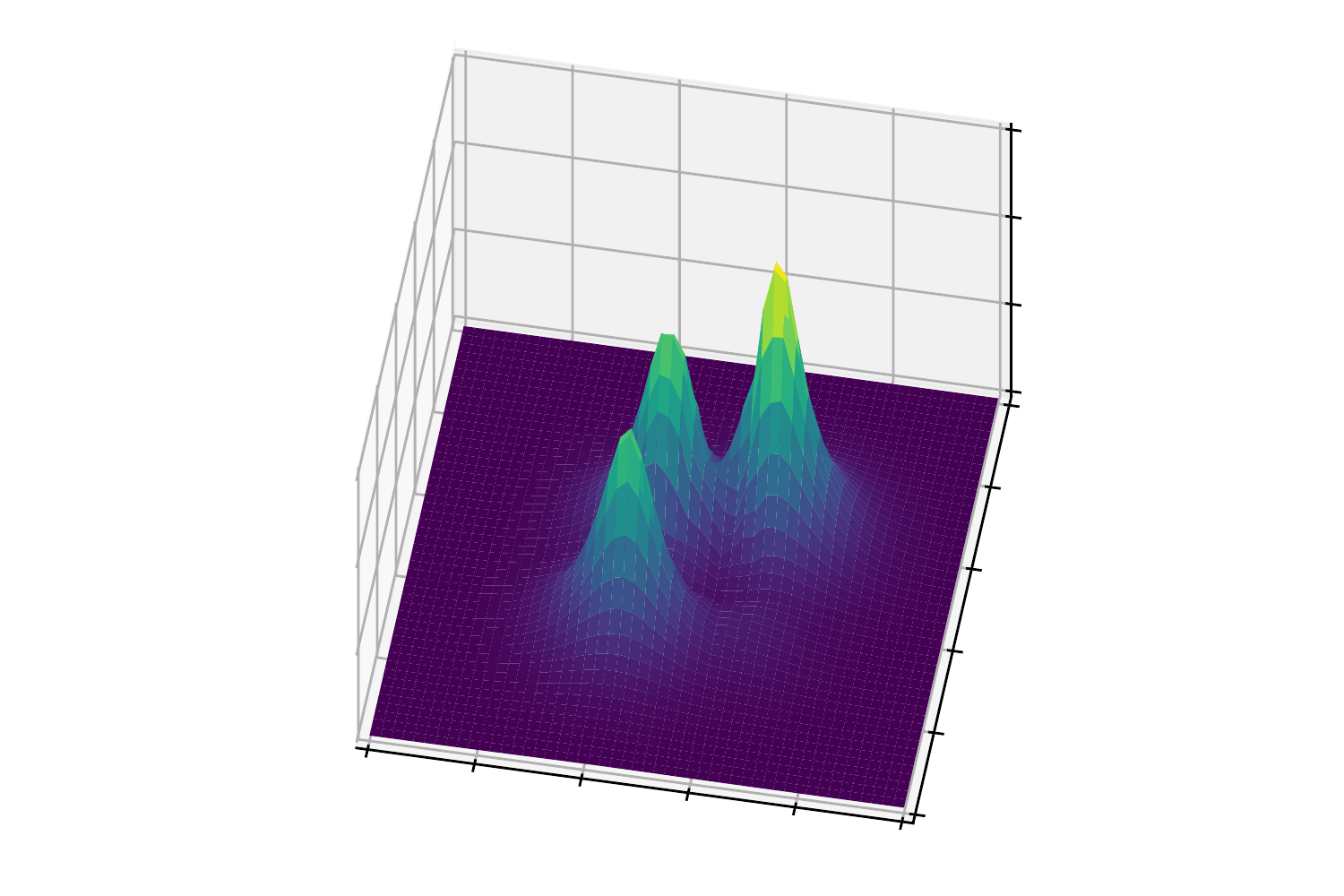}} 
    \caption{Comparison between $W_2$ and $W_{2,\mathcal M}$ barycenters for $t=0.5$ and the mixtures presented on Figure~\ref{fig:slater2dmw2_contour}.}
    \label{fig:slater2dmw2}
\end{figure}

\subsubsection{Wigner-semicircle-type elliptical distributions}

We now consider elliptical distributions based on the Wigner-semicircle distribution, i.e. we consider 
\[
h(x) = \left\{ 
\begin{array}{ll}
    \sqrt{1 - \alpha_d x}, & x \le \frac{1}{\alpha_d} \\
    0,  & x > \frac{1}{\alpha_d},
\end{array}
\right.
\]
For 
$
\alpha_d = \frac{1}{d+3},
$
one easily checks that equation~\eqref{eq:h-condition} is satisfied. Therefore, the densities of the probability distributions take the form
\[
\forall x\in\R^d, \quad g_{m,\Sigma}(x) =  \frac{1}{Z_{\Sigma}} \sqrt{ 
1- \frac{1}{d+3} (x-m)^T\Sigma^{-1} (x-m) },
\]
with 
\[
    Z_{\Sigma} = \frac{ \pi^{\frac{d+1}{2}} (d+3)^{d/2}}{2 \Gamma\left(\frac{d+3}{2}\right)} | \det\Sigma|^{1/2}.
\]

The identifiability of the mixtures based on these Wigner-semicircle type elliptical distribution distributions can be easily proved noting that the atoms can be iteratively uniquely characterized by their support.

\paragraph{One-dimensional case} 

We take $h(x) = \left\{ 
\begin{array}{ll}
    \sqrt{1 - x/4}, & x \le 4 \\
    0,  & x > 4
\end{array}
\right.$. 
In Figure~\ref{fig:semicircle1d}, we plot the $W_2$ and $W_{2,\mathcal M}$ barycenters between two mixtures of two atoms each. We observe that the two types of barycenters behave quite differently, as the zones with largest densities as pretty different. 

\begin{figure}
    \centering
    \subfigure[$t=0$]{\includegraphics[width=0.19\textwidth]{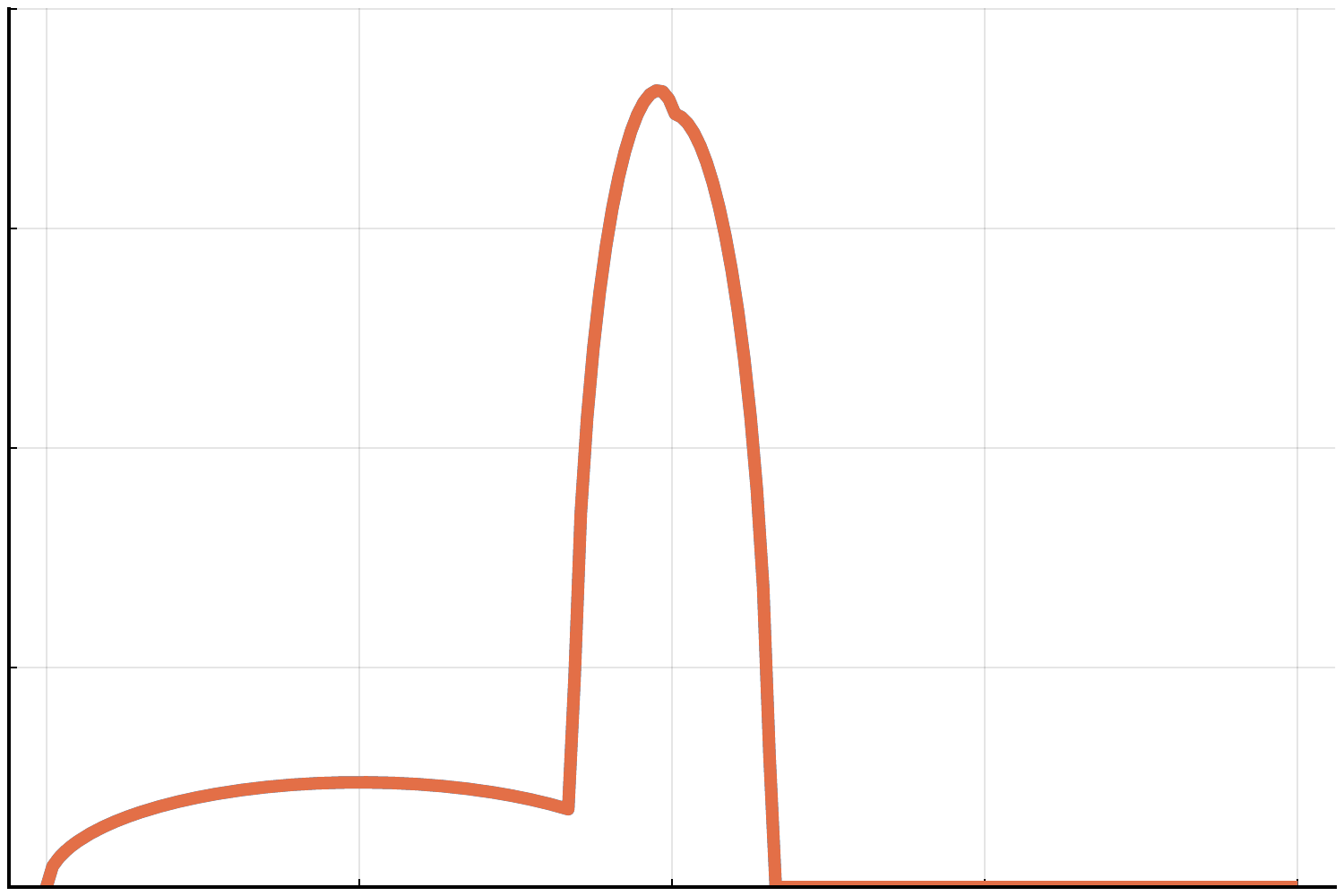}} 
    \subfigure[$t=0.25$]{\includegraphics[width=0.19\textwidth]{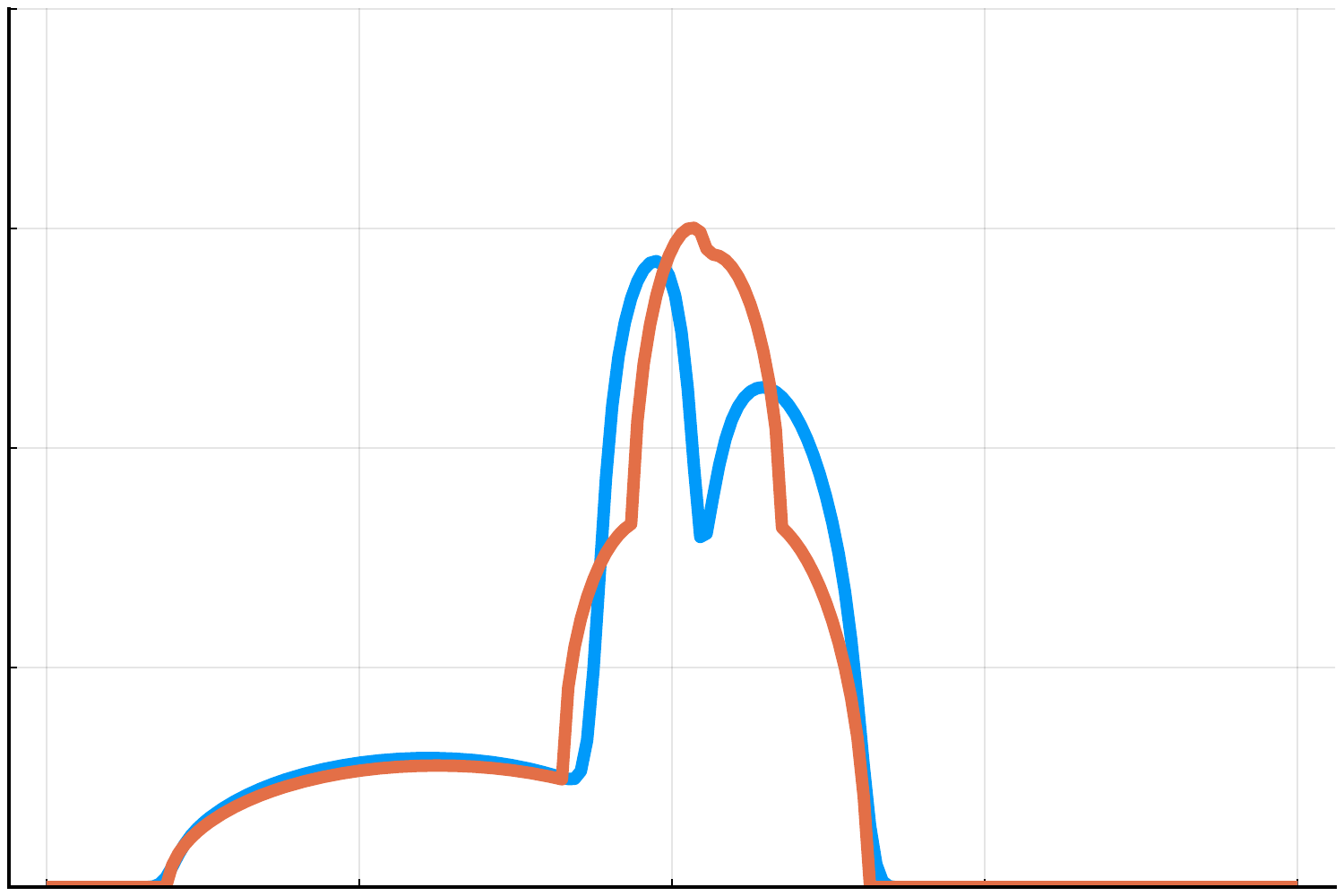}} 
    \subfigure[$t=0.5$]{\includegraphics[width=0.19\textwidth]{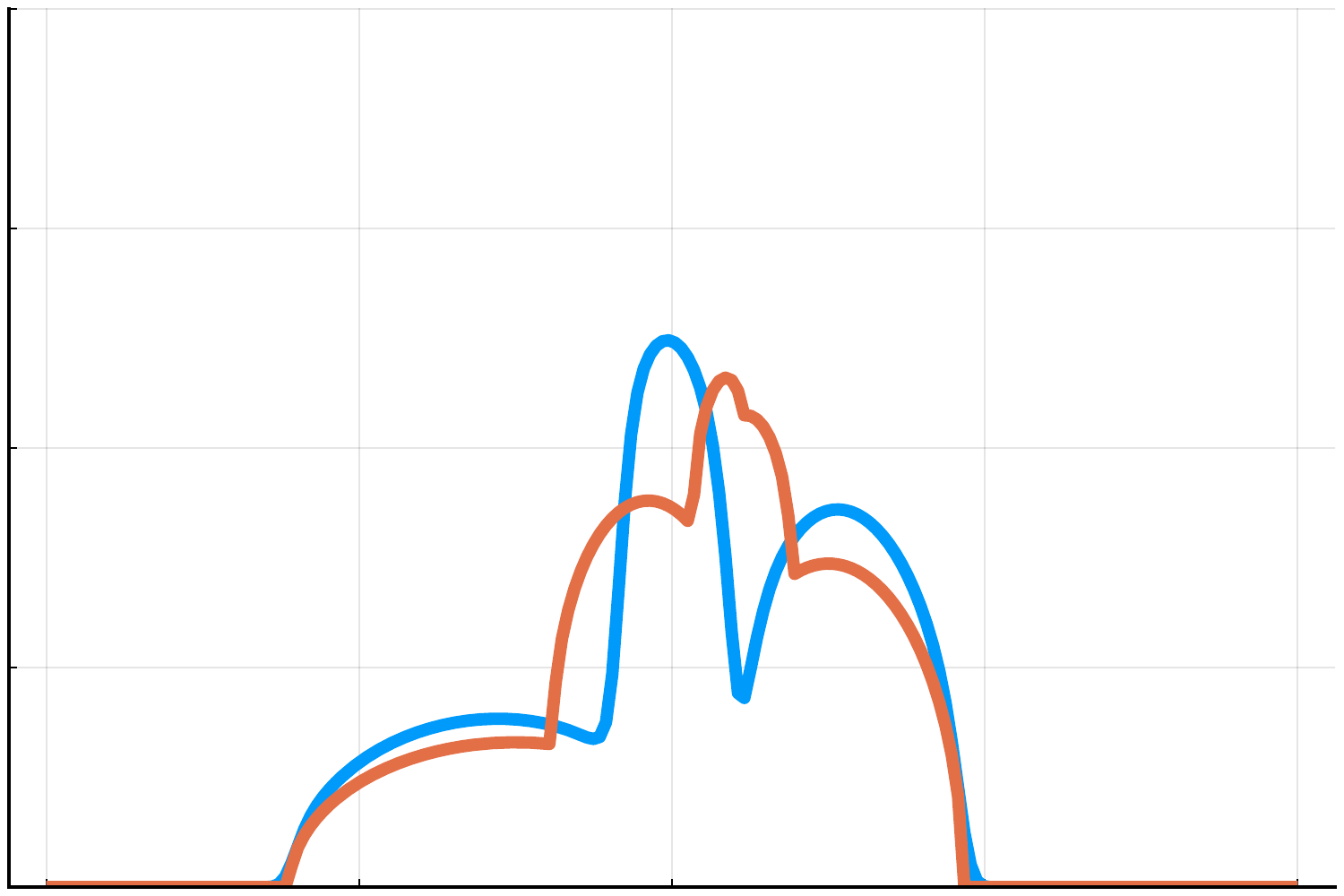}}
    \subfigure[$t=0.75$]{\includegraphics[width=0.19\textwidth]{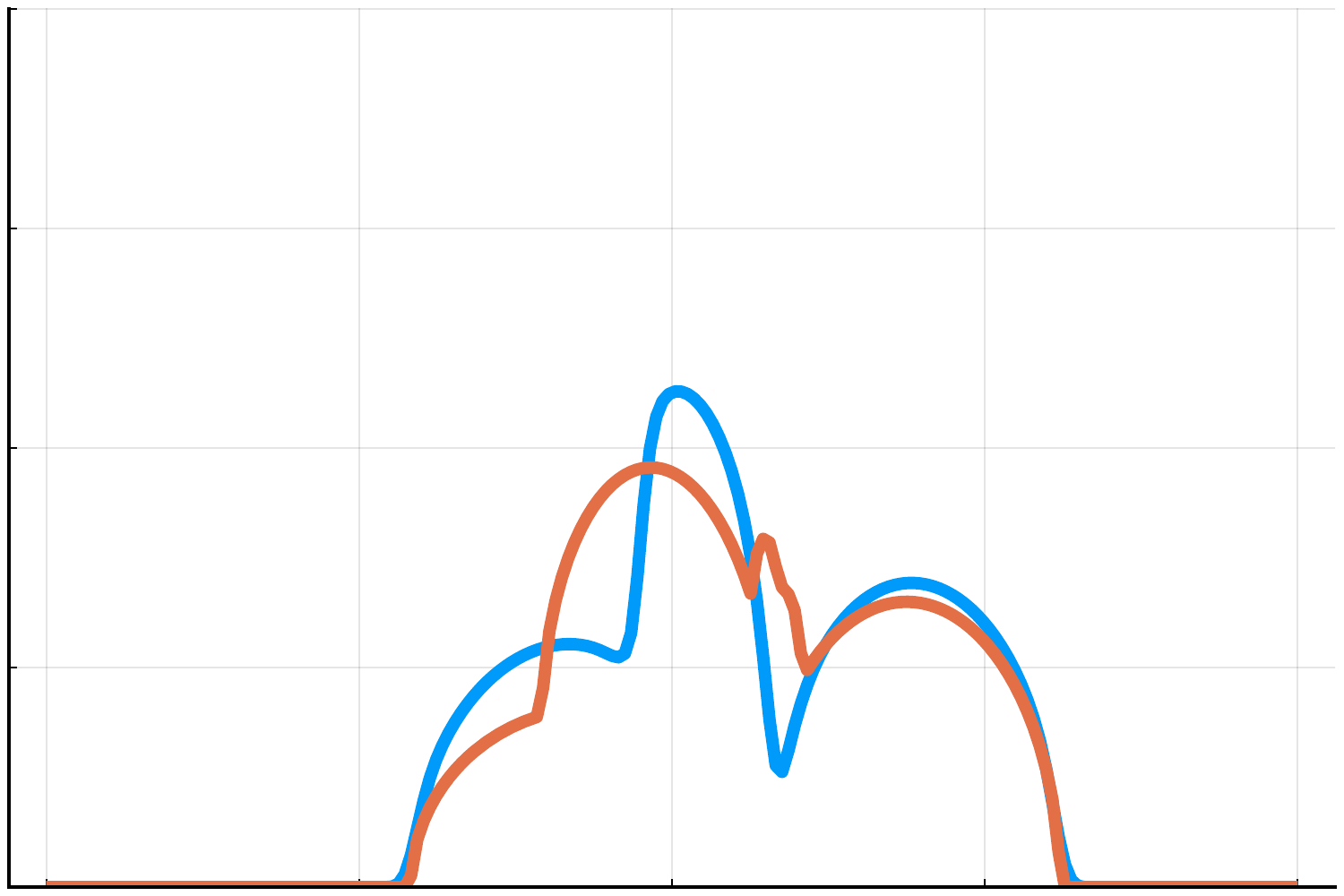}}
    \subfigure[$t=1$]{\includegraphics[width=0.19\textwidth]{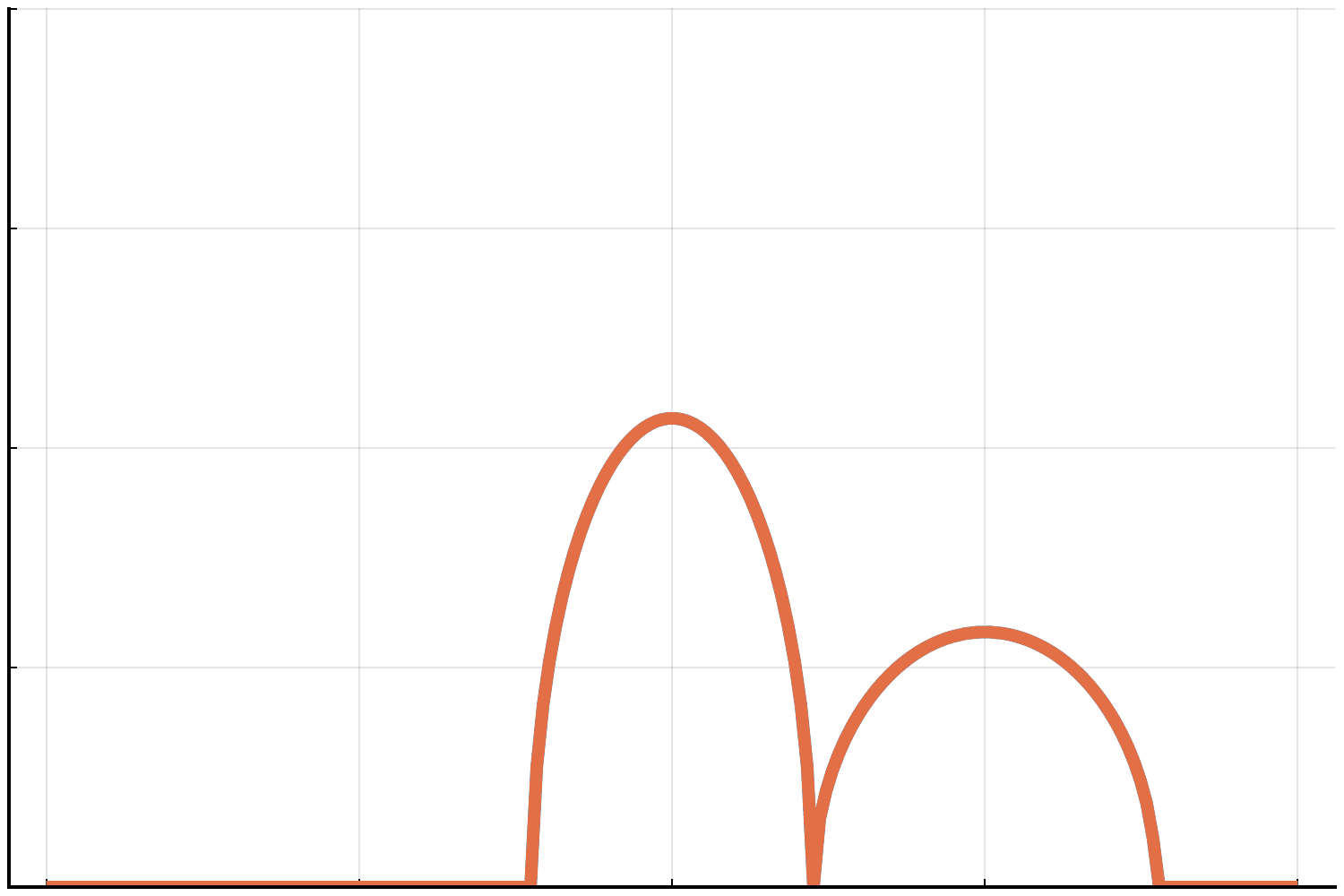}}
    \caption{Wasserstein barycenters between two mixtures of Wigner-semicircle elliptical distributions for the $W_2$ metric (blue) and the $W_{2,\mathcal M}$ metric (red).}
    \label{fig:semicircle1d}
\end{figure}

\paragraph{Two-dimensional case} 

We take $h(x) = \left\{ 
\begin{array}{ll}
    \sqrt{1 - x/5}, & x \le 5 \\
    0,  & x > 5
\end{array}
\right.$. 
In Figure~\ref{fig:semicircle2dmw2_contour} 
we plot the $W_2$ and $W_{2,\mathcal M}$ barycenters between two mixtures of two atoms each for the two-dimensional case. The observations are similar to the one-dimensional case.

\begin{figure}
    \centering
    \subfigure[$t=0$]{
    \begin{tabular}{@{}c@{}}
         \includegraphics[width=0.18\textwidth]{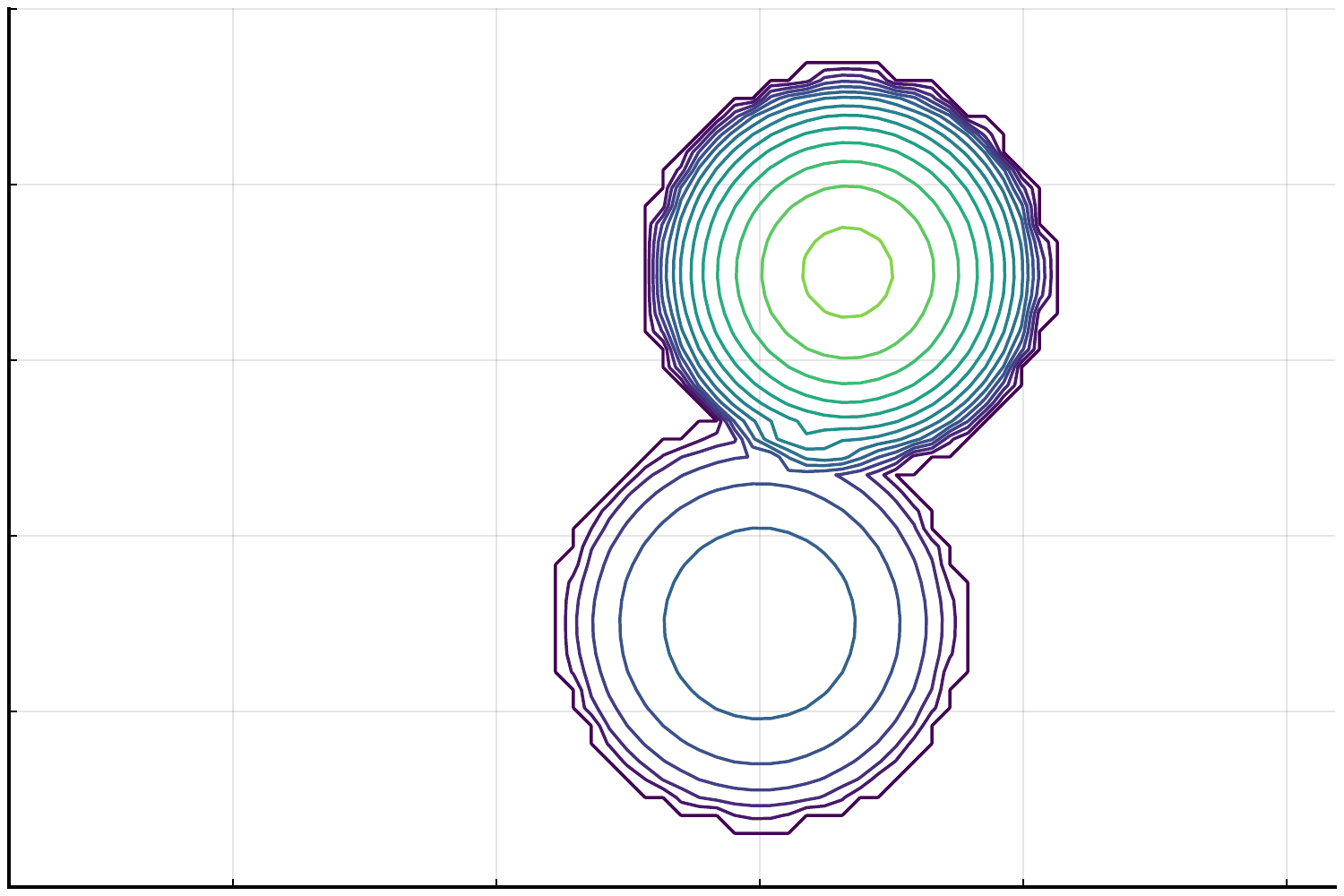} \\
         \includegraphics[width=0.18\textwidth]{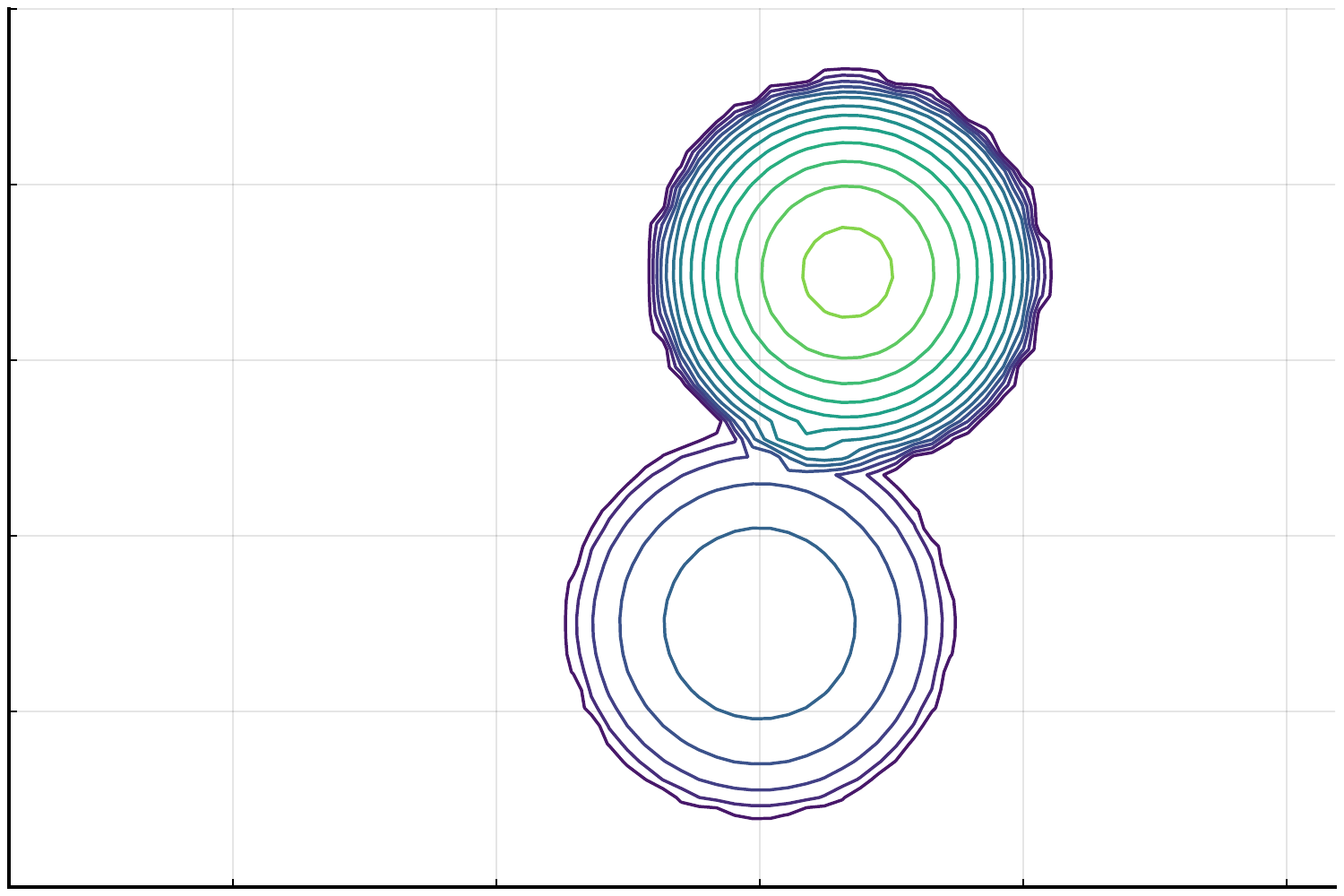}
    \end{tabular}
    } 
    \subfigure[$t=0.25$]{\begin{tabular}{@{}c@{}}
         \includegraphics[width=0.18\textwidth]{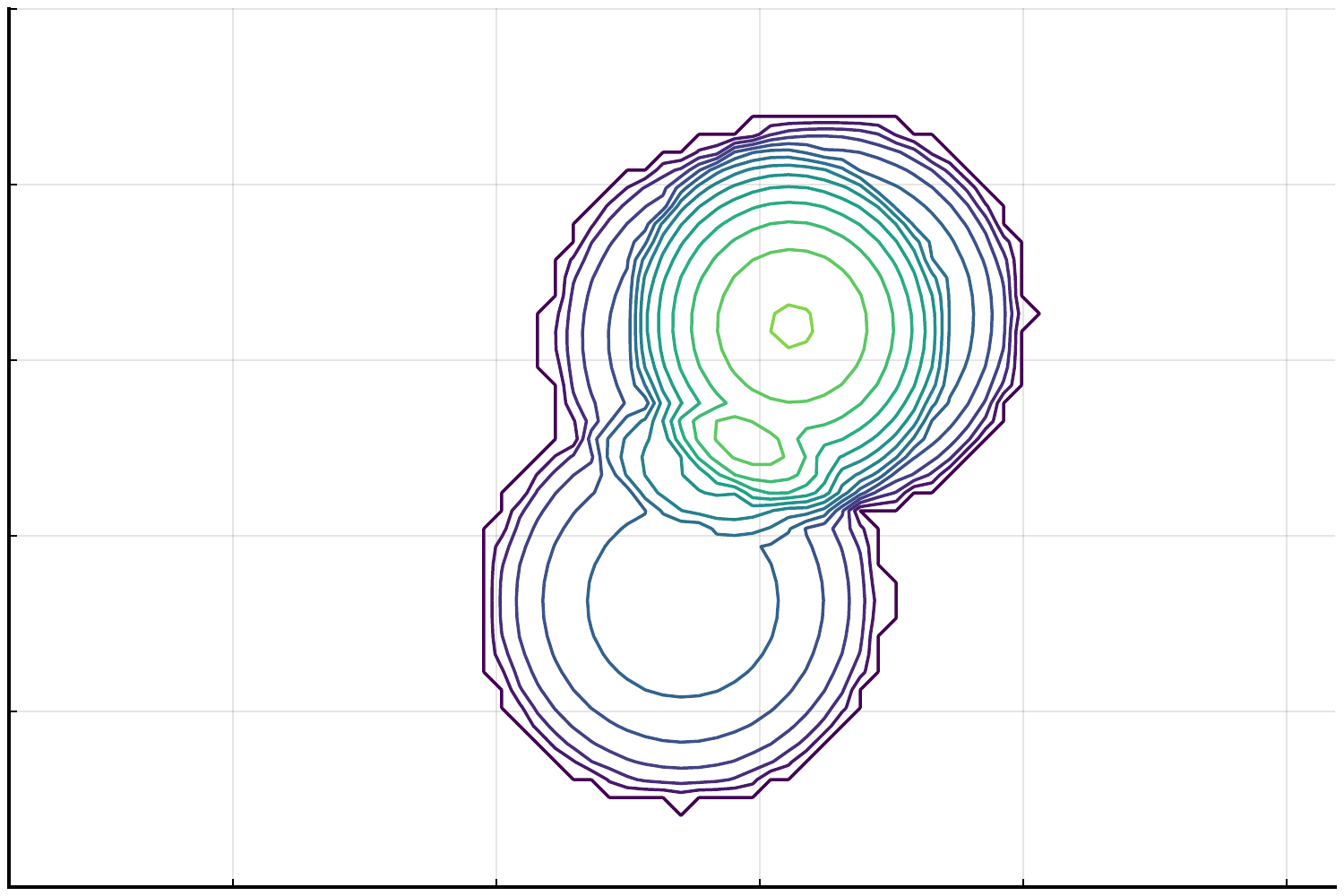} \\
         \includegraphics[width=0.18\textwidth]{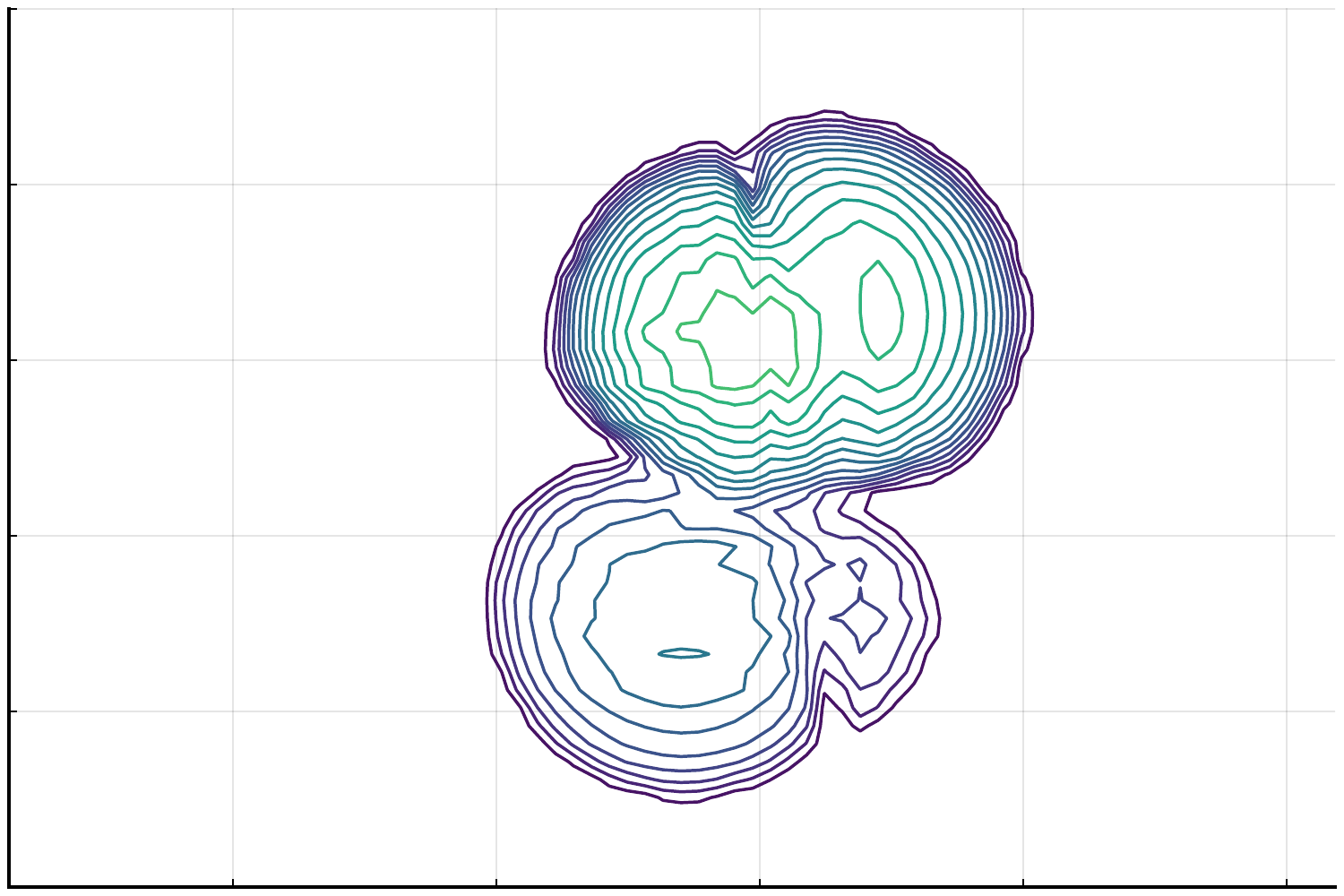}
    \end{tabular}} 
    \subfigure[$t=0.5$]{\begin{tabular}{@{}c@{}}
         \includegraphics[width=0.18\textwidth]{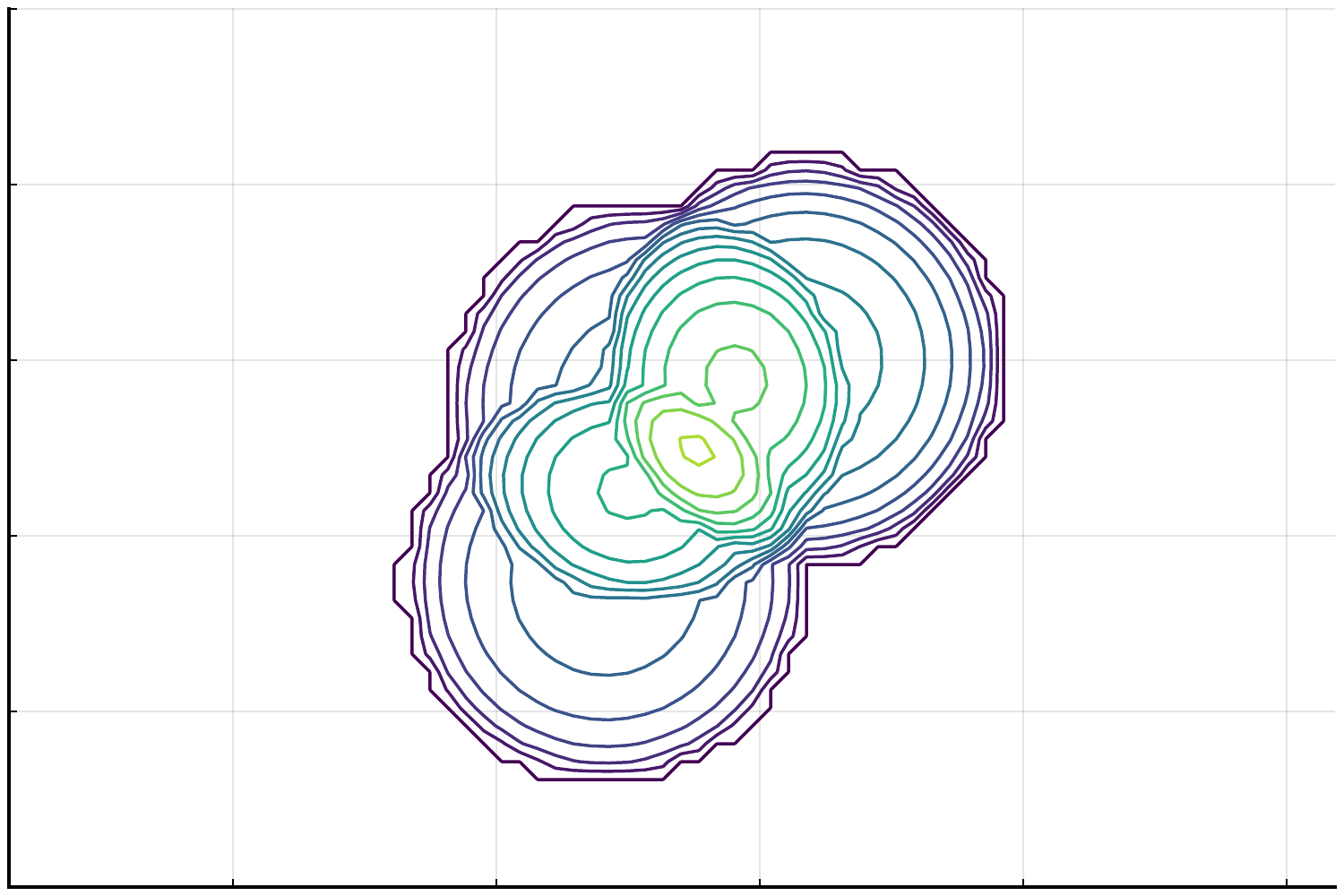} \\
         \includegraphics[width=0.18\textwidth]{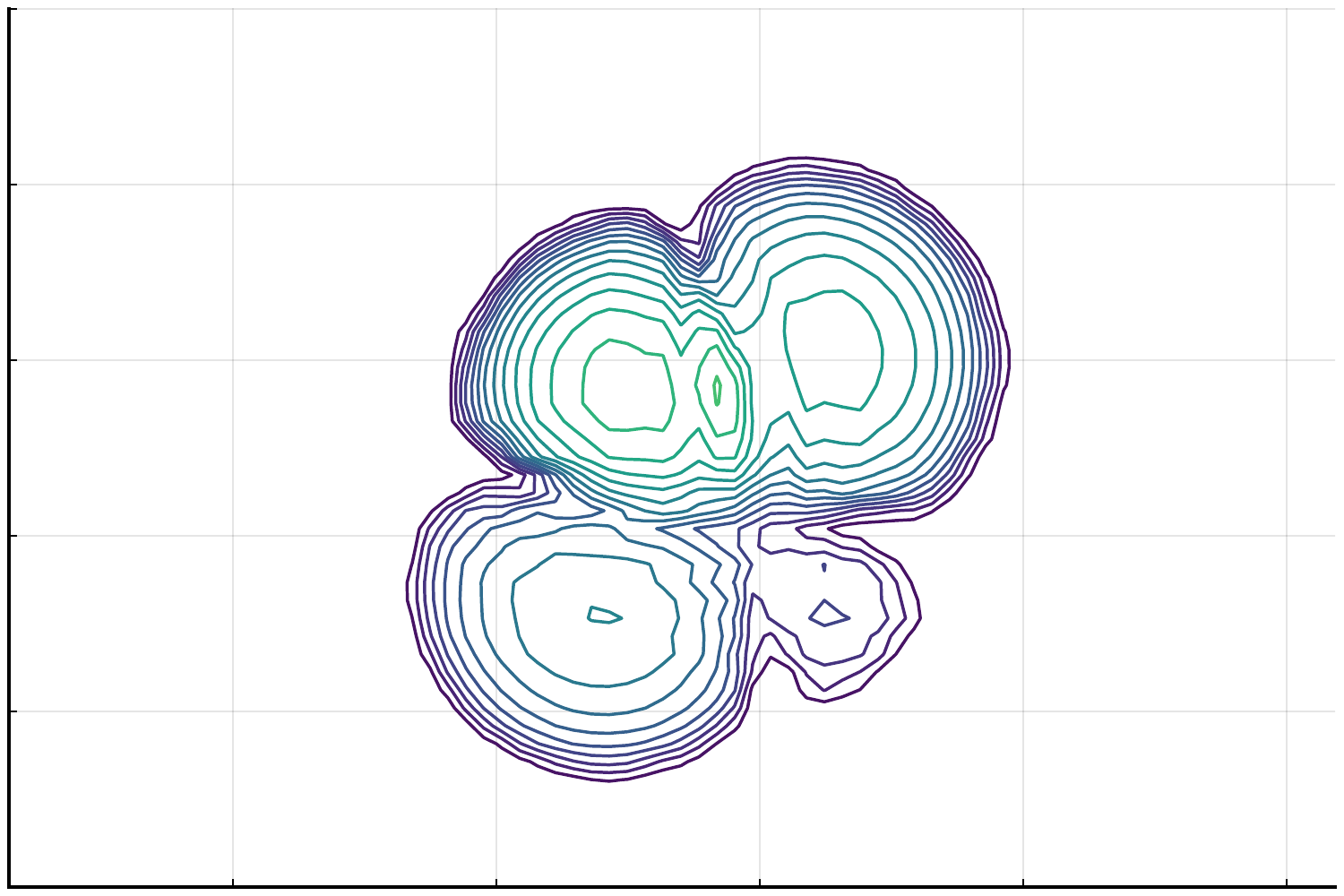}
    \end{tabular}} 
    \subfigure[$t=0.75$]{\begin{tabular}{@{}c@{}}
         \includegraphics[width=0.18\textwidth]{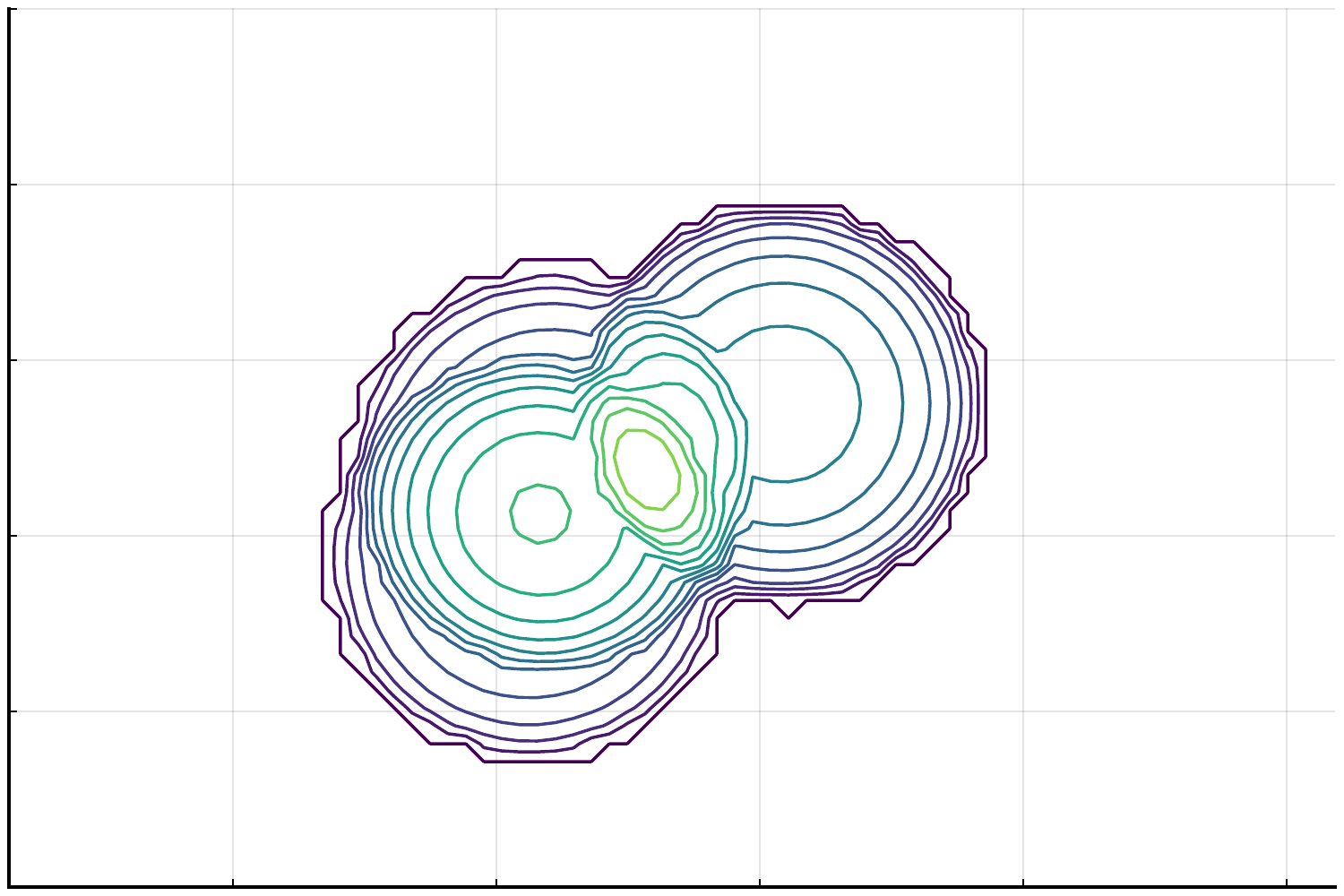} \\
         \includegraphics[width=0.18\textwidth]{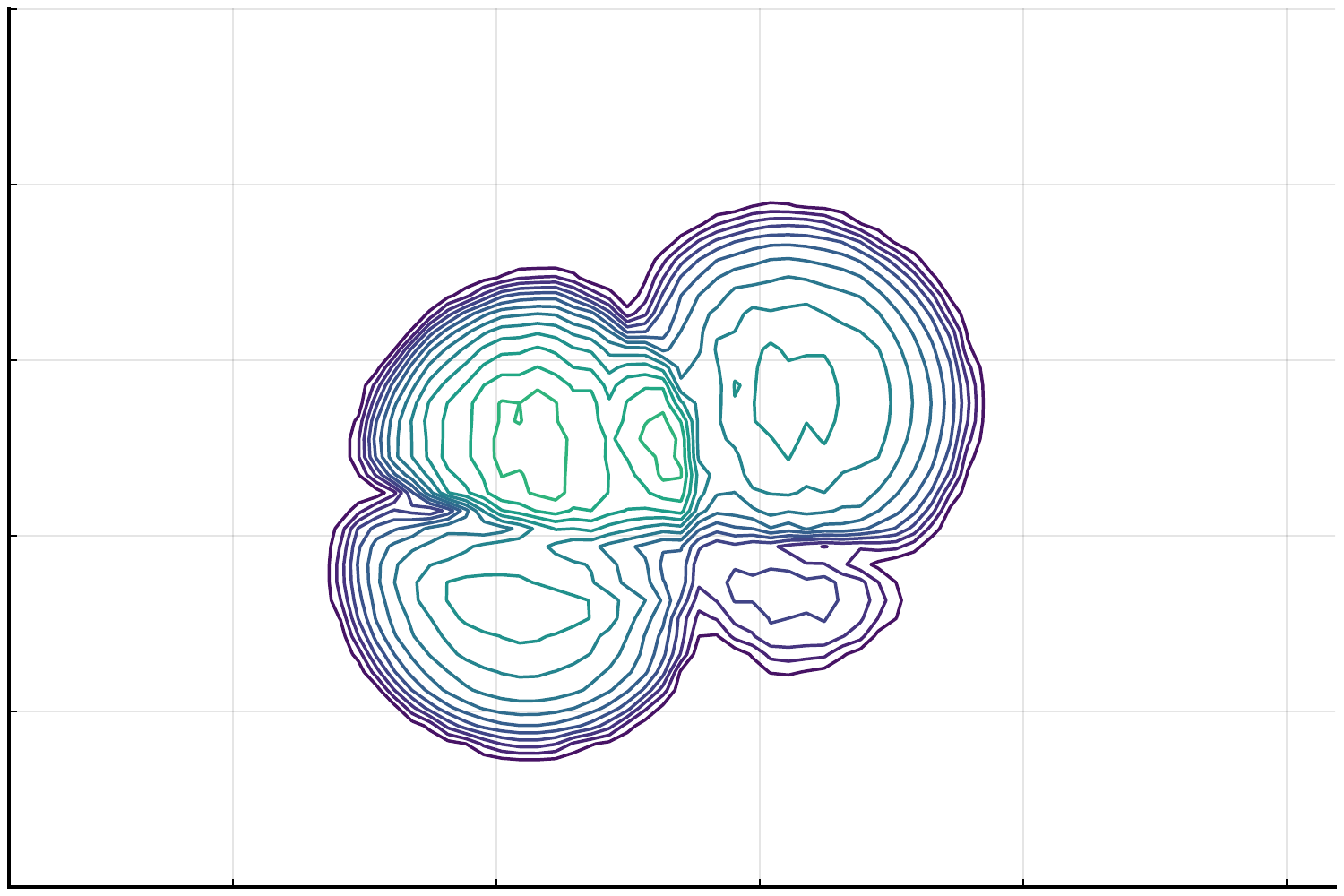}
    \end{tabular}} 
    \subfigure[$t=1$]{\begin{tabular}{@{}c@{}}
         \includegraphics[width=0.18\textwidth]{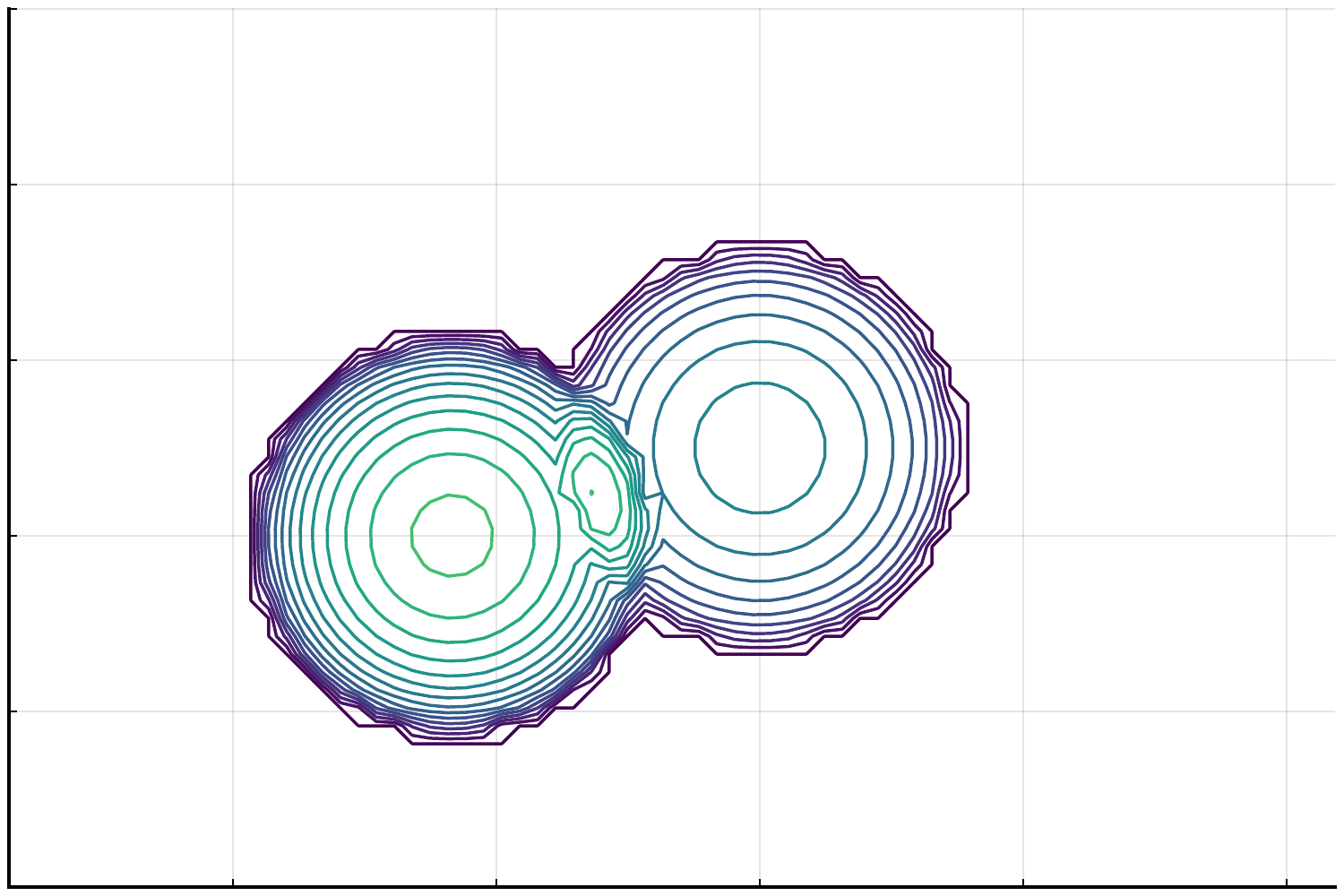} \\
         \includegraphics[width=0.18\textwidth]{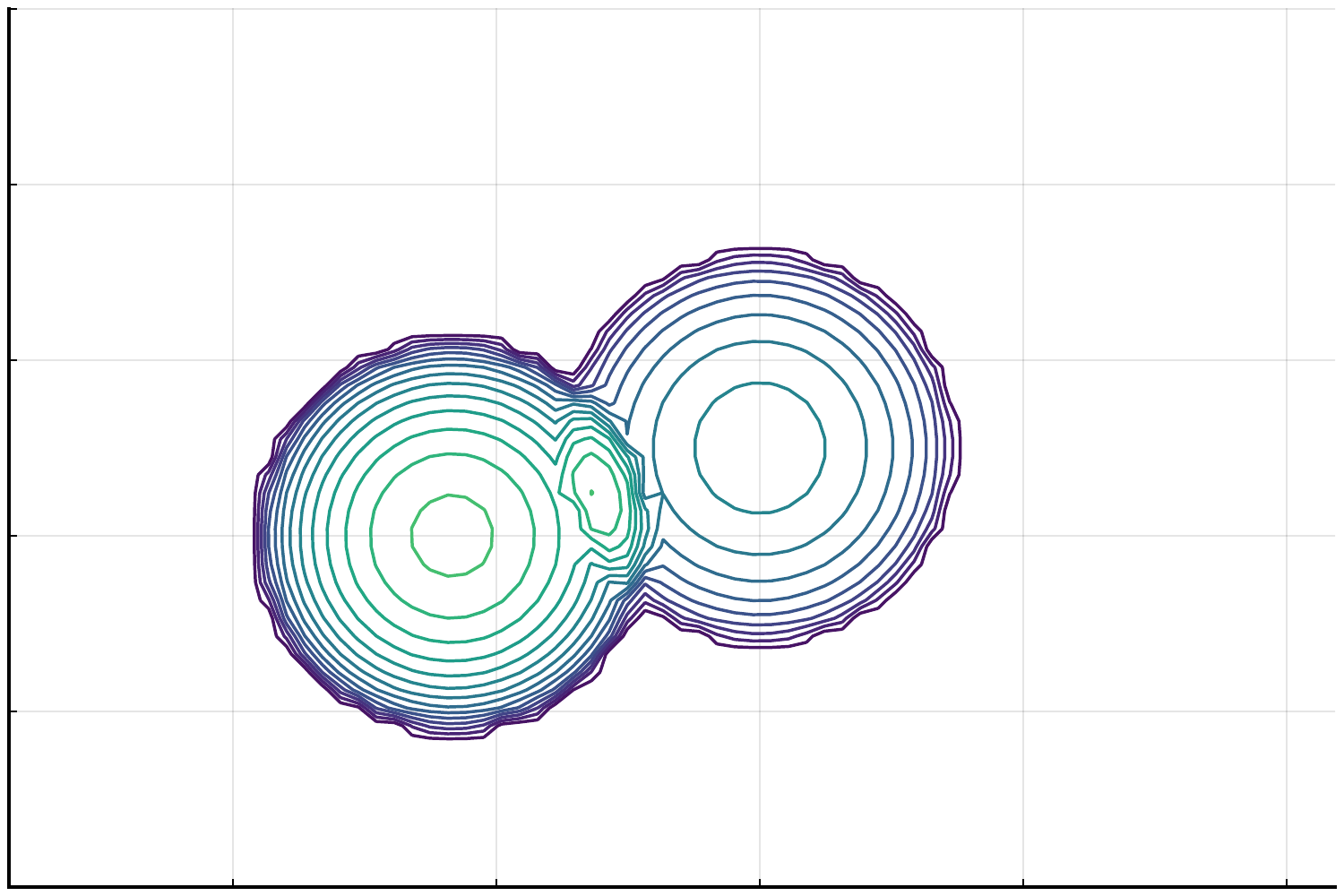}
    \end{tabular}} 
    \caption{Contour plots of $W_{2,\mathcal M}$ (top) and $W_2$ (bottom) barycenters between two mixtures of Wigner semicircle elliptic distributions.}
    \label{fig:semicircle2dmw2_contour}
\end{figure}

\subsubsection{Gamma-distribution-based atoms}

Finally, we consider atoms based on a Gamma distribution in $\R$. Since Gamma distributions are not elliptical distributions, we choose one particular gamma distribution and generate atoms from location-scatters of this particular distribution. For $\alpha,\beta>0$, the probability density function of the gamma distribution with these parameters is
\[
    \forall x\ge 0,\quad f_{\alpha,\beta}(x) =\frac{\beta^\alpha}{\Gamma(\alpha)} x^{\alpha-1} e^{-\beta x},
\]
and has mean $\alpha/\beta$ and variance $\alpha/\beta^2$, so that we choose $\alpha,\beta$ such that $\alpha/\beta^2 = 1,$ and we then define the atoms characterized by their mean $m$ and covariance $\Sigma > 0$ by 
\[
    \forall x\in\R, \quad  g_{m,\Sigma}(x) =
    \left\{
    \begin{array}{cc}
       \frac{1}{\sqrt{\Sigma}} f_{\alpha,\beta}(\Sigma^{-1/2}(x-m+\frac{\alpha}{\beta})),  & x \ge m - \frac{\alpha}{\beta} \\
       0,  &  x < m - \frac{\alpha}{\beta}.
    \end{array}
    \right.
\]

As for the Slater-type case, it is easy to prove identifiability of the mixtures by following the arguments presented in~\cite[Proposition 2]{Delon2020-wk}.

\begin{figure}[h!]
    \centering
    \subfigure[$t=0$]{\includegraphics[width=0.19\textwidth]{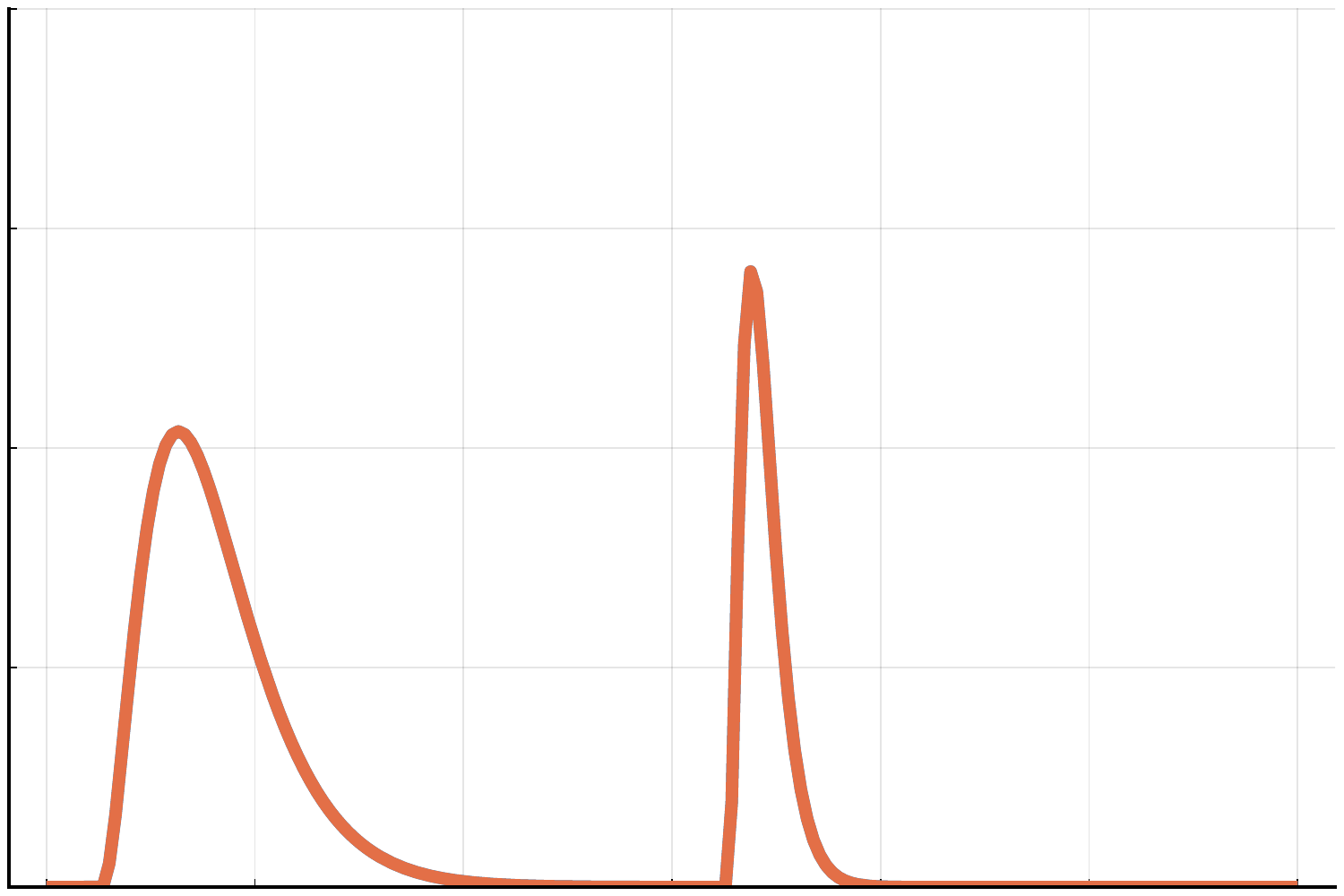}} 
    \subfigure[$t=0.25$]{\includegraphics[width=0.19\textwidth]{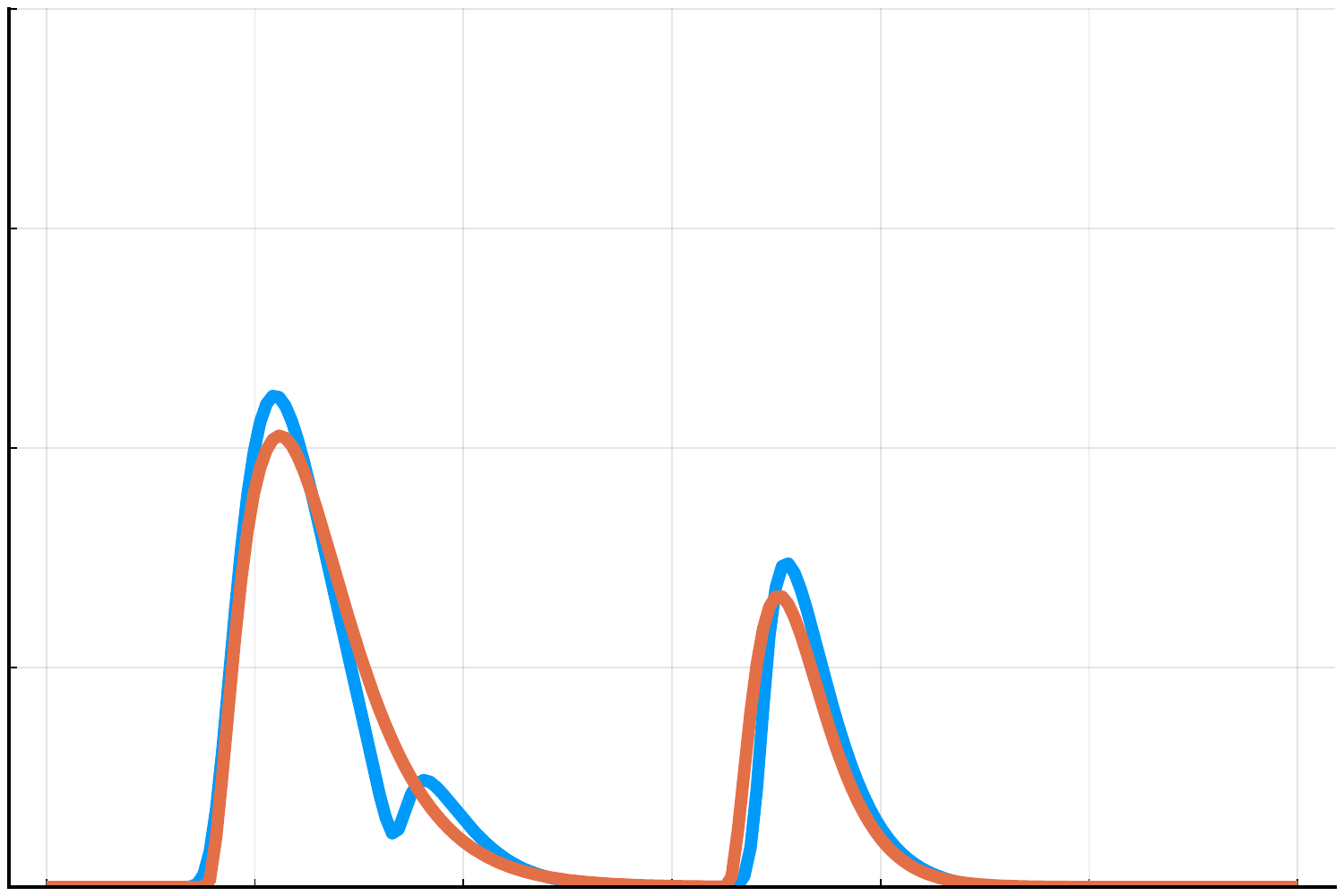}} 
    \subfigure[$t=0.5$]{\includegraphics[width=0.19\textwidth]{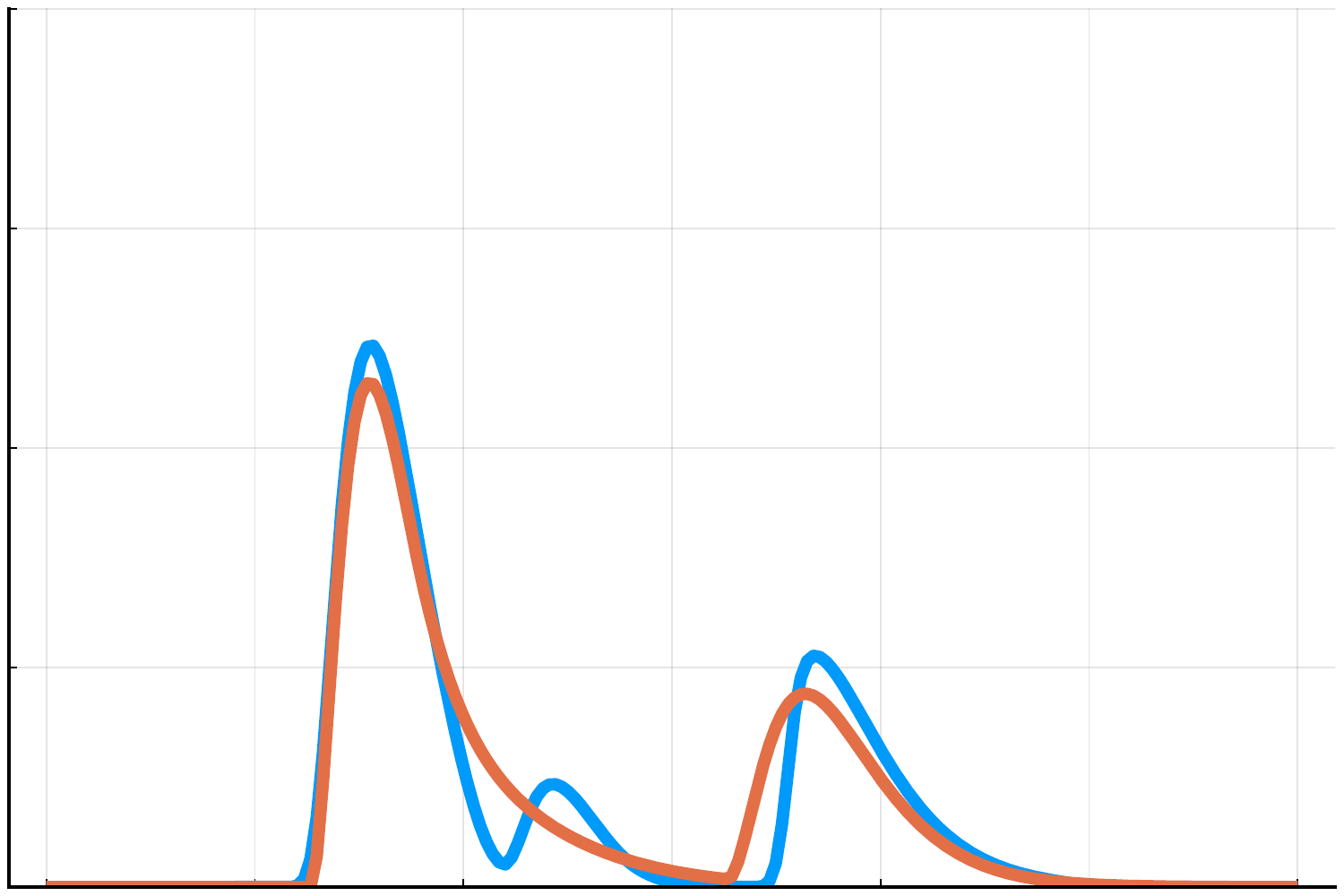}}
    \subfigure[$t=0.75$]{\includegraphics[width=0.19\textwidth]{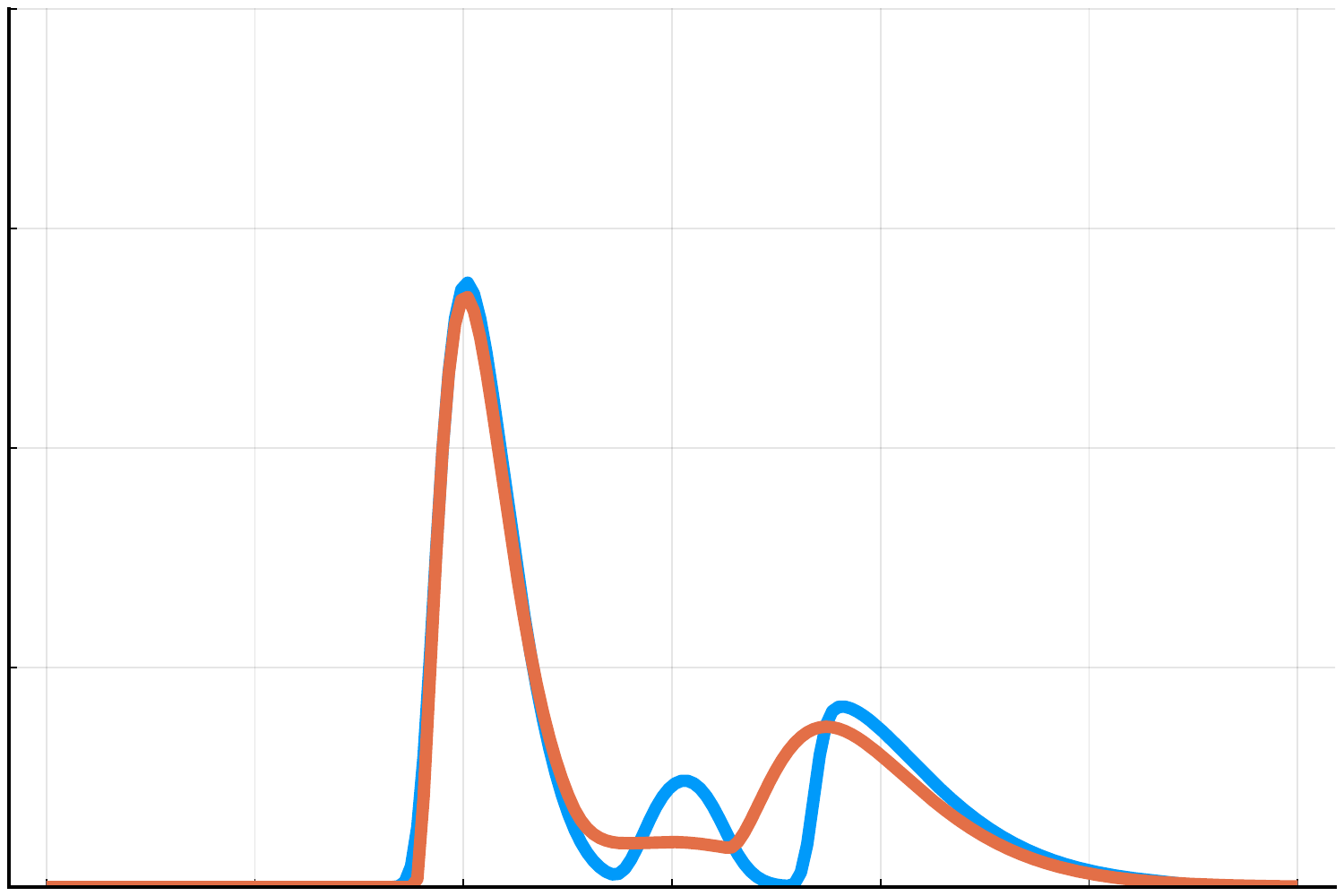}}
    \subfigure[$t=1$]{\includegraphics[width=0.19\textwidth]{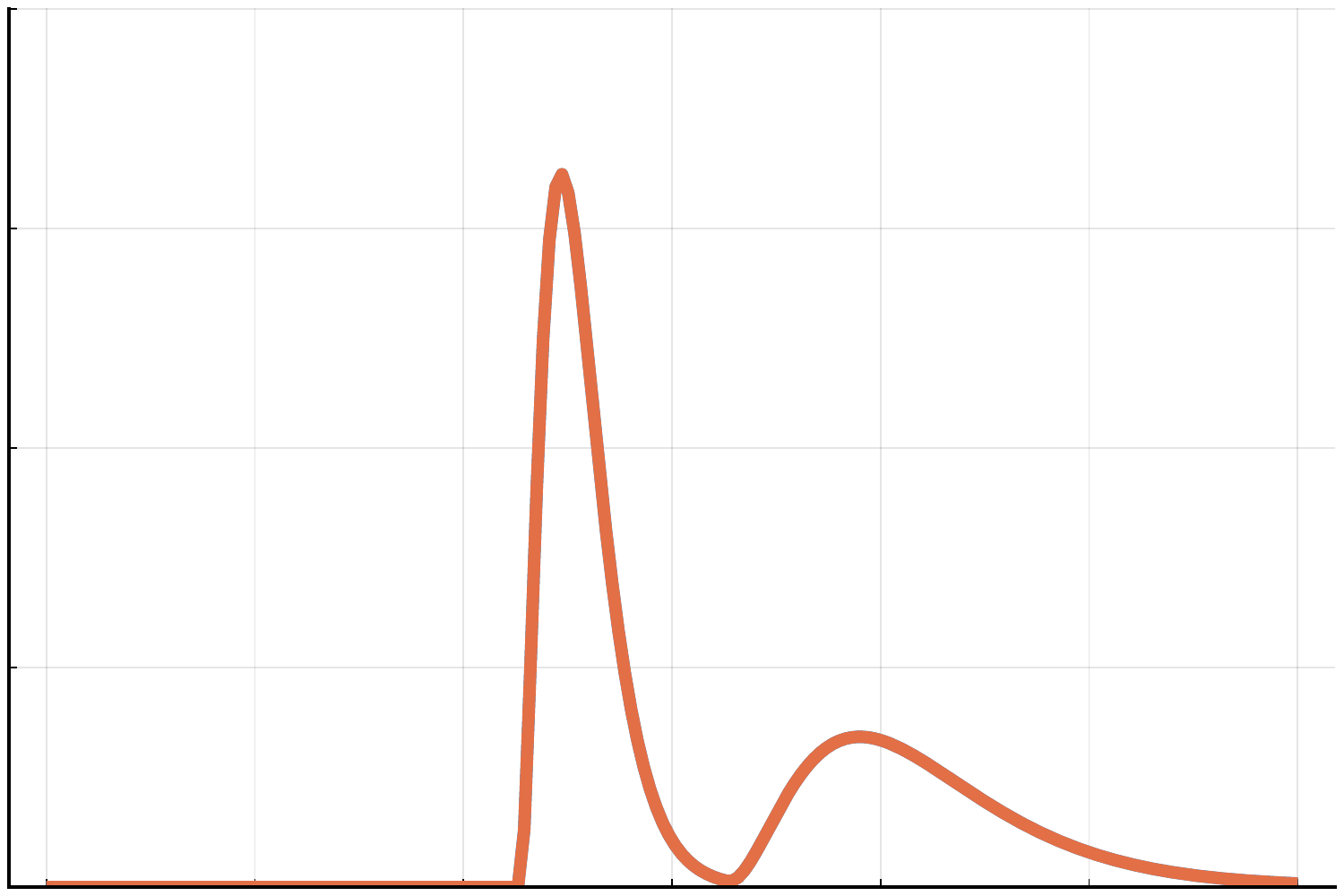}}
    \caption{Wasserstein barycenters between two mixtures of Gamma distributions for the $W_2$ metric (blue) and the $W_{2,\mathcal M}$ metric (red).}
    \label{fig:gamma1d1d}
\end{figure}

In Figure~\ref{fig:gamma1d1d}, we present the $W_2$ and $W_{2,\mathcal M}$ barycenters between two mixtures of two atoms each, where we have taken $\alpha = 3.$, $\beta=9.$ to define the atom having a variance of one. We observe that the $W_{2,\mathcal M}$ barycenter seems smoother and to have a lower middle mass movement than the $W_2$ barycenter.

\section{Symmetry group invariant measures}
\label{sec:5}

In this section, we consider dictionaries of atoms that are defined as sets of symmetric probability measures, i.e. invariant with respect to some transformation, as stated in the following definition. Here, $\Omega$ is a convex open subset (or the closure of a convex open subset) of $\mathbb{R}^d$, $p>1$ and $c:\Omega \times \Omega$ a given metric on $\Omega$.  

\begin{definition}[Invariant measure] 
\label{def:sym}
Let $S : \D \to \D$ be a measurable function from $\D$ to itself. A measure $\mu\in \mathcal{P}(\D)$ is said to be invariant with respect to $S$ if $S\#\mu = \mu$, that is if for every measurable subset $A\subset\D$, 
$\mu\left(S^{-1}(A)\right) = \mu(A).$
\end{definition}

Typical examples we have in mind are measures that are invariant with respect to permutations of the variable ordering, as in the case of marginals of squares of the wavefunction in electronic structure calculations. Another example is given by the set of measures that are invariant with more generic isometries, such as rotations, rigid-body motions, or combination thereof.

\subsection{A few properties of optimal transport plans and Wasserstein bary\-centers}\label{sec:properties}

We start by proving a few results on the symmetries of the optimal transport plans as well as Wasserstein barycenters.
First, we prove that the optimal transport plan in a multi-marginal context is symmetric. Second, we show that the Wasserstein barycenter between several symmetric measures is symmetric as well. Third, we prove some sparsity properties on the support of the optimal transport plan in the case where the symmetry is a reflection.

\medskip

We first study the symmetry of the transport plan between symmetric measures. Let $\mu_1,\dots,\mu_n$ be probability measures in 
$\mathcal{P}_p^c(\D)$ and $C: \D^n \to \mathbb{R}_+\cup\{+\infty\}$ be a multi-dimensional cost function.
We consider here the following multi-marginal optimal transport problem:  find $\gamma \in \Pi(\mu_1,\ldots,\mu_n)$ solution to 
\begin{equation}\label{eq:mmOTsym}
\mathop{\inf}_{\gamma \in \Pi(\mu_1,\ldots,\mu_n)} \int_{\D^n} C\,d\gamma. 
\end{equation}
We have the following proposition. 
\begin{proposition}
\label{prop:sym_gamma}
Let $S: \D \to \D$ be a $\mathcal C^1$-diffeomorphism such that $|{\rm det} \nabla S(x)| = 1$ for all $x\in\D$. We introduce
\begin{equation}\label{eq:defSn}
S_n: \left\{
\begin{array}{ccc}
\D^n &\to& \D^n\\
(x_1,\cdots,x_n) & \mapsto & (S(x_1), \cdots, S(x_n)).\\
\end{array}
\right.
\end{equation}
In addition, let $\mu_1,\dots,\mu_n$ be probability measures in 
$\mathcal{P}(\D)$ that are invariant under $S$. 
Let us assume that $C$ is invariant with respect to $S_n$, that is
\begin{equation}
    \label{eq:as_on_c}
    C(S_n(x_1, \ldots,x_n))= C(S(x_1), \ldots, S(x_n)) = C(x_1,\ldots,x_n), \quad \forall (x_1,\ldots, x_n)\in \D^n,
\end{equation}
and that there exists a unique solution $\gamma$ to (\ref{eq:mmOTsym}). 
Then $\gamma$ is invariant with respect to $S_n$.
\end{proposition}

\begin{proof}
Let $\gamma$ be the optimal transport plan solution to~\eqref{eq:mmOTsym}. Let $\widetilde{\gamma}:= S_n \# \gamma$. Let us prove that $\widetilde{\gamma}= \gamma$. First, it is easy to see that for all $1\leq i \leq n$, $P_i\#\widetilde{\gamma}= S\# \mu_i$. Since $\mu_i$ is invariant under $S$ for all $1\leq i \leq n$, it holds that $\widetilde{\gamma} \in \Pi(\mu_1,\cdots,\mu_n)$.  

Moreover using a change of variables, we obtain from~\eqref{eq:as_on_c} that
\[
    \int_{\D^n}C(x_1,\ldots,x_n)\,d\widetilde{\gamma} =     \int_{\D^n}C(S(x_1),\ldots,S(x_n))\,d\gamma
     = \int_{\D^n}C(x_1,\ldots,x_n)\,d\gamma.
\]
Thus, $\widetilde{\gamma}$ is also an optimal transport plan between $\mu_1,\ldots,\mu_n$ which implies that $\widetilde{\gamma} = \gamma$, hence the desired result.
\end{proof}

We now prove that the $2$-Wasserstein barycenter associated to the euclidean distance between symmetric measures is symmetric. 

\begin{proposition}
\label{prop:5.3}
Let us assume that $\D$ is convex and that for all $x,y\in \D$, $c(x,y)= \|x-y\|$. Let $S: \D \to \D$ be a measurable map such that for all $x,y\in \D$, 
$c(S(x), S(y)) = c(x,y)$ and such that for all $\boldsymbol{t} = (t_i)_{1\leq i \leq n}\in \mathcal L_n$ 
and all $x_1,\ldots,x_n\in \D$, 
$S\left( \sum_{i=1}^n t_i x_i \right) = \sum_{i=1}^n t_i S(x_i)$.
 Let $\mu_1,\dots,\mu_n$ be probability measures in $\mathcal{P}_p^c(\D)$ that are invariant under $S$. Let $\boldsymbol{t} = (t_1,\ldots,t_n)\in \mathcal L_n$ and let us assume that there exists a unique optimal transport plan solution to (\ref{eq:mmOTsym}) with 
 \begin{equation}
     \label{eq:Cdef}
      C(x_1,\ldots,x_n) := \sum_{1\leq i ,j \leq n} t_i t_j c(x_i ,x_j)^p.
 \end{equation}
 Let $\nu$ be the $p$-Wasserstein barycenter between the measures $\mu_1,\dots,\mu_n$ with weigths $\boldsymbol{t}$.
Then $\nu$ is also invariant under $S$. 
\end{proposition}
\begin{proof}
It suffices to show that for all real-valued measurable functions $f:\D \to \R$, there holds
\begin{equation}
    \int_{\D} f(x) \,d\nu(x)=\int_{\D} f(S(x)) \,d\nu(x).
    \label{eq:sym_barycentre_mesurable}
\end{equation}
Let $f$ be a measurable real-valued function defined on $\D$. 
From~\cite{agueh2011barycenters}, there holds $\nu=T\# \gamma$ with $$
T(x_1,\dots,x_n)=\sum_{i=1}^{n}t_ix_i, \quad \forall (x_1,\cdots,x_n)\in \D^n,
$$ 
and $\gamma$ the unique solution to (\ref{eq:mmOTsym}) associated to the cost function $C$ defined in~\eqref{eq:Cdef}.
 
Thus
\begin{equation}
    \label{eq:eq_fT}
    \int_{\D} f(x) \, d\nu(x)=\int_{\D^n} f(T(x_1,\dots,x_n))\,d\gamma(x_1,\dots,x_n),
\end{equation}
and 
\[
\int_{\D} f(S(x)) \, d\nu(x)=\int_{\D^n} f(S(T(x_1,\dots,x_n)))\,d\gamma(x_1,\dots,x_n).
\]
Using the assumption on $S$ and the fact that $\sum_{i=1}^n t_i = 1$, we obtain
$$
S(T(x_1,\dots,x_n)) = T(S(x_1), \cdots, S(x_n)), \quad \forall (x_1,\cdots,x_n)\in \D^n.
$$
As a consequence, noting that the assumptions on $S$ and Proposition~\ref{prop:sym_gamma} show that $\gamma$ is invariant under $S_n$ where $S_n$ is the map defined by~\eqref{eq:defSn}, and using the assumption on the cost function~\eqref{eq:as_on_c},
\begin{align*}
\int_{\D} f(S(x)) \, d\nu(x)
& =\int_{\D^n} f(T(S(x_1),\dots,S(x_n)))\,d\gamma(x_1,\dots,x_n)\\
& = \int_{\D^n} 
 f(T(S_n(x_1,\dots,x_n)))\,d\gamma(x_1,\dots,x_n) \\
& = \int_{\D^n}  f(T(x_1,\dots,x_n))\,d\gamma(x_1,\dots,x_n).
\end{align*}
Using~\eqref{eq:eq_fT} finishes the proof.
\end{proof}

Finally, we prove that when the considered symmetry is a reflection and the domain can be split into two parts such that $S(\D_1) = \D_2$ and $S(\D_2) = \D_1$, the optimal transport plan between symmetric measures is zero on many parts of the domain $\D^n.$ Indeed, the only parts of the domain where the optimal transport plan 
is nonzero are $\D_1^n$ and $\D_2^n$.

\begin{proposition}
Let $\mu_1,\dots,\mu_n$ be probability measures in $\mathcal{P}_p^c(\D)$ invariant under a map $S:\D \to \D$. 
We assume that $S$ is a reflection (i.e. $S$ is an isometry in the sense that $c(S(x),S(y)) = c(x,y)$ for all $x,y\in \D$ and such that $S^2 = Id$) and that there exist $\D_1,\D_2$ such that $\D_1\cap \D_2 = \emptyset$, $\D_1 \cup \D_2 = \D$, and that $S(\D_1) = \D_2$, $S(\D_2) = \D_1$. 
We also assume that for all $(x,y)\in (\D_1\times \D_1) \cup (\D_2\times \D_2)$,
\[
    c(x,y) \le c( S(x), y).
\]
Let $\boldsymbol{t} = (t_1,\ldots,t_n)\in \mathcal L_n$. Let us assume that there exists a unique optimal transport plan $\gamma \in \Pi(\mu_1,\dots,\mu_n)$ solution to (\ref{eq:mmOTsym}) with $C$ defined in~\eqref{eq:Cdef}.
Then, the support of $\gamma$ is included in  $\D_1^n\cup \D_2^n$.
\end{proposition}

\begin{proof}
For any $\bm{i} = (i_1, i_2,\ldots,i_n)\in \{0,1\}^n$, denoting by $\bm{x} = (x_1,\ldots, x_n) \in \Omega^n$, we define the map
\begin{equation}
    \label{eq:Si1in}
    S_{\bm{i}}: \left\{
\begin{array}{ccc}
\D^n & \to & \D^n \\
\bm{x} & \mapsto & (S^{i_1}(x_1), S^{i_2}(x_2), \ldots, S^{i_n}(x_n)),\\
\end{array}
\right.
\end{equation}
so that $S_{i_1,i_2,\ldots,i_n}^2 = {\rm Id}$. We define a non-negative measure $\widetilde{\gamma}$ on $\D^n$ as follows: for any measurable subset $B\subset \D^n$, we define
\begin{equation}
    \label{eq:tildegamma}
    \widetilde{\gamma}(B) = \frac{1}{2}\sum_{\bm{i}\in \{0,1 \}^n} \left(S_{\bm{i}}\# \gamma\right)(B \cap (\D_1^n \cup \D_2^n)).
\end{equation}
By construction, the support of $\tilde \gamma$ is included in $\D_1^n \cup \D_2^n$. Let us prove that $ \tilde \gamma $ is an optimal Wasserstein transport plan between $ \mu_1,\dots, \mu_n $. 
To this aim, we first show that $\tilde\gamma$ has for marginals $ \mu_1,\dots, \mu_n,$ starting with the first marginal, the others being dealt with similarly.
Let $f:\D \to \mathbb{R}$ be a measurable map and define by $F: \D^n \to \mathbb{R}$ the function such that $F(\bm{x}) = f(x_1)$ for all $\bm{x}\in \D^n$. 
Using~\eqref{eq:tildegamma} we obtain
\begin{align*}
    \int_{\D^{n}} f(x_1)\,d\tilde{\gamma}(\bm{x})   =   \int_{\D^{n}} F(\bm{x})\,d\tilde{\gamma}(\bm{x})  = \frac{1}{2}\sum_{\bm{i}\in \{0,1 \}^n} \int_{\D_1^{n} \cup \D_2^n}  F(\bm{x}) \,d(S_{\bm{i}}\#\gamma)(\bm{x}).
\end{align*}
Introducing functions $S^{i_k}$ in the arguments of the function $F$ which do not change the values of $F$ for $k$ from 2 to $n$, and then using a change of variables and writing explicitly the sum on $i_1$, we obtain
\begin{align*}
    \int_{\D^{n}} f(x_1)\,d\tilde{\gamma}(\bm{x})
   & = \;  
 \frac{1}{2}\sum_{\bm{i}\in \{0,1 \}^n} \int_{\D_1^{n} \cup \D_2^n}  F(x_1,S^{i_2}(x_2), \ldots, S^{i_n}(x_n)) \,d(S_{\bm{i}}\#\gamma)(\bm{x})    \\
    & = \;  
 \frac{1}{2}\sum_{ i_2,\cdots,i_n=0}^1 \int_{S_{0, i_2,\ldots,i_n}(\D_1^{n} \cup \D_2^n)}  F(\bm{x}) \,d\gamma(\bm{x})   \\
 & \;  +  \frac{1}{2}\sum_{ i_2,\ldots, i_n = 0}^1 \int_{S_{1, i_2,\ldots,i_n}\left(\D_1^n \cup \D_2^{n}\right)}  F(S(x_1),x_2,\ldots,x_n)\,d\gamma(\bm{x}). 
\end{align*}
Using properties of $S$ leads to
 \begin{align*}
 \int_{\D^{n}} f(x_1)\,d\tilde{\gamma}(\bm{x})
   = \; & 
  \frac{1}{2}\int_{ (\D_1 \cup \D_2) \times \D^{n-1} }  \left[ F(\bm{x})  +  F(S(x_1),x_2,\ldots,x_n) \right]\,d\gamma(\bm{x})  \\
   = \; &  
  \frac{1}{2}\int_{ \D  }  f(x_1)\,d\mu_1(x_1)  + \frac{1}{2}  \int_{\D}  f(S(x_1))\,d\mu_1(x_1).  
\end{align*}
Noting that $f$ is symmetric under $S$, we obtain that the first marginal of $\widetilde{\gamma}$ is $\mu_1$. The proof for the other marginals are similar. Hence $\widetilde{\gamma} \in \Pi(\mu_1, \ldots, \mu_n)$. 
\medskip
Now, let us prove that $\tilde \gamma$ is an optimal Wasserstein transport plan. Indeed, noting that
\[
c(S(x), S(y)) = c(x,y) \; \mbox{ if } (x,y)\in (\D_1 \times \D_1) \cup (\D_2 \times \D_2)
\] 
and 
\[
c(S(x),y) \leq c(x,y) \mbox{ if } (x,y)\in (\D_1 \times \D_2) \cup (\D_2 \times \D_1),
\]
there holds

 \begin{align*}
      \int_{\D^n} \sum_{1\leq k,l \leq n} 
     t_k t_l & c(x_k,x_l)^p\,d\tilde\gamma(\bm{x}) 
      =
    \frac{1}{2} \sum_{\bm{i}\in \{0,1 \}^n} \int_{\D_1^n \cup \D_2^n}
    \sum_{1\leq k,l\leq n }t_k t_l c(x_k,x_l)^p d( S_{\bm{i}}\#\gamma)(\bm{x}) \\
  & =
    \frac{1}{2} \sum_{\bm{i}\in \{0,1 \}^n} \int_{S_{\bm{i}}(\D_1^n \cup \D_2^n)}
    \sum_{1\leq k,l\leq n }t_k t_l c(S^{i_k}(x_k),S^{i_l}(x_l))^p\,d\gamma(\bm{x}) \\
      & \leq
    \frac{1}{2} \sum_{\bm{i}\in \{0,1 \}^n} \int_{S_{\bm{i}}(\D_1^n \cup \D_2^n)}
    \sum_{1\leq k,l\leq n }t_k t_l c(x_k,x_l)^p\,d\gamma(\bm{x}) \\
    & = 
 \int_{\D^n}
    \sum_{1\leq k,l\leq n }t_k t_l c(x_k,x_l)^p\,d\gamma(\bm{x}).
 \end{align*}
Therefore, $\gamma$ being a Wasserstein optimal transport plan between $\mu_1, \ldots, \mu_n$, so is $\tilde\gamma$. The desired result then follows from the uniqueness of the optimal transport plan.
\end{proof}

\subsection{Mixture distance for group invariant measures: general case}

We now introduce dictionaries of symmetric atoms in order to define a mixture distance on symmetric measures.
Let $G$ be a finite or compact group acting on $\mathcal P_2(\Omega)$ through a group action  denoted by $\cdot$. 
 We denote by $H$ the normalised Haar measure on $G$. Note that for finite groups, this measure corresponds to Dirac masses on the elements of the group with equal weight that is the inverse of the cardinal of the group.

Let $\A$ be a set of atoms such that for all $a\in \A$ and all $g\in G$, $g\cdot a \in \A$. We define $\A_{\rm sym}$ as the set of symmetric measures defined from $\A$ as follows:
\begin{equation}
   \label{eq:Asym}
   \A_{\rm sym} = \left\{  \int_G  g \cdot a \; H(dg) , \quad a \in \A \right\}.
\end{equation}
For any $a\in\A$, we denote by
\[
   S(a) := \int_G  g \cdot a \; H(dg).
\]
Note that, for any $ g\in G$ and any $a\in \A$, $S(g \cdot a) = S(a)$.
This leads us to introduce a metric on $\A_{\rm sym}$. 
In order to impose uniqueness up to the group action, we make the following assumption. 

\begin{assumption}
\label{as:equality_group}
    For any $a_1,a_2\in\A$, there holds $S(a_1) = S(a_2)$ if and only if there exists $g\in G$ such that $a_1 = g \cdot a_2$.
\end{assumption}

This assumption is well-adapted to atoms generating identifiable mixtures, but not satisfied in general. Let us give a toy example, taking $\Omega = \left(-\frac{1}{2}, \frac{1}{2}\right)$ and the two-element group $(e, \tilde e)$. Let us consider the following group action: for all $a\in \mathcal P_2(\Omega)$,

\begin{equation}
    \label{eq:parity_group}
     e \cdot a = a, \quad  \tilde e \cdot a = (-{\rm Id} )\# a.
\end{equation}
If the dictionary of atoms $\A$ is chosen so that (i) for all $a\in \A$, $ (-{\rm Id} )\# a \in \A$ and (ii) the set mixtures $\M(\A)$ is identifiable, it is easy to show that Assumption~\ref{as:equality_group} is satisfied. However, taking $a_0(\,dx) = 2 \mathbb{1}_{[-1/2,0]}(x)\,dx$ and $a_1(\,dx) = \frac{2}{1+e^{-x}}\,dx$, there holds $\frac{1}{2}a_0 +\frac{1}{2} (-{\rm Id} )\# a_0 = \frac{1}{2}a_1 + \frac{1}{2}(-{\rm Id} )\# a_1$, while it does not hold that $a_1 = a_0$ or $a_1 = (-{\rm Id} )\# a_0$.

A metric on the set of symmetric measures $\A_{\rm sym}$ is defined as follows.

\begin{proposition}
   \label{prop:sym_metric}
   Let $d:\A\times\A\rightarrow \R^+$ be a metric such that for all $a_0,a_1\in \A$ and all $g\in G$, 
   \begin{equation}
      \label{eq:d_sym_prop}
      d(g \cdot a_0,g \cdot a_1) = d(a_0,a_1).
   \end{equation}
Then, the map $\bar d:\Asym \times \Asym \rightarrow \R^+$ defined by
\begin{equation}
    \label{eq:symm_distance}
        \forall a_0, a_1\in\A,\quad \bar d(\bar a_0,\bar a_1) = \inf_{g \in G} d( a_0, g \cdot a_1), \quad \text{ where } \bar a_0 = S(a_0), \; 
    \bar a_1 = S(a_1),
\end{equation}
is a metric on $\A_{\rm sym}$.
\end{proposition}

\begin{remark}
   Note that the definition given in Proposition~\ref{prop:sym_metric} is equivalent to 
   \[
    \forall \bar a_0,\bar a_1\in\Asym,\quad \bar d(\bar a_0,\bar a_1) = \inf_{g_0,g_1 \in G} d(g_0 \cdot a_0, g_1 \cdot a_1), \;\; \text{ where } \bar a_0 = S(a_0), \; 
    \bar a_1 = S(a_1),
\]
due to~\eqref{eq:d_sym_prop}.
\end{remark}

\begin{proof}
   First, $\bar d$ is clearly symmetric. 
   Second, if $\bar d(\bar a_0,\bar a_1)=0$, there exists $g\in G$ such that $d(a_0,g \cdot a_1)=0$ and so $a_0 = g \cdot a_1$. 
   Therefore, $S(a_0) = S( g \cdot a_1) = S(a_1)$, i.e. $\bar a_0 = \bar a_1$.
   Third, we prove the triangle inequality. 
   Let $a_0, a_1, a_2\in \A$ so that $\bar a_0, \bar a_1,\bar a_2\in \A_{\rm sym}$. It then holds that 
    \[
          \bar d(\bar a_0, \bar a_2) =
      \inf_{g \in G} d( a_0, g \cdot a_2) 
       \le \inf_{g \in G} [d( a_0, a_1) + d( a_1, g \cdot a_2)] 
       \leq \bar d(\bar a_0, \bar a_1) + \bar d(\bar a_1, \bar a_2).
    \]
   Hence $\bar d$ is a metric.
\end{proof}

\begin{proposition}
   $\A_{\rm sym}$ equipped with the metric $\bar d$ is a geodesic space.
\end{proposition}

\begin{proof}
     The proof is similar to that of Proposition~\ref{eq:geodesic_space_mixture}. 
    We consider paths $(\rho_t)_{t\in[0,1]}$ with $\rho_t \in \A_{\rm sym}$ for all $t\in [0,1]$, we define the length of the path relative to the $\bar d$ metric as in~\eqref{eq:path_length},
    and we show that given any two points $\bar a_0,\bar a_1 \in \A_{\rm sym},$ there exists a path between them the length of which equals the distance $\bar d(\bar a_0, \bar a_1)$.
   We first show using the triangle inequality  that ${\rm Len}_{\bar d}(\rho) \ge \bar d(\bar a_0,\bar a_1)$. 
   The equality is shown by defining $\bar g = \mathop{\rm argmin}_{g\in G}d(a_0,g\cdot a_1),$ where $\bar a_0 = S(a_0), a_0\in \A$ and $\bar a_1 = S(a_1), a_1\in \A$, and taking $(\mu_t)_{t\in [0,1]}$ be a constant speed geodesic between $a_0$ and $\bar g \cdot a_1$. Noting that for all $0\le s,t \le 1,$
\[
   d(\mu_t,\mu_s) = |t-s| d(a_0,\bar g \cdot a_1).
\]
we define $\bar \mu_t = S(\mu_t)$ to obtain that for any  $0\le s,t \le 1,$
\[
   \bar d(\bar \mu_t,\bar \mu_s) = \inf_{\widetilde g \in G} d(\mu_t,\widetilde g\cdot \mu_s) 
    \le d(\mu_t,\mu_s) 
    \le |t-s| d(a_0, \bar g \cdot a_1) 
   = |t-s| \bar d(\bar a_0, \bar a_1),
\]
    from which we deduce that 
   $
   {\rm Len}_{\bar d}((\mu_t)_{t\in[0,1]}) = \bar d(\bar a_0, \bar a_1).
   $
   We conclude that $\A_{\rm sym}$ equipped with $\bar d$ is a geodesic space.
\end{proof}

Having now proved that $(\A_{\rm sym},\bar d)$ is a geodesic space, we can now express barycenters in terms of geodesics of atoms.

\begin{corollary}
   Let $\bar a_0=S(a_0)$ and $\bar a_1=S(a_1)$ be two elements of $\A_{\rm sym}$. 
   Let $\bar g = \mathop{\rm argmin}_{g\in G}d(a_0,g\cdot a_1).$
   The barycenters between $\bar a_0$ and $\bar a_1$ belong to $\A_{\rm sym}$ and can be written as 
   \[
      \forall t\in [0,1],\quad \bar a_t = S(a_t),
   \]
   where $(a_t)_{t\in[0,1]}$ is a constant speed geodesic between $a_0$ and $\bar g\cdot a_1$.
\end{corollary}

Note that if Assumption~\ref{as:equality_group} is satisfied, it is easy to show that the barycenter between $\bar a_0$ and $\bar a_1$ does not depend on the choice of $a_0$ and $a_1$.

\subsection{Wasserstein case}

In this section, we focus on the particular case where $d$ is the Wasserstein distance defined in~\eqref{eq:Wpc}, and where the group action of $G$ on $\mathcal P^c_p(\Omega)$ is inherited from a group action of $G$ onto $\Omega$ as follows: for all $g\in G$, $g\cdot a  = a \# T_g$ where $T_g : \Omega \to \Omega$ is defined by $T_g(x) = g \cdot x$ for all $x\in \Omega$. For all $k\in \mathbb{N}^*$, we define $T^k_g: \Omega^k \to \Omega^k$ as the map such that
$T_g^k(x_1,\cdots,x_k) = (g\cdot x_1, \cdots, g\cdot x_k)$ for all $x_1,\cdots,x_k\in \Omega$. Using a slight abuse of notation, for all $k\in \mathbb{N}^*$, and for all $\gamma \in \mathcal P_p(\Omega^k)$, we denote by $g\cdot \gamma = \gamma \# T_g^k$. 
Then there holds
   \begin{equation}\label{eq:Wass-sym}
      \bar d(\bar a_0, \bar a_1) = \left[ \inf_{g_0,g_1\in G}
      \inf_{\gamma\in\Pi(g_0 \cdot a_0,g_1 \cdot a_1)} 
      \int c(x,y)^p d\gamma(x,y) \right]^{1/p}.
   \end{equation}
By analogy, one can define the following symmetrized optimal transport plan as 
\[
   \bar \gamma = \int_{g \in G} g\cdot \gamma\, H(dg).
\]
Note that this is not equivalent to the true Wasserstein transport plan. 
Indeed,  the transport plans are in general different, as shown in Figure~\ref{fig:otplans} for two densities being even hence symmetric with respect to the group action defined in~\eqref{eq:parity_group}. 
This can be easily understood. For the true Wasserstein distance, the optimal transport plan satisfies the monotone rearrangment property, while in the case of the symmetrized distance, the atom is first transported to the closest one, and then the transport plan is symmetrized.
However, this allows for a simple generalization to the multi-marginal problem.

\begin{figure}
    \centering
    \subfigure[two even densities]{\includegraphics[width=0.25\textwidth]{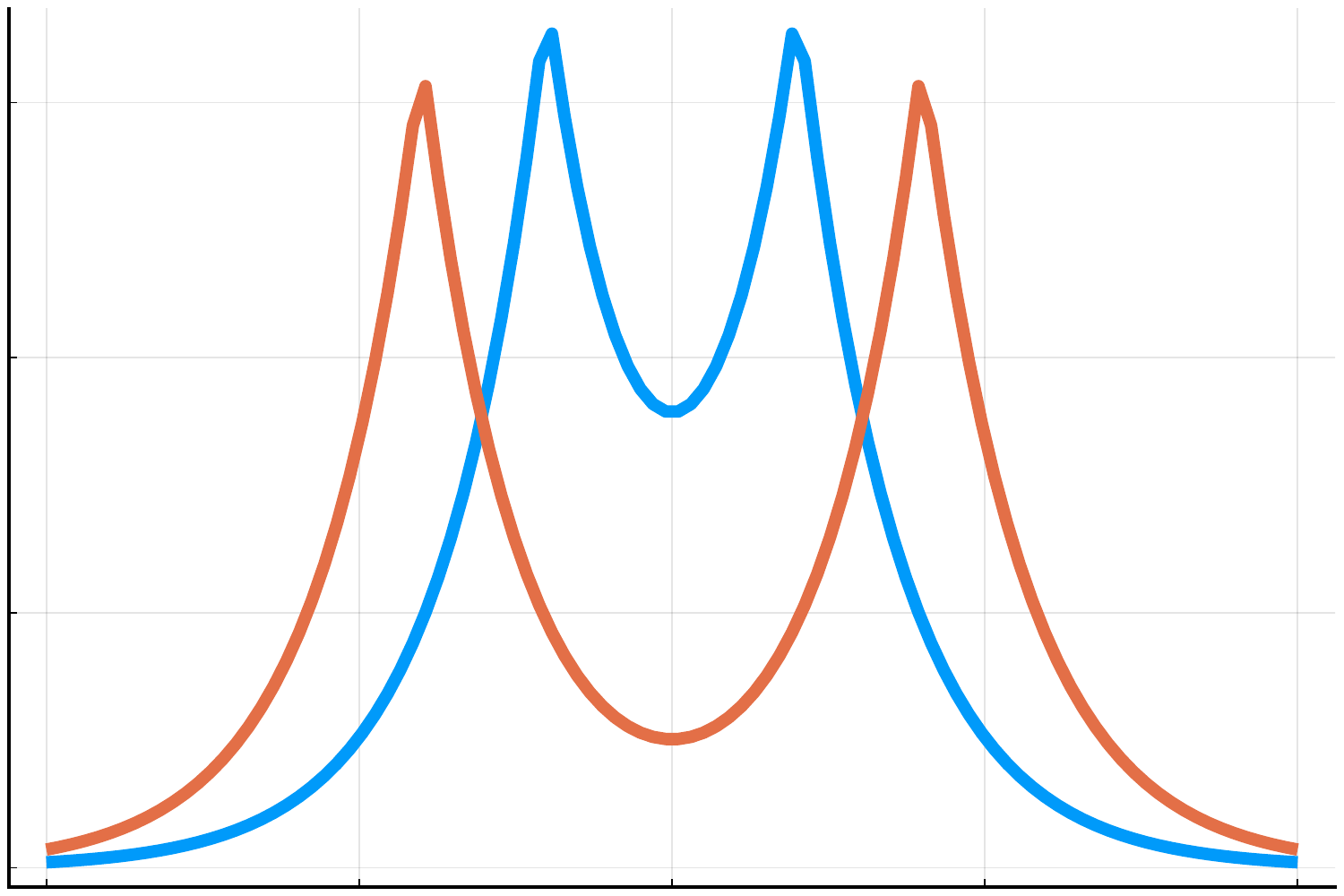}} 
    \subfigure[$W_2$ optimal transport plan]{\includegraphics[width=0.25\textwidth]{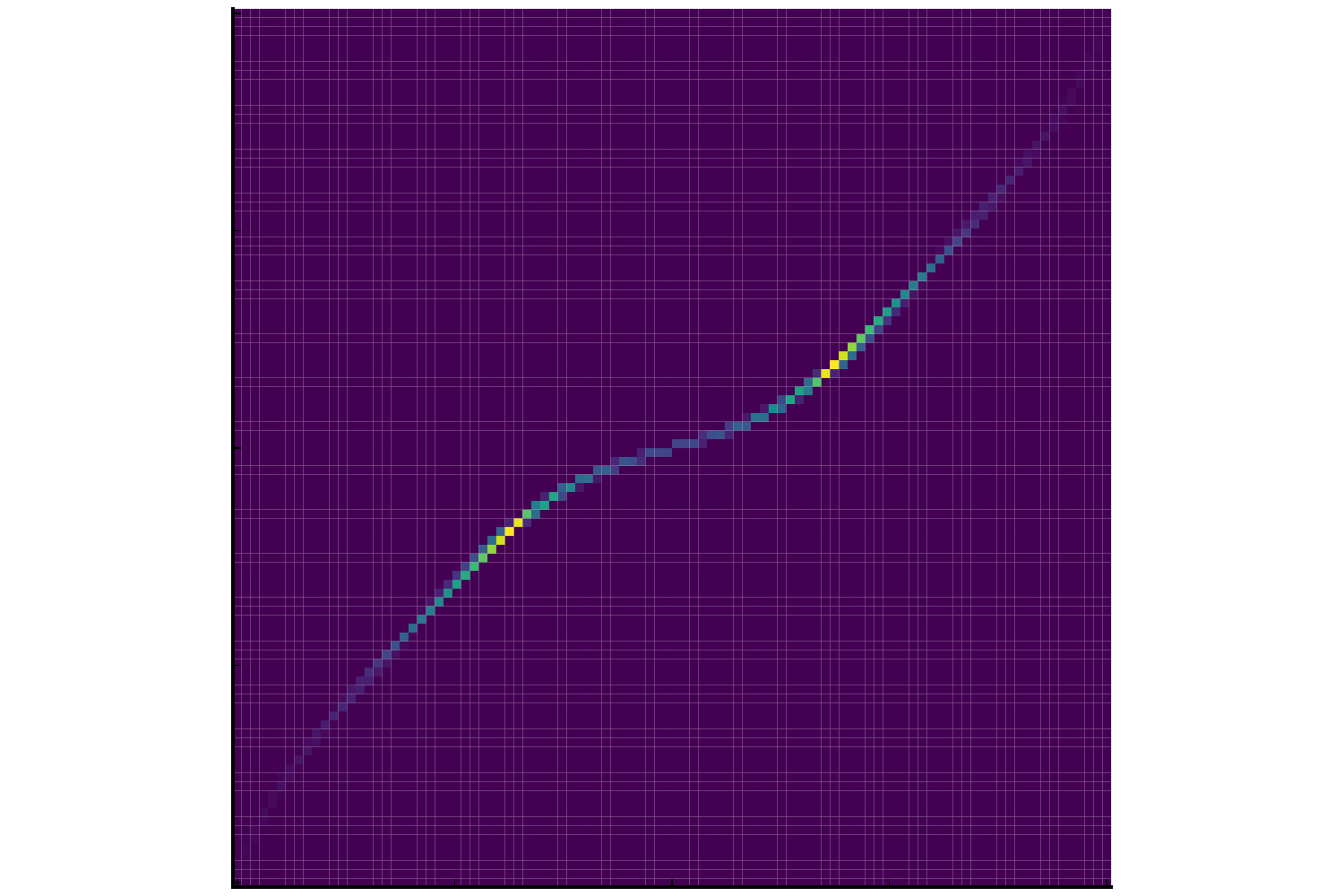}} 
    \subfigure[$W_{2,\mathcal M}$  optimal transport plan]{\includegraphics[width=0.25\textwidth]{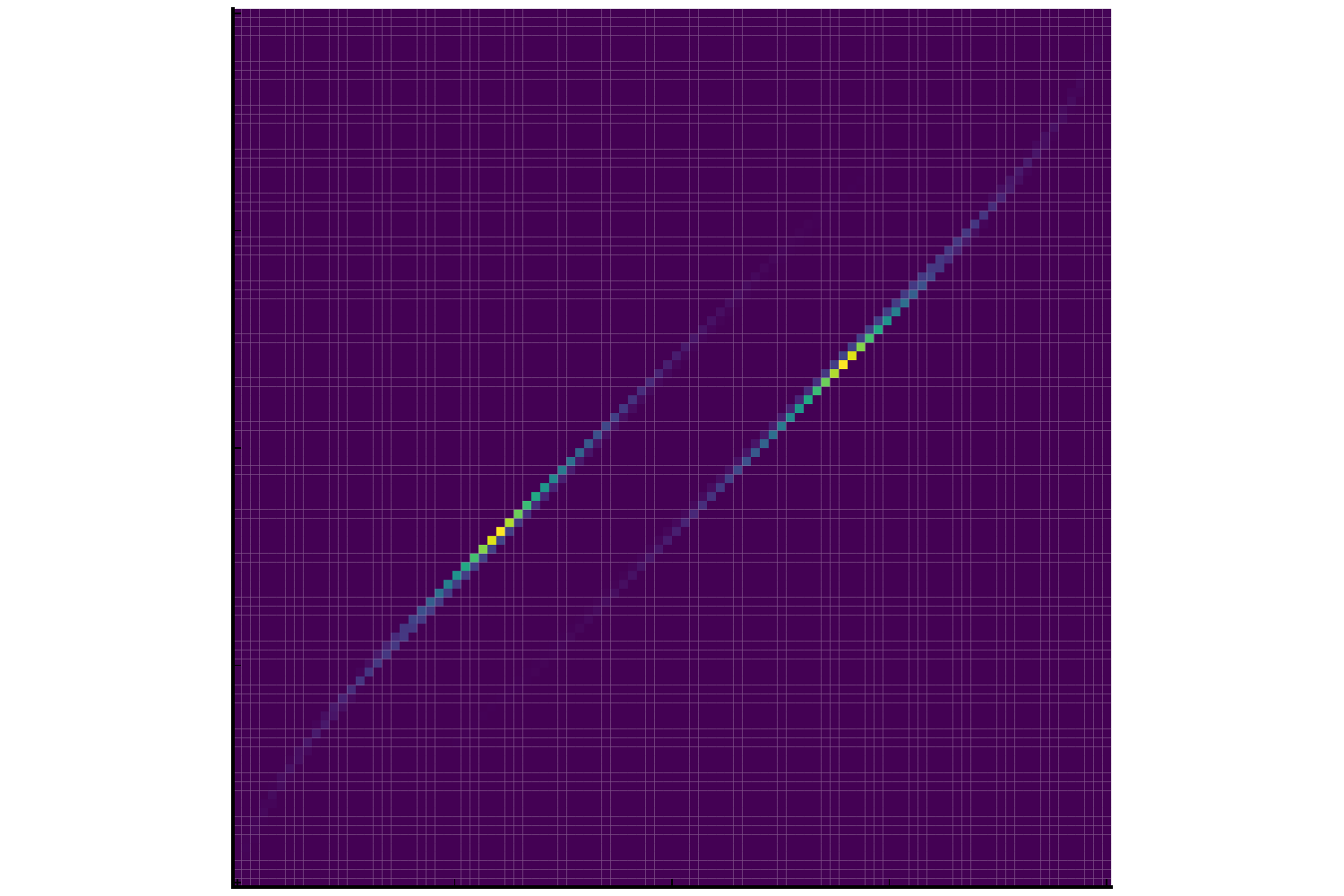}} 
    \caption{Comparison between $W_2$ and $W_{2,\mathcal M}$ transport plans between two even densities}
    \label{fig:otplans}
\end{figure}

\begin{definition}[Symmetric multi-marginal problem]
   \label{def:multimar_groups}
   Let $Q\in \N^*$. 
   Let $\bm{t} := (t_q)_{1\le q\le Q}\in \mathcal L_Q$.
   Let $a_1, \ldots, a_Q \in\A$ and $\bar a_1 = S(a_1), \ldots, \bar a_Q = S(a_Q)$ for $\A$ a set of atoms. We define the symmetric multi-marginal transport problem by
   \begin{equation}
      \label{eq:MMsym}
      \bar d_Q^{\bm{t}}(\bar a_1, \ldots, \bar a_Q)
      = \left( 
        \inf_{g_1,\ldots,g_Q\in G} \inf_{\gamma\in \Pi(g_1\cdot a_1,\ldots, g_Q\cdot  a_Q)} s( x_1,\ldots, x_Q) 
        d\gamma(x_1,\ldots, x_Q)
      \right)^{1/p},
   \end{equation} 
   with 
   \[
      s(x_1,\ldots,x_Q) = \frac12 \sum_{q=1}^Q \sum_{q'=1}^Q 
      t_q t_{q'} c(x_q,x_{q'})^p.
   \]
   Denoting by $\Gamma^{\bm{t}}_{\bar a_1,\ldots, \bar a_Q}$ the set of minimizers of~\eqref{eq:MMsym}, for all $\gamma \in \Gamma^{\bm{t}}_{\bar a_1,\ldots, \bar a_Q}$, we define the associated symmetrized optimal transport plan as 
   \[
      \bar \gamma = \int_{g \in G}
      g\cdot \gamma \; H(dg) .
   \]
   Finally, assuming that Wasserstein barycenters between atoms in $\A$ also belong to $\A$, the barycenters of $(\bar a_1,\cdots,\bar a_Q)$ with barycentric weights $\bm{t}$ can be written as 
      \begin{equation}
         \label{eq:barycenter_symm}
         \overline{\rm bar}_{\bm{t}}(\bar a_1,\ldots, \bar a_Q) =
         S({\rm bar}_{\bm{t}}(\bar g_1\cdot a_1, \ldots, \bar g_Q\cdot a_Q) )
      \end{equation}
      where $(\bar g_1,\ldots, \bar g_Q) \in G^Q$ is such that  
      \[
         (\bar g_1,\ldots \bar g_Q) \in \mathop{\rm arginf }_{g_1,\ldots,g_Q\in G} \; \;
         \inf_{\gamma\in \Pi(g_1\cdot a_1,\ldots, g_Q\cdot  a_Q)} s( x_1,\ldots, x_Q) 
        d\gamma(x_1,\ldots, x_Q).
      \]
\end{definition}

With this definition, we show that the multi-marginal problem is equivalent to the symmetric barycentric problem.

\begin{proposition}
    The infimum in~\eqref{eq:MMsym} matches the following barycentric mi\-nimization problem 
    \begin{equation}
    \label{eq:multimarg_bar_symm}
           \bar d_Q^{\bm{t}}(\bar a_1, \ldots, \bar a_Q)^p =  \inf_{\bar a\in \A_{\rm sym}} \; \sum_{q=1}^Q t_q \bar d(\bar a, \bar a_q)^p.
    \end{equation}
\end{proposition}

\begin{proof}
    Starting from the right handside of~\eqref{eq:multimarg_bar_symm}, and using the distance definition~\eqref{eq:symm_distance}, we obtain
    \begin{align*}
        \inf_{\bar a\in \A_{\rm sym}} \sum_{q=1}^Q t_q \bar d(\bar a, \bar a_q)^p & = \hspace{-.1cm}
        \inf_{a\in \A} \sum_{q=1}^Q t_q [\inf_{g_q\in G} d(a, g_q\cdot a_q)]^p 
         = \hspace{-.28cm} \inf_{g_1,\ldots, g_Q\in G} \inf_{a\in \A} \sum_{q=1}^Q t_q  d(a, g_q\cdot a_q)^p.
    \end{align*}
    Then using the equivalence between the multi-marginal problem and the barycentric problem on the set of atoms $\A$, which is well-posed since the atoms are stable under Wasserstein barycenter proves~\eqref{eq:multimarg_bar_symm}.
\end{proof}

\subsection{Examples}

We now provide a few examples in dimension $d=1,2$ for group actions defined over permutation groups and the rotation group $SO(2)$. We also present an example closely related to quantum chemistry applications.
For simplicity, we consider the case $\D = \mathbb{R}^d$, $p=2$ and $c(x,y)= \|x-y\|$. 
All the numerical results presented below are performed with the same setting as in Section~\ref{seq:numerics} for the $W_2$ barycenters. The $W_{2,\mathcal M}$ barycenters are computed following Definition~\eqref{eq:barycenter_symm}. In terms of computational cost, let us remark that the computational cost of the mixture barycenter highly depends on the cost of computing the barycenter between two atoms. This cost namely depends on the cardinal of the group if finite, or the complexity of solving~\eqref{eq:symm_distance} for infinite groups. However, in the tests presented below, the cardinal of the considered finite groups, which are the permutation groups, is only two, therefore the computation of the mixture barycenters stays orders of magnitude cheaper than the computation of the $W_2$ barycenters.

\subsubsection{Parity group}\label{sec:parity}

We first consider the case $d=1$, $\D= \mathbb{R}^d$ and where $G$ is defined as the two-element group $(e, \tilde e)$ with the group action defined in~\eqref{eq:parity_group}.
We compute the $W_2$ and $W_{2,\mathcal M}$ barycenters for the symmetric measure of a mixture of two Slater functions with respect to this group action and present them in Figure~\ref{fig:parity1d}. We observe that the barycenters correspond to the symmetrized version of the barycenters presented on Figure~\ref{fig:slater1d} which is naturally expected. 

\begin{figure}[h!]
    \centering
    \subfigure[$t=0$]{\includegraphics[width=0.19\textwidth]{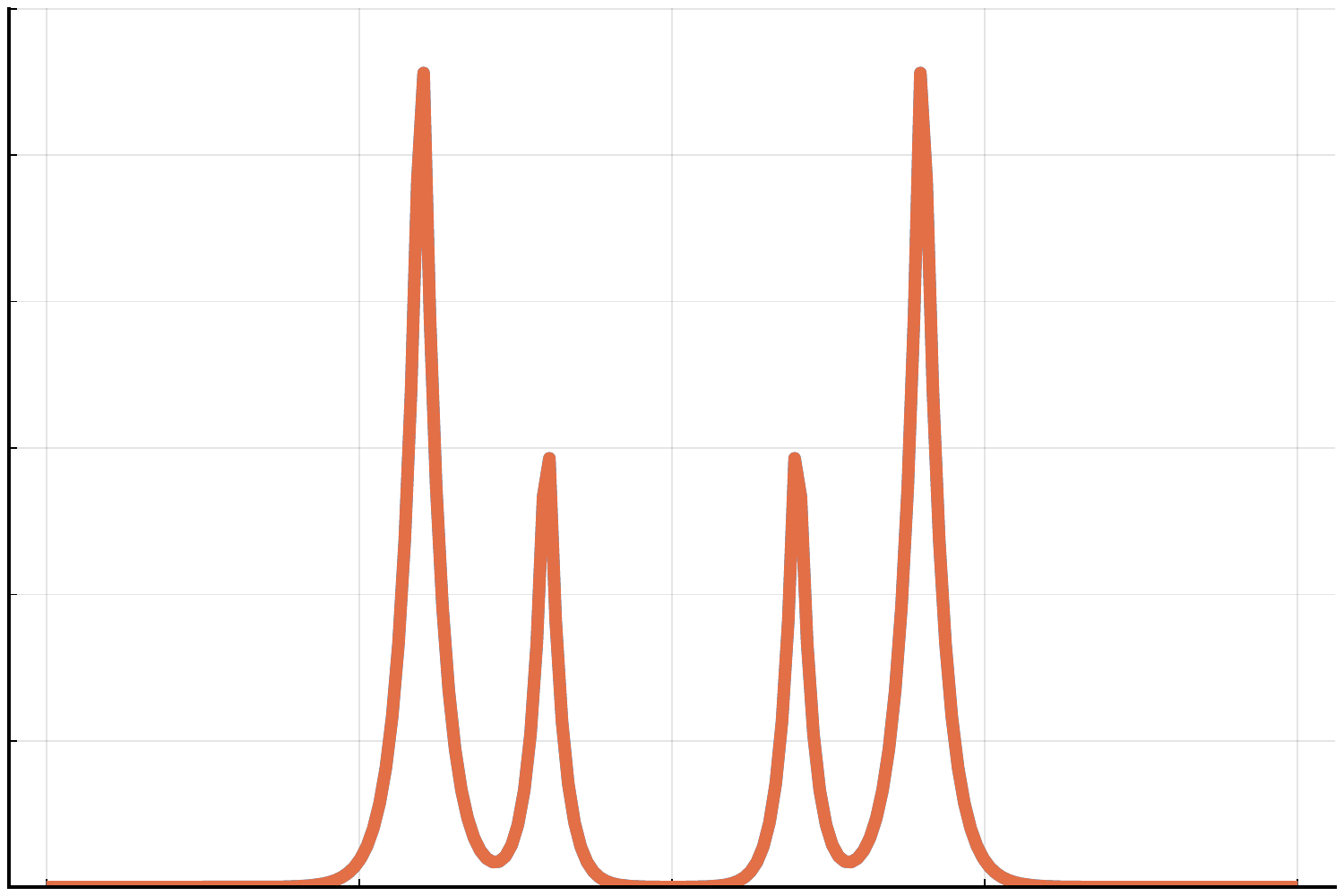}} 
    \subfigure[$t=0.25$]{\includegraphics[width=0.19\textwidth]{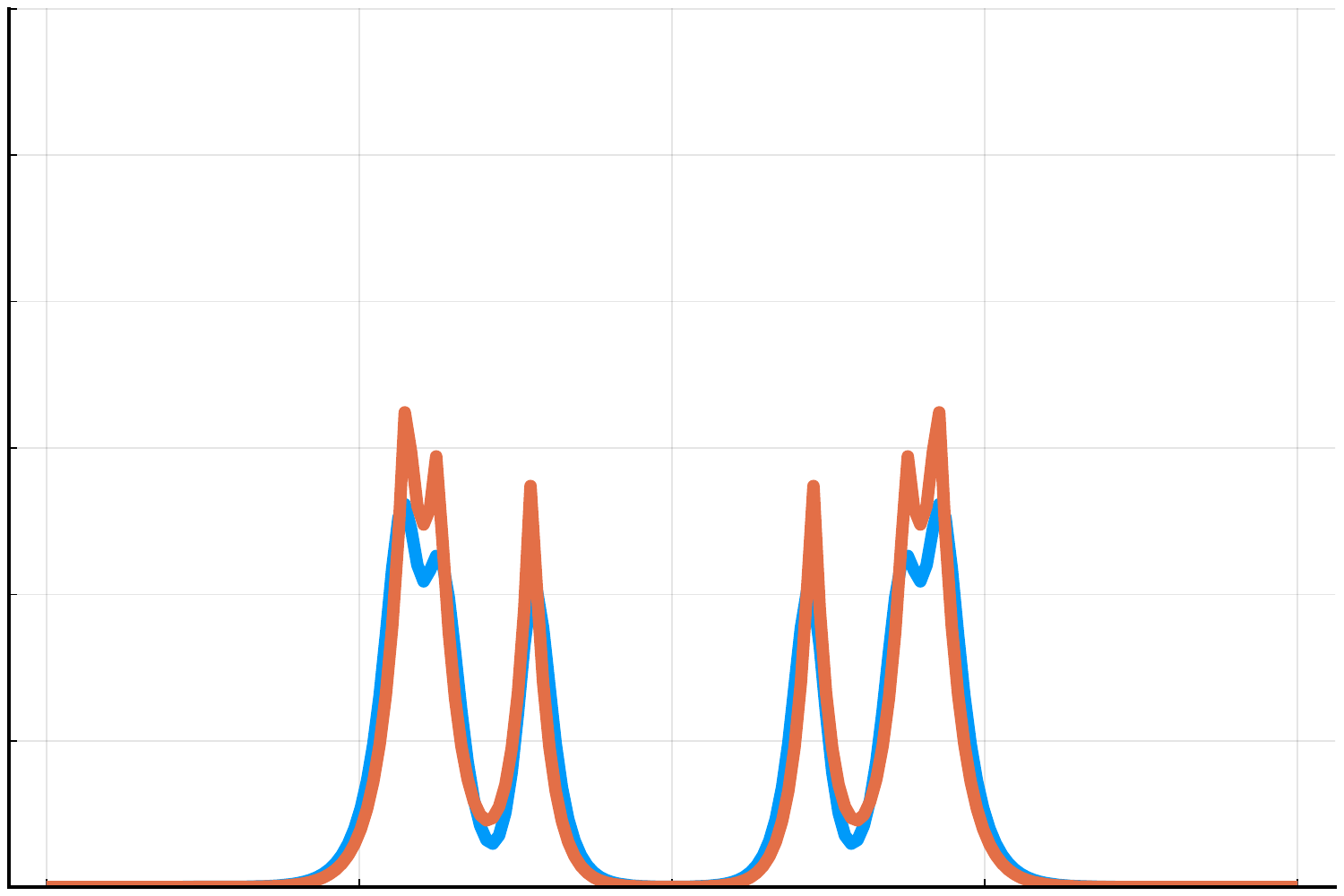}} 
    \subfigure[$t=0.5$]{\includegraphics[width=0.19\textwidth]{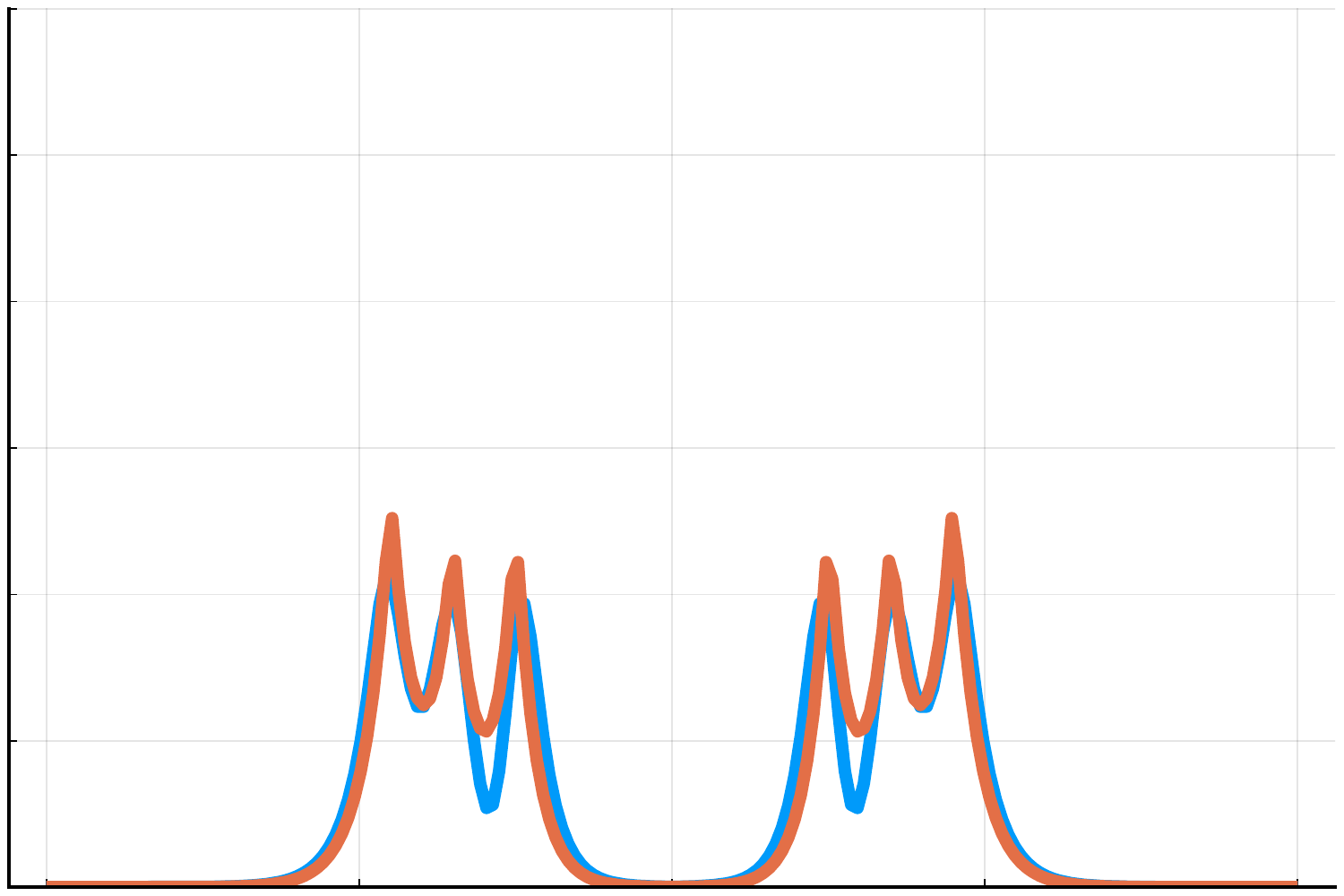}}
    \subfigure[$t=0.75$]{\includegraphics[width=0.19\textwidth]{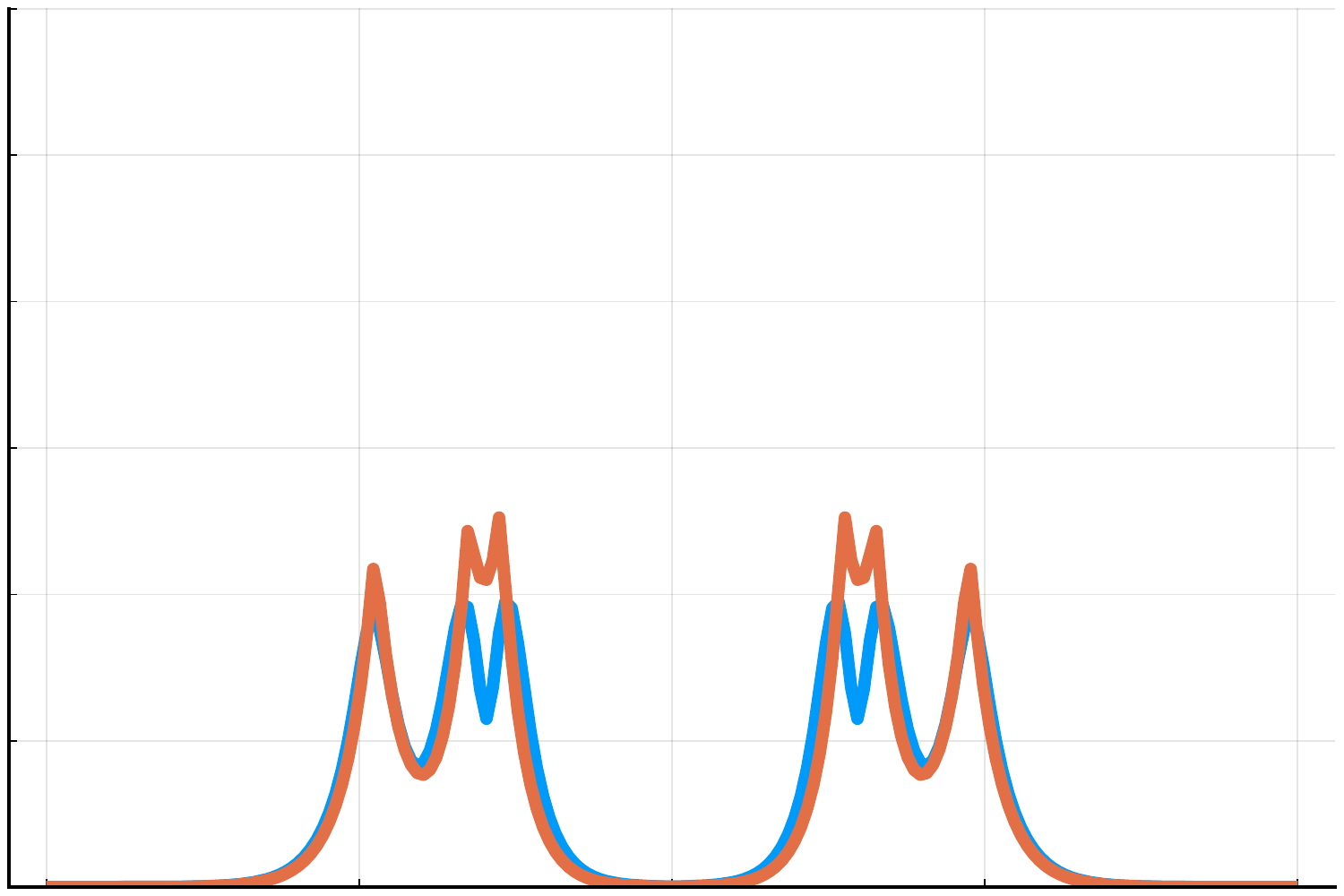}}
    \subfigure[$t=1$]{\includegraphics[width=0.19\textwidth]{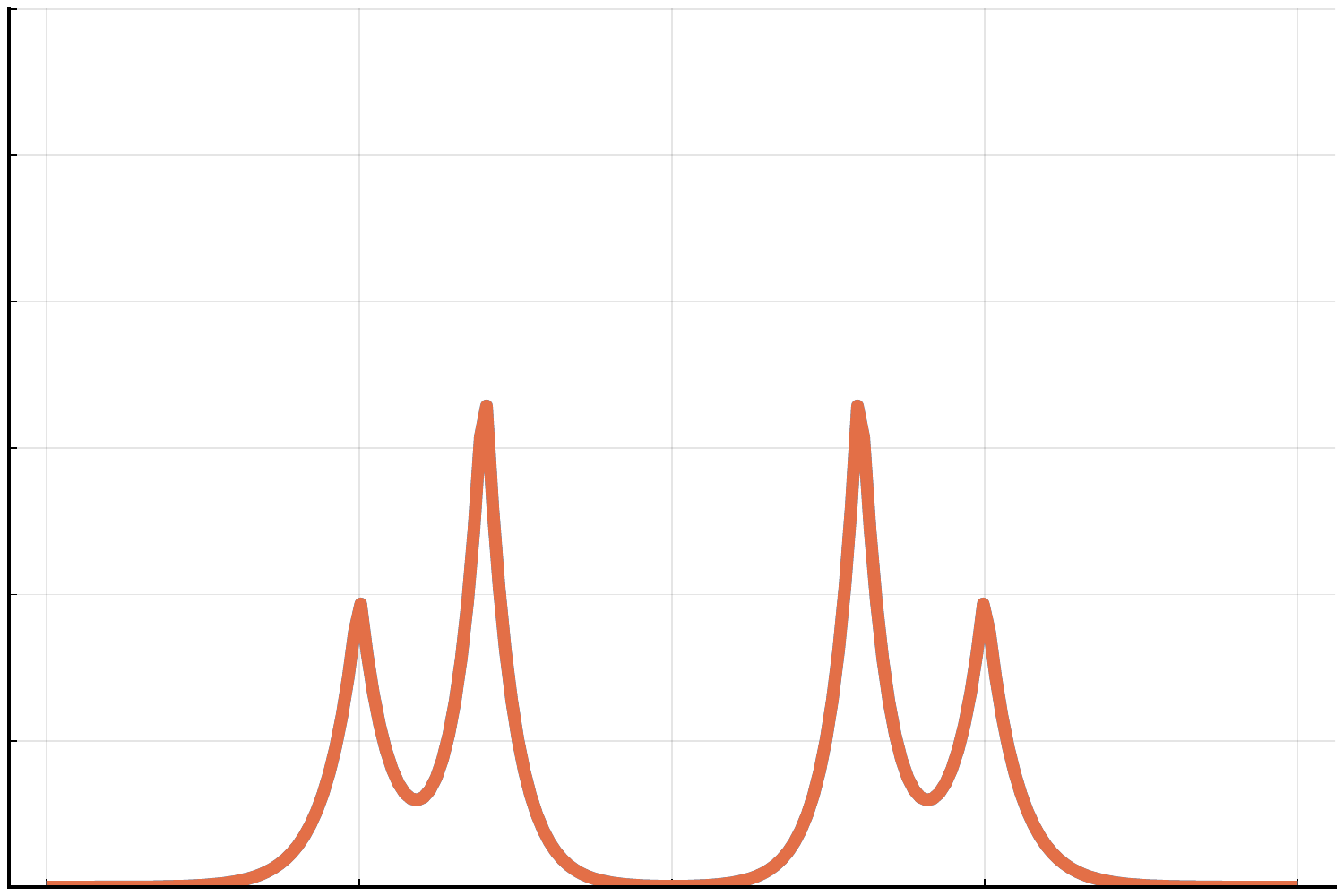}}
    \caption{Wasserstein barycenters between two mixtures of symmetric Slater distributions for the $W_2$ metric (blue) and the $W_{2,\mathcal M}$ metric (red)}
    \label{fig:parity1d}
\end{figure}

\subsubsection{Permutation group}\label{sec:permutation}

We then consider the permutation group $\mathfrak{S}_n$ on  the set $\{1,\ldots,n\}$. We define a group action $\cdot$  on elements of $\P_2(\D)$ by
\[
   \forall a\in\P_2(\D), \quad \forall \sigma \in \mathfrak{S}_n,\quad 
   \sigma \cdot a = a\# T_\sigma,
\]
where 
$$
T_\sigma: \left\{
\begin{array}{ccc}
\D \to \D \\
(x_1,\ldots,x_n) & \mapsto & (x_{\sigma(1)}, \ldots, x_{\sigma(n)}),\\
\end{array}
\right.
$$
which corresponds to a permutation of the variables.
It is easy to check that this indeed defines a group action. For a given set of atomic measures $\mathcal A$, the corresponding set of symmetric atomic measures $\A_{\rm sym}$ is then given as the set of elements $\overline{a}$ defined for all $a \in \mathcal A$ by
\[
\bar a =  \frac{1}{n!}
\sum_{\sigma \in\mathfrak{S}_n} \sigma \cdot a.
\]
The computation of the symmetric barycenters can be then easily done, provided that the underlying set of atoms allows for an efficient - or even explicit - computation of distances $d$. Choosing for the atoms $\mathcal A$ any set of location-scatter atoms, the computation of the symmetric Wasserstein distance $\bar d$ only requires the computation of $n!$ explicit Wasserstein distances, which stays cheap for moderate values of $n$.
In Figure~\ref{fig:symm_gauss2dmw2_contour} 
we plot the $W_2$ barycenters between symmetric mixtures of gaussian measures for the group action defined above and compare it to the symmetric barycenter based on gaussian mixtures. We observe that the mixture barycenter is smoother than the $W_2$ barycenter, and obtained at a fraction of the cost.

\begin{figure}
    \centering
    \subfigure[$t=0$]{
    \begin{tabular}{@{}c@{}}
         \includegraphics[width=0.18\textwidth]{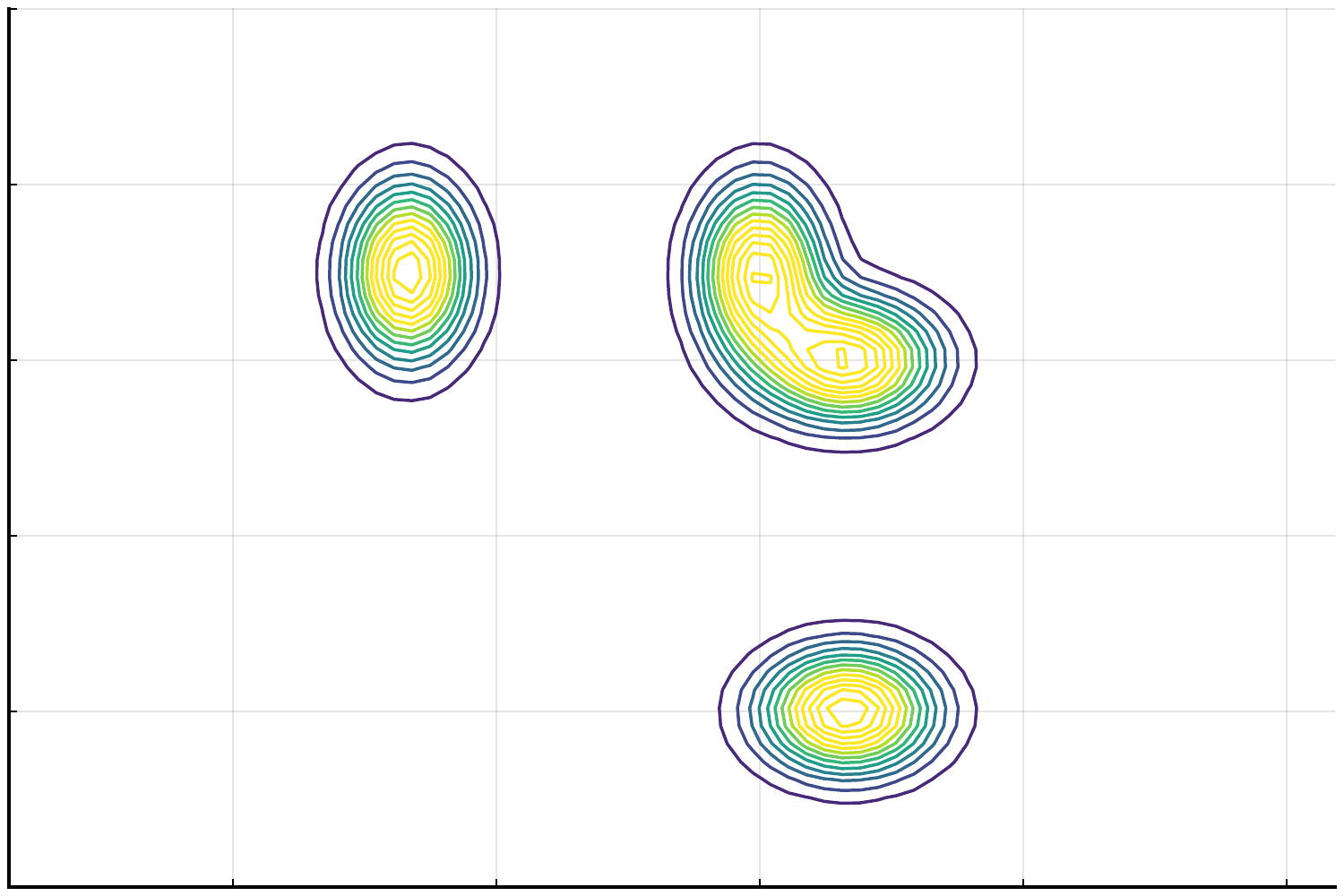} \\
         \includegraphics[width=0.18\textwidth]{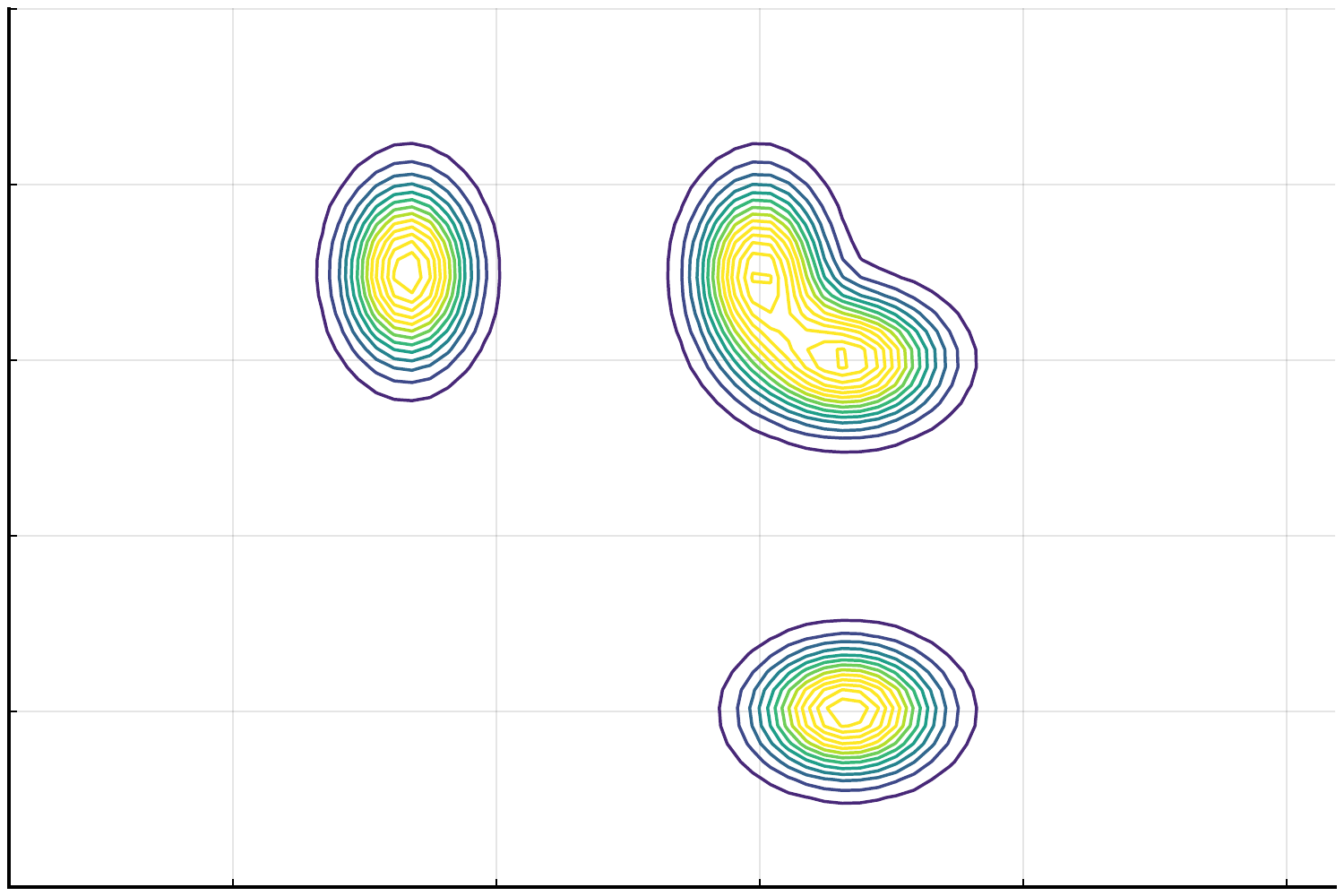}
    \end{tabular}
    } 
    \subfigure[$t=0.25$]{\begin{tabular}{@{}c@{}}
         \includegraphics[width=0.18\textwidth]{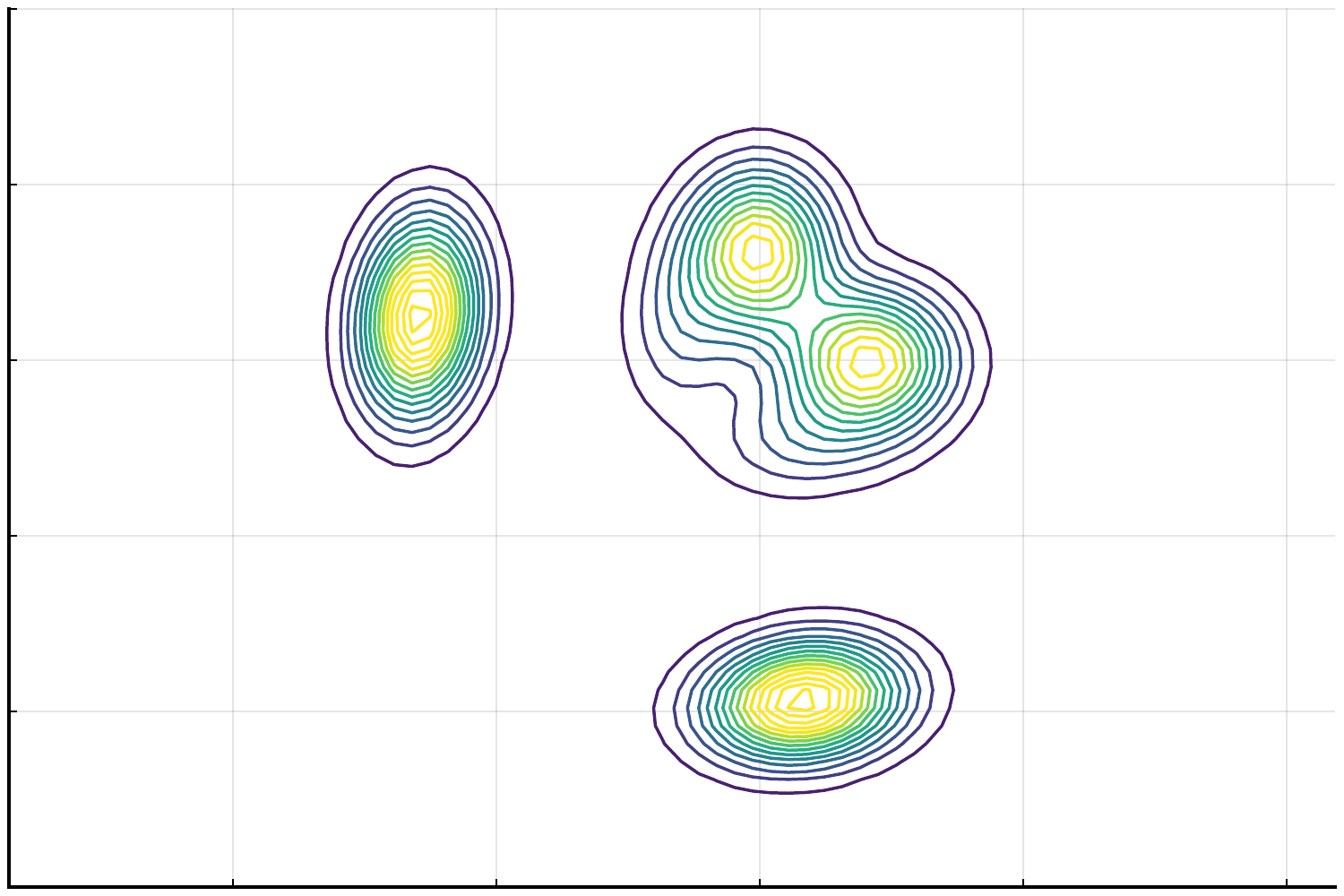} \\
         \includegraphics[width=0.18\textwidth]{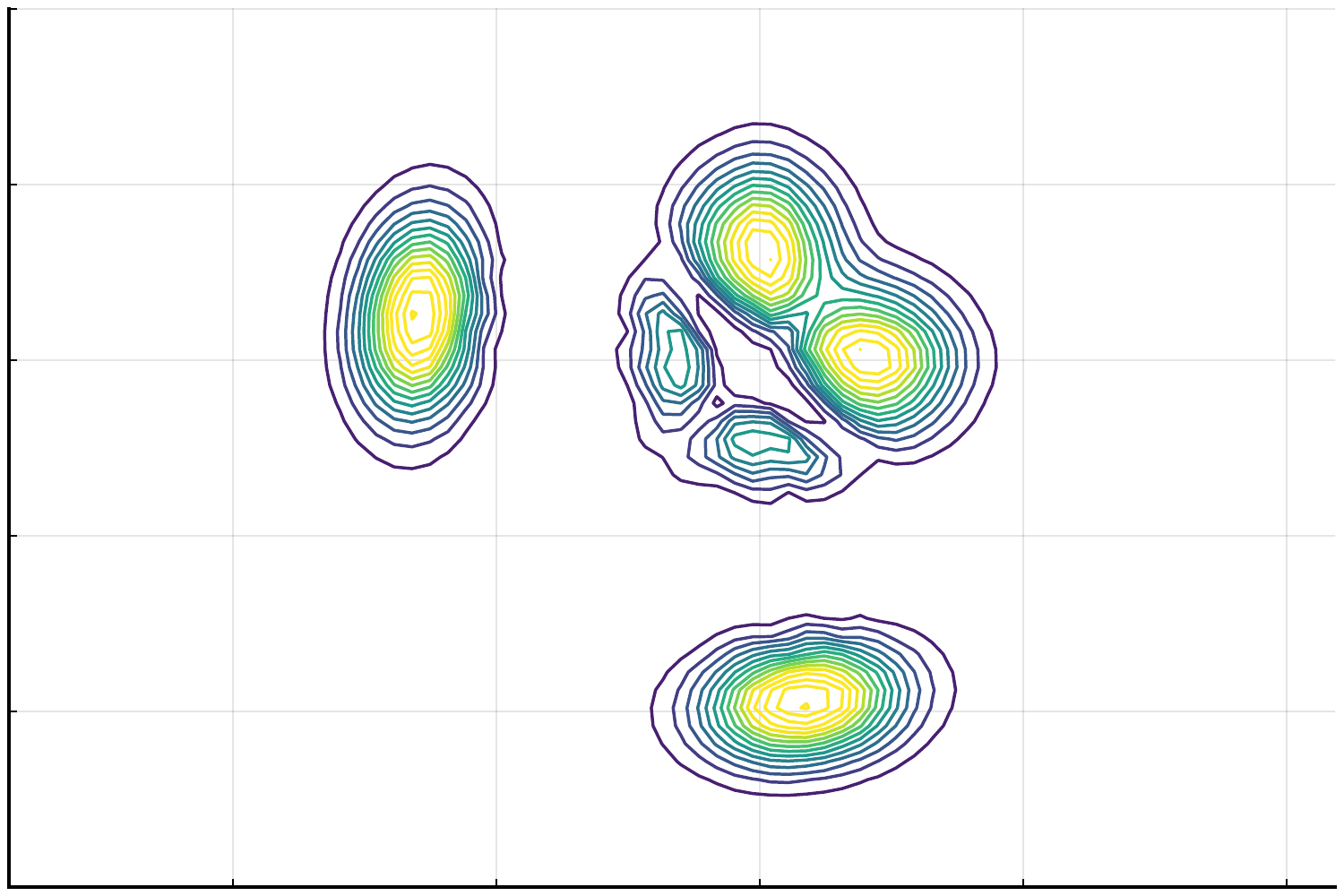}
    \end{tabular}} 
    \subfigure[$t=0.5$]{\begin{tabular}{@{}c@{}}
         \includegraphics[width=0.18\textwidth]{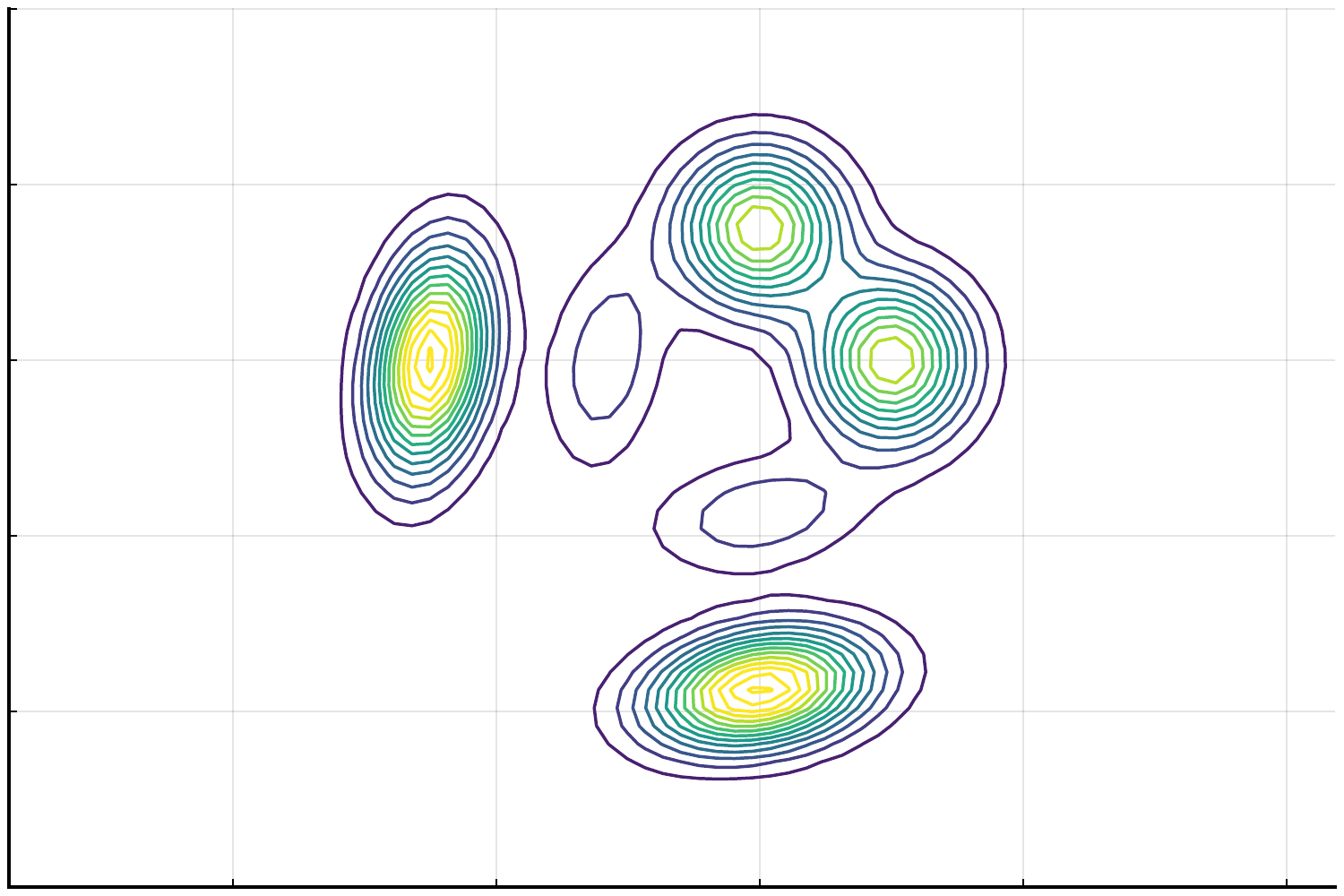} \\
         \includegraphics[width=0.18\textwidth]{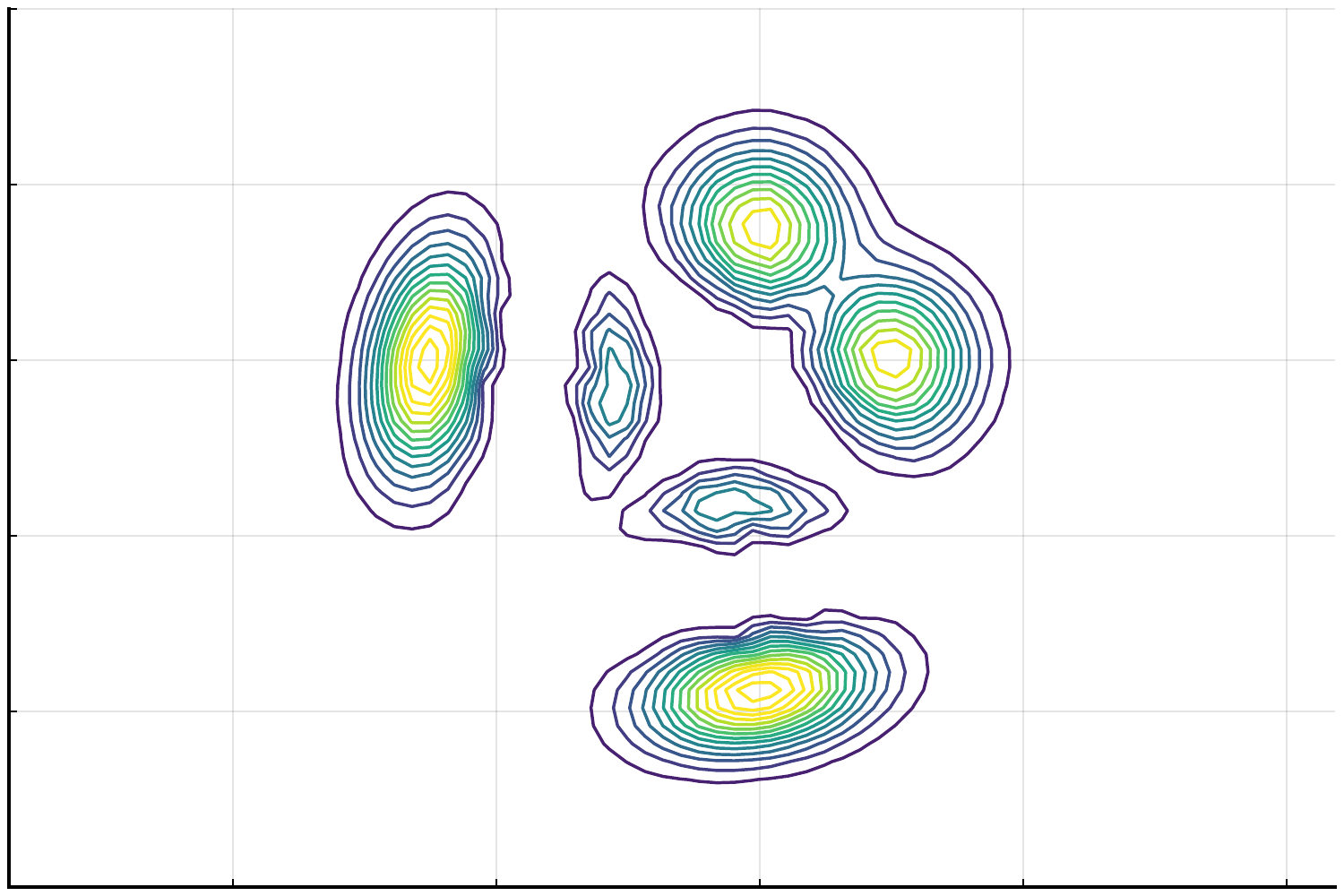}
    \end{tabular}} 
    \subfigure[$t=0.75$]{\begin{tabular}{@{}c@{}}
         \includegraphics[width=0.18\textwidth]{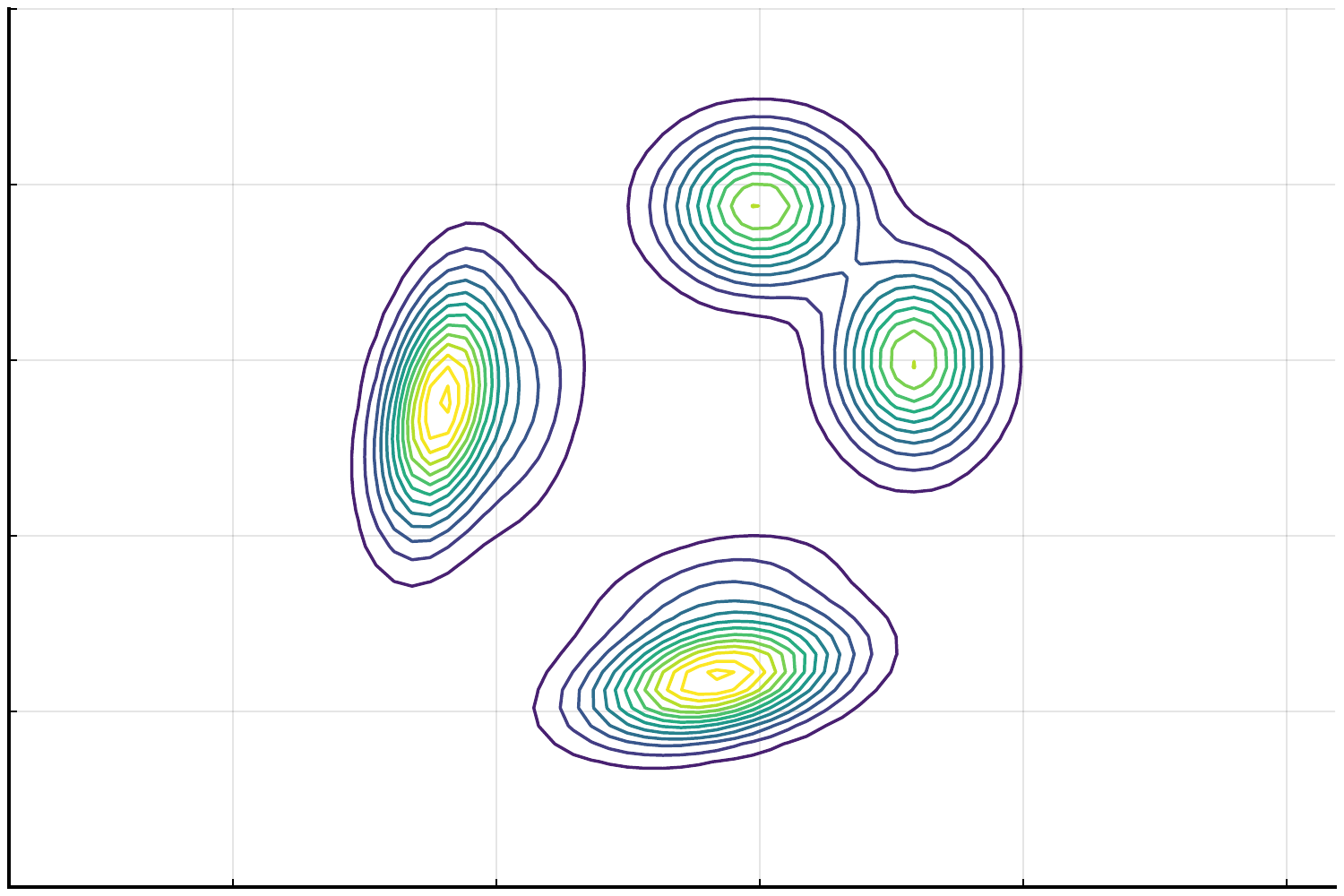} \\
         \includegraphics[width=0.18\textwidth]{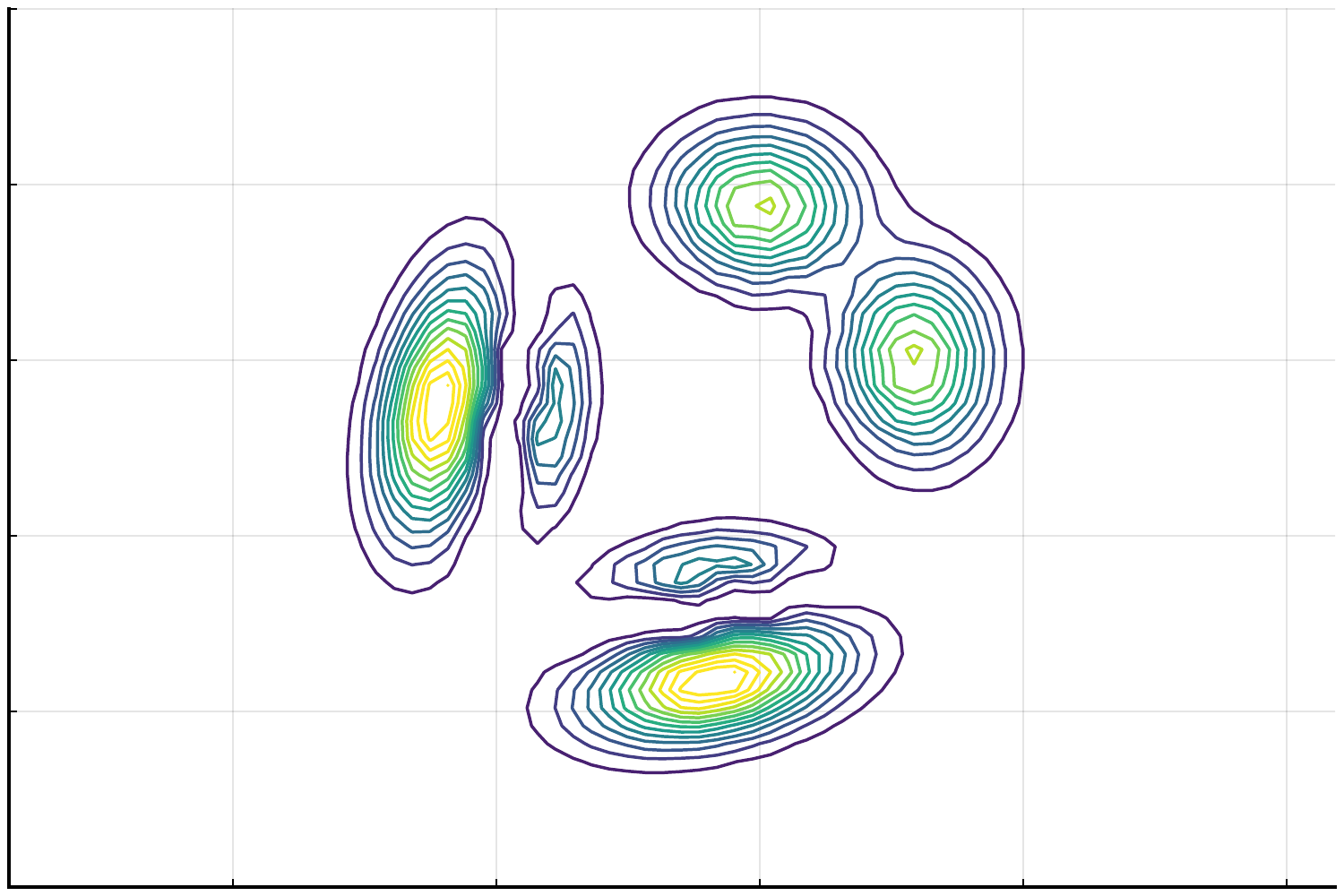}
    \end{tabular}} 
    \subfigure[$t=1$]{\begin{tabular}{@{}c@{}}
         \includegraphics[width=0.18\textwidth]{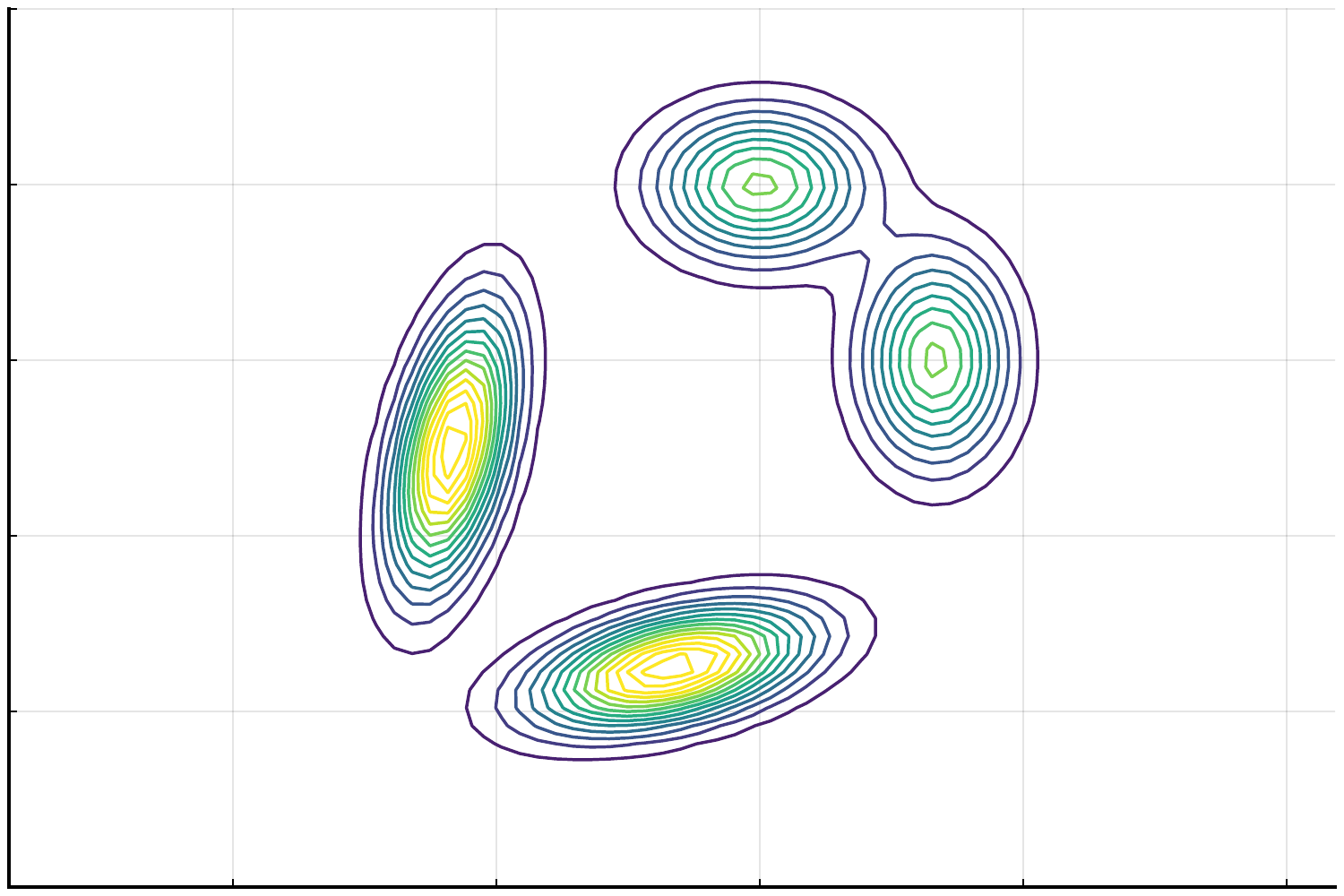} \\
         \includegraphics[width=0.18\textwidth]{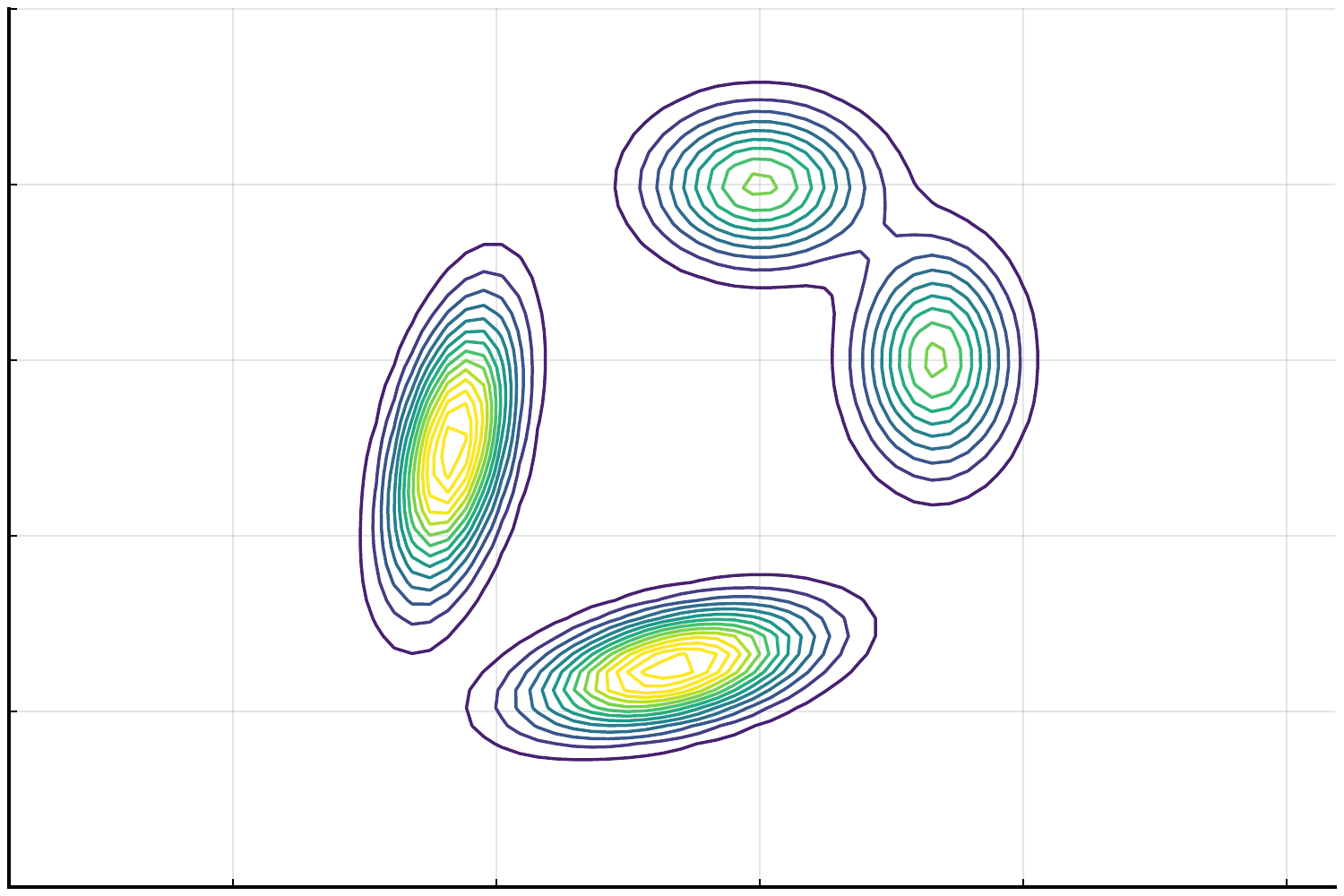}
    \end{tabular}} 
    \caption{Contour plots of $W_{2,\mathcal M}$ (top) and $W_2$ (bottom) barycenters between two mixtures of symmetric gaussian  distributions.}
    \label{fig:symm_gauss2dmw2_contour}
\end{figure}

\subsubsection{Application in quantum chemistry}

One application of the symmetry group invariant mixtures presented in the previous section is to compute Wasserstein-type distances and interpolants (i.e. barycenters) of many-body densities in electronic structure theory. As mentioned in the introduction, this was the first motivation of the present work. 

In quantum chemistry, the state of a system of $N$ electrons is fully characterized by its electronic wavefunction, which is a function $\psi \in L^2(\mathbb{R}^{dN}; \mathbb{C})$ with $d$  the dimension of the underlying physical space (typically $d=1,2,3$). This wavefunction is normalized in the sense that 
$$
\int_{\mathbb{R}^{dN}} |\psi|^2 = 1.
$$
Note that we omit here the spin variables for simplicity of the presentation.
Since electrons are fermions, the function $\psi$ is antisymmetric with respect to permutations of the ordering of the variables. More precisely, for all $\sigma \in \mathfrak{S}_N$ that is a permutation of the set $\{1, \cdots, N\}$ and all $(x_1,\ldots,x_N)\in (\mathbb{R^d})^N$, there holds
$$
\psi(x_{\sigma(1)}, \ldots, x_{\sigma(N)}) = \epsilon(\sigma) \psi(x_1, \ldots, x_N),
$$
where $\epsilon(\sigma)$ denotes the signature of the permutation $\sigma$. 
In electronic structure calculations, such a wavefunction is often approximated as a finite linear combination of so-called Slater determinants, which are defined as follows. For any set $\Phi = \{\phi_1, \ldots, \phi_N\}$ of $N$ functions of $L^2(\mathbb{R}^d; \mathbb{C})$, the associated Slater determinant $S_\Phi$ is defined as the function in $L^2(\mathbb{R}^{dN}; \mathbb{C})$ such that for almost all 
$ (x_1,\ldots, x_N)\in (\mathbb{R}^d)^N$, 
$$
S_\Phi(x_1,\ldots,x_N) = \frac{1}{Z_\Phi}{\rm det}\left( \begin{array}{cccc}
\phi_1(x_1) & \phi_1(x_2) & \ldots & \phi_1(x_N)\\
\phi_2(x_1) & \phi_2(x_2) & \ldots & \phi_2(x_N)\\
\vdots & \vdots & \ddots & \vdots\\
\phi_N(x_1) & \phi_N(x_2) & \ldots & \phi_N(x_N)\\
\end{array}
\right),
$$
where $Z_\Phi >0$ is the normalisation constant such that $\int_{\mathbb{R}^{dN}} |S_\Phi|^2 = 1$.

Since the wavefunction is defined on a high-dimensional space  when the number of electrons in the system is large, it is more convenient to handle the (normalized) one-body density which is defined as follows: 
$$
\forall x_1\in \mathbb{R}^d, \quad \rho_1(x_1):= \int_{(x_2,\ldots,x_N)\in (\mathbb{R}^d)^{N-1}} |\psi(x_1,x_2,\ldots, x_N)|^2\,dx_2\ldots\,dx_N.
$$
More generally, for all $1\leq n \leq N$, the $n$-body density associated to the wavefunction $\psi$ is defined as follows:
\begin{equation}\label{eq:nbody}
\forall \bm{x} = (x_1,x_2, \ldots x_n)\in (\mathbb{R}^d)^n, \quad \rho_n(\bm{x}):= \int_{(x_{n+1},\ldots,x_N)\in (\mathbb{R}^d)^{N-n}} |\psi(\bm{x})|^2\,dx_{n+1}\ldots\,dx_N.
\end{equation}
For all $1\leq n \leq N$, $\rho_n$ can then be seen as the density associated to a probability measure on $\mathbb{R}^{dn}$.

Due to the linear approximation with Slater determinants based on gaussian functions often used in quantum chemistry codes to approximate the wave-function, as well as the symmetry constraints on the $n$-body densities, it seems natural to approximate $n$-body densities by mixtures of \itshape squared Slater determinants with Gaussian functions\normalfont. 
More precisely, for all $m_i\in \mathbb{R}^d$ and $\Sigma_i \in \mathcal S_d$ for $i=1,\ldots,n$, denoting by ${\bm m}:=(m_1,\ldots,m_n)$ and by ${\bm \Sigma}:=(\Sigma_1, \ldots, \Sigma_n)$, one can define
$$
a^{SD}_{{\bm m},{\bm \Sigma}}(dx_1,\ldots,dx_n):= \frac{1}{Y_{{\bm m},{\bm \Sigma}}}\left|S_{\Phi_{{\bm m},{\bm \Sigma}}}(x_1,\ldots,x_n)\right|^2\,dx_1\ldots\,dx_n,
$$
where $\Phi_{{\bm m},{\bm \Sigma}}:= \{ G_{m_1,\Sigma_1},\ldots, G_{m_n, \Sigma_n}\}$ and where $Y_{{\bm m},{\bm \Sigma}}>0$ is the normalization constant such that $a^{SD}_{{\bm m},{\bm \Sigma}}$ is a probability measure on $\mathbb{R}^{dn}$.

Denoting by 
$$
\mathcal A^{nd}_{SD}:= \left\{ a^{SD}_{{\bm m},{\bm \Sigma}}, \quad {\bm m}\in (\mathbb{R}^d)^n, \; {\bm \Sigma} \in (\mathcal S_d)^n \right\}, 
$$
electronic $n$-body densities of the form (\ref{eq:nbody}) can be approximated as elements of $\mathcal M\left(\mathcal A^{nd}_{SD}\right)$. 
We would therefore define a Wasserstein-type distance and Wasserstein-type barycenters on this set of mixtures. To this aim, we exploit the following remark. Consider $\mathcal A_{g,BD}^{nd}$ the set of gaussian measures on $\mathbb{R}^{nd}$ with block-diagonal covariance matrices, i.e. 
$$
\mathcal A_{g,BD}^{nd}:= \left\{ g_{{\bm m}, {\rm diag}({\bm \Sigma})}, \; {\bm m} \in (\mathbb{R}^d)^n, \; {\bm \Sigma} \in (\mathcal S_d)^n\right\},
$$
where ${\rm diag}({\bm \Sigma})$ is the $nd\times nd$ block-diagonal matrix with diagonal blocks given by $\Sigma_1,\ldots,\Sigma_n$.
It can be easily checked that $\mathcal A_{g,BD}^{nd}$ embedded with the Wasserstein distance is a geodesic space. 

Denoting by $\mathcal A_{g, {\rm sym}}^{nd}$ the set of symmetrized gaussian measures obtained from the dictionnary $\mathcal A_{g,BD}^{nd}$ via the permutation symmetry group introduced in Section~\ref{sec:permutation}, it can easily be checked that $\mathcal M\left(\mathcal A^{nd}_{SD}\right)$ is isomorphic to $\mathcal M \left( \mathcal A_{g, {\rm sym}}^{nd} \right)$.
More precisely, the following application $\mathcal I$ is an isomorphism:
$$
\mathcal I : \left\{
\begin{array}{ccc}
\mathcal M\left(\mathcal A^{nd}_{SD}\right) & \to & \mathcal M \left( \mathcal A_{g, {\rm sym}}^{nd} \right)\\
\sum_{k=1}^K \pi_k a_{{\bm m}^k, {\bm \Sigma}^k}^{SD} & \mapsto & \sum_{k=1}^K \pi_k \overline{g}_{{\bm m}^k,{\rm diag}({\bm \Sigma}^k)},
\end{array}
\right.
$$
where for all $1\leq k \leq K$, we denote by ${\bm m}^k\in (\R^d)^n$, ${\bm \Sigma}^k\in (\mathcal S_d)^n$, by $(\pi_k)_{1\leq k \leq K}$ an element of $\mathcal L_K$, and by $\overline{g}_{{\bm m}^k,{\rm diag}({\bm \Sigma}^k)}$ the symmetrized gaussian obtained from  $g_{{\bm m}^k,{\rm diag}({\bm \Sigma}^k)}$.
It is then natural to define a Wasserstein-type distance $d_{SD}$ on $\mathcal M\left(\mathcal A^{nd}_{SD}\right)$ as follows:
$$
\forall a_1, a_2 \in \mathcal M\left(\mathcal A^{nd}_{SD}\right), \quad d_{SD}(a_1, a_2):= \overline{d}(\mathcal I(a_1), \mathcal I(a_2)),  
$$
where $\overline{d}$ is the symmetry-invariant Wasserstein distance defined in~\eqref{eq:Wass-sym}. Wasserstein-like barycenters are then also naturally defined using the isomorphism $\mathcal I$.

To illustrate this case, we plot on Figure~\ref{fig:permantisym2dmw2_contour} the barycenters between two mixtures of squares of Slater determinants for the $W_2$ and $W_{2,\mathcal M}$ distances. First we observe that both types of barycenters satisfy that they are zero on the diagonal, as is constrained by the form of the densities for the $W_{2,\mathcal M}$ distance, and naturally expected from Proposition~\ref{prop:5.3} for the 2-Wasserstein barycenters. Second, we observe that the mixture Wasserstein barycenters are smoother than the $W_2$-barycenters.

The application of the Wasserstein-type distance and Wasserstein-type barycenters presented here to accelerate computations in quantum chemistry will be the object of the future work~\cite{PolackDusson}.

\begin{figure}
    \centering
    \subfigure[$t=0$]{
    \begin{tabular}{@{}c@{}}
         \includegraphics[width=0.18\textwidth]{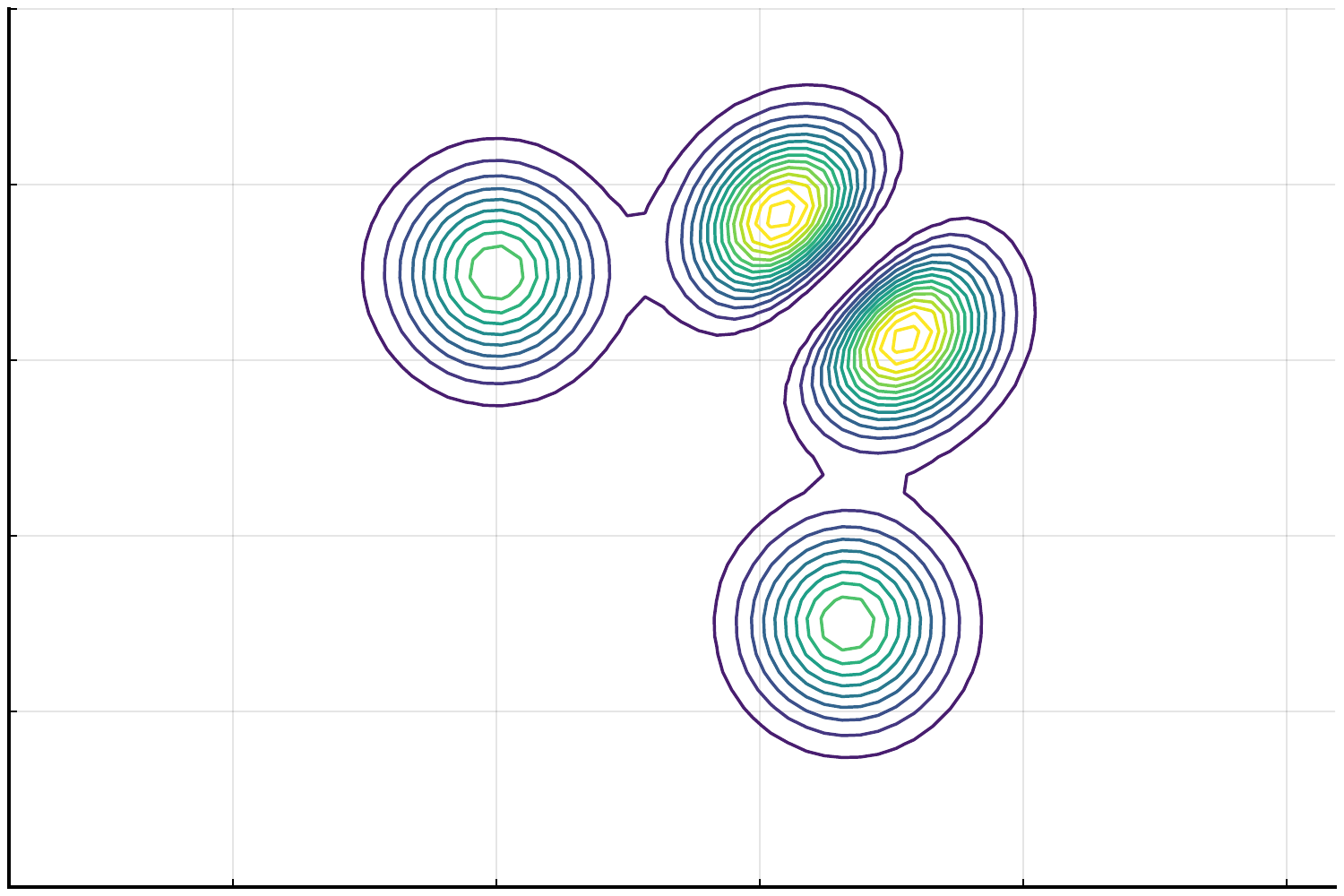} \\
         \includegraphics[width=0.18\textwidth]{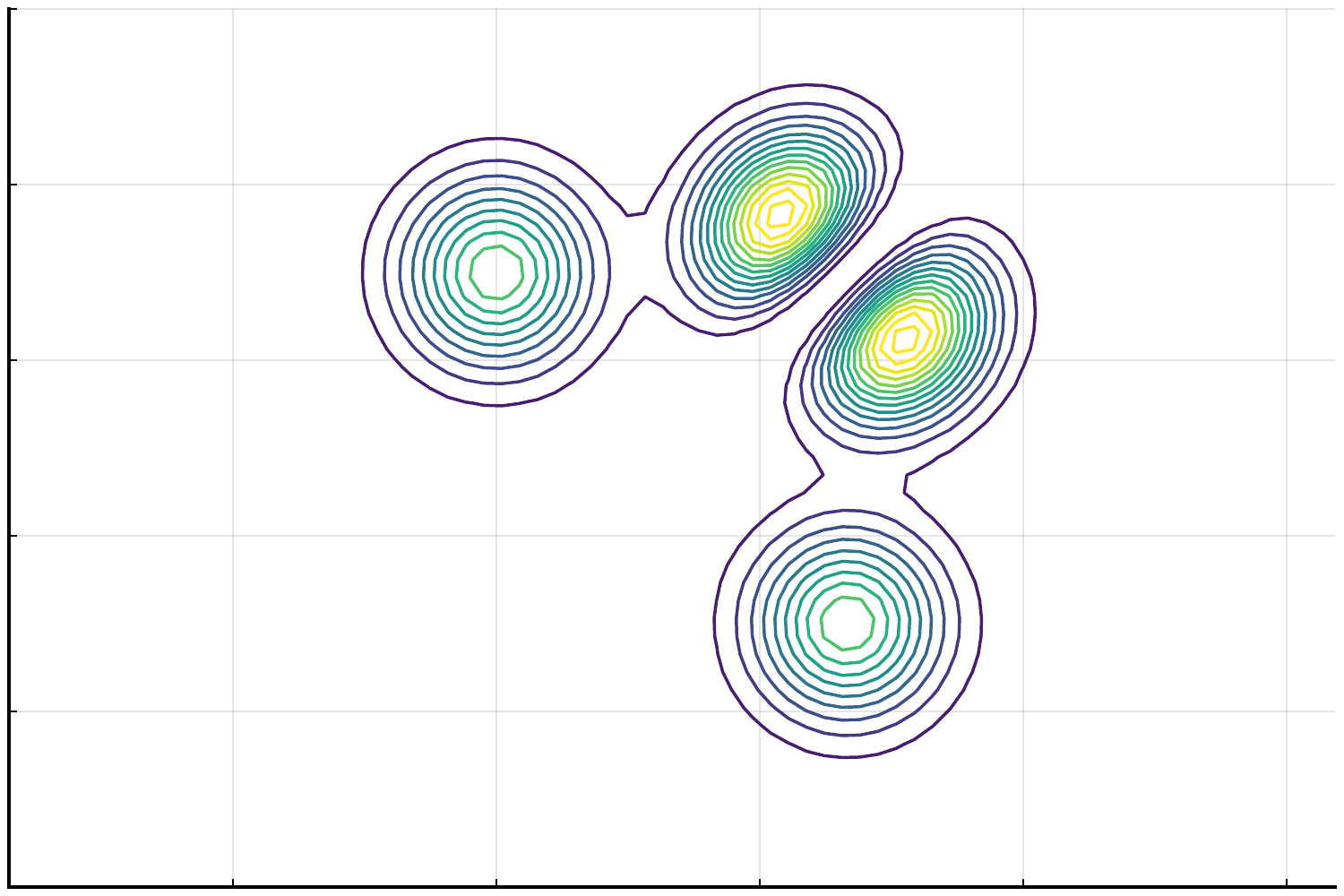}
    \end{tabular}
    } 
    \subfigure[$t=0.25$]{\begin{tabular}{@{}c@{}}
         \includegraphics[width=0.18\textwidth]{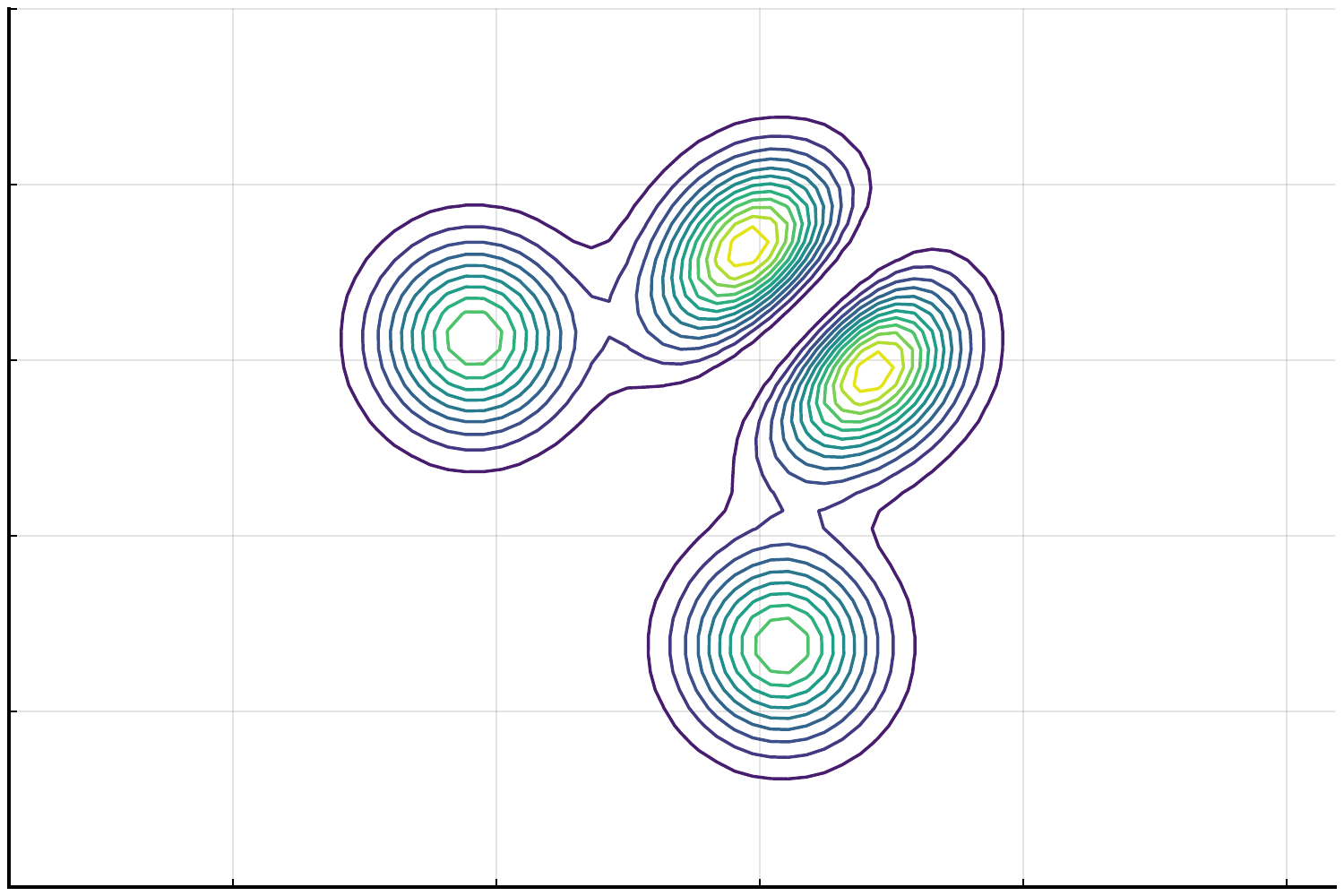} \\
         \includegraphics[width=0.18\textwidth]{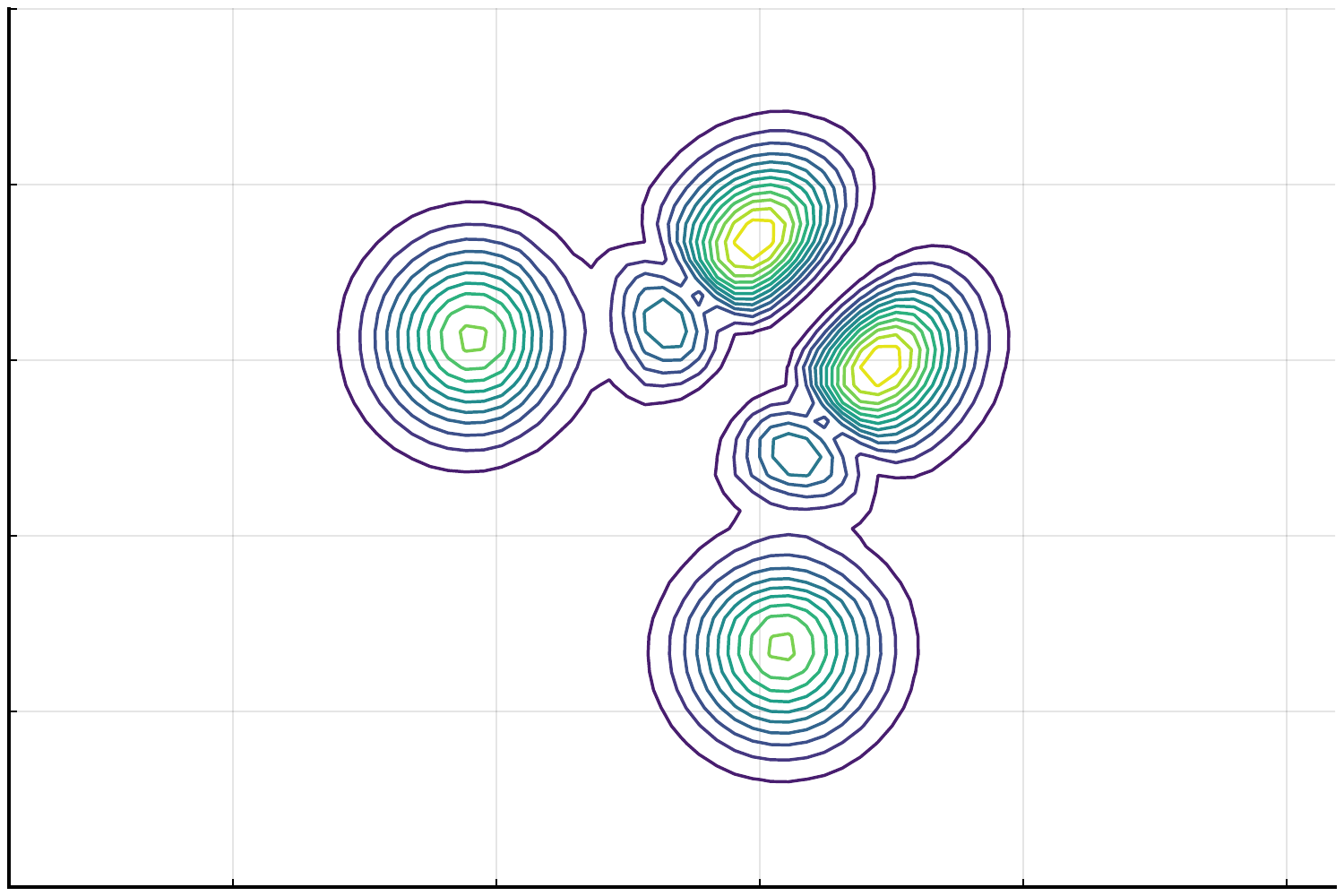}
    \end{tabular}} 
    \subfigure[$t=0.5$]{\begin{tabular}{@{}c@{}}
         \includegraphics[width=0.18\textwidth]{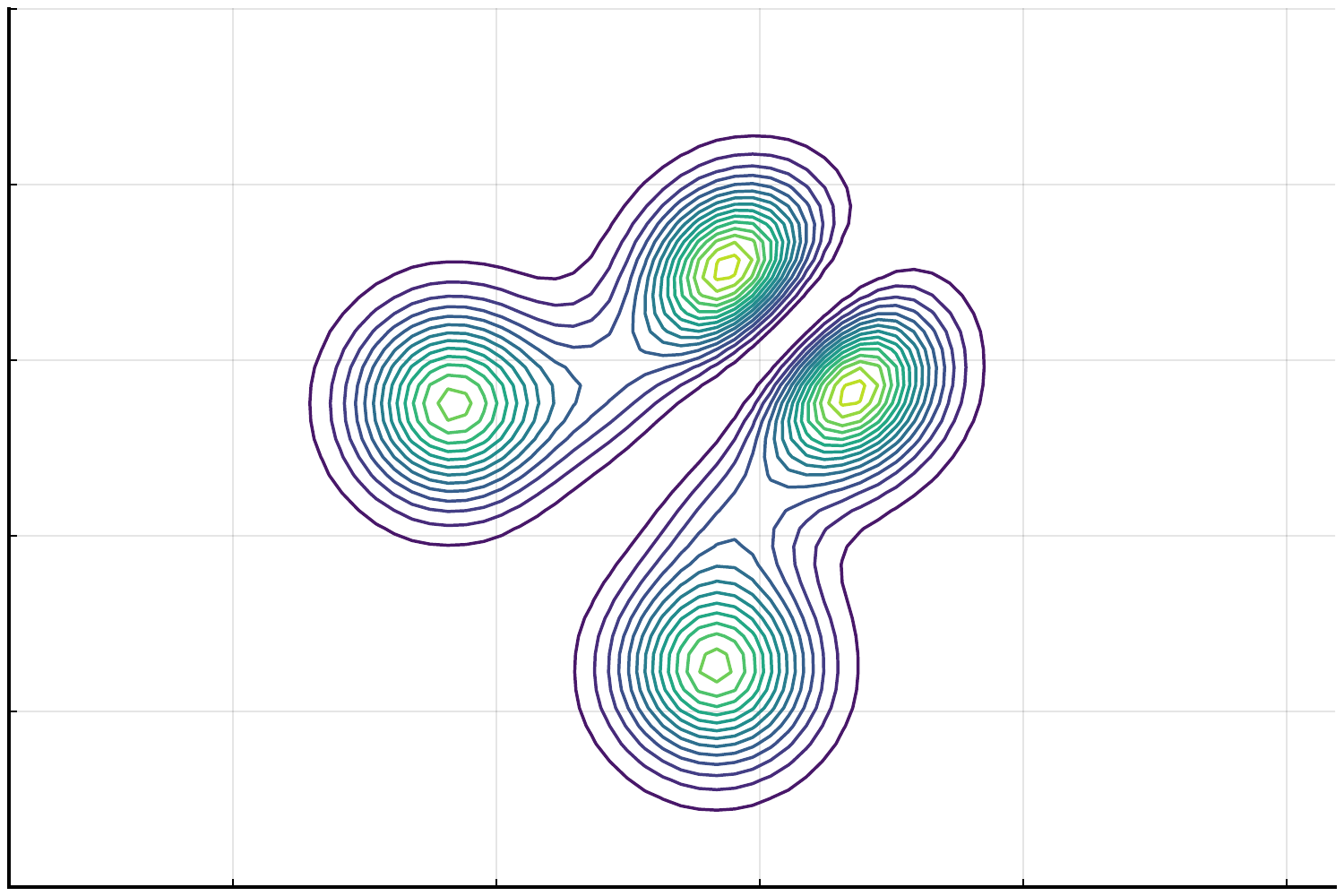} \\
         \includegraphics[width=0.18\textwidth]{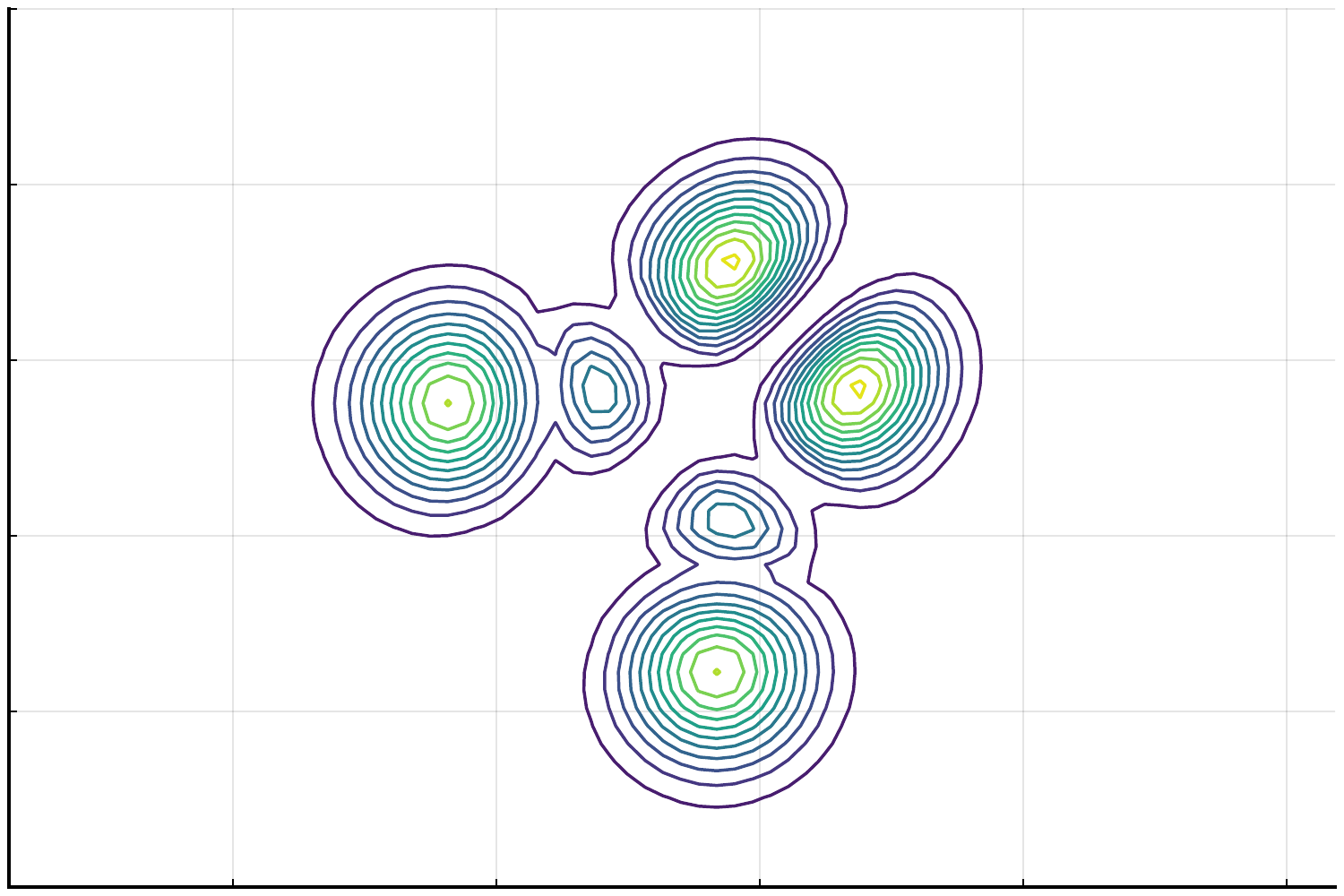}
    \end{tabular}} 
    \subfigure[$t=0.75$]{\begin{tabular}{@{}c@{}}
         \includegraphics[width=0.18\textwidth]{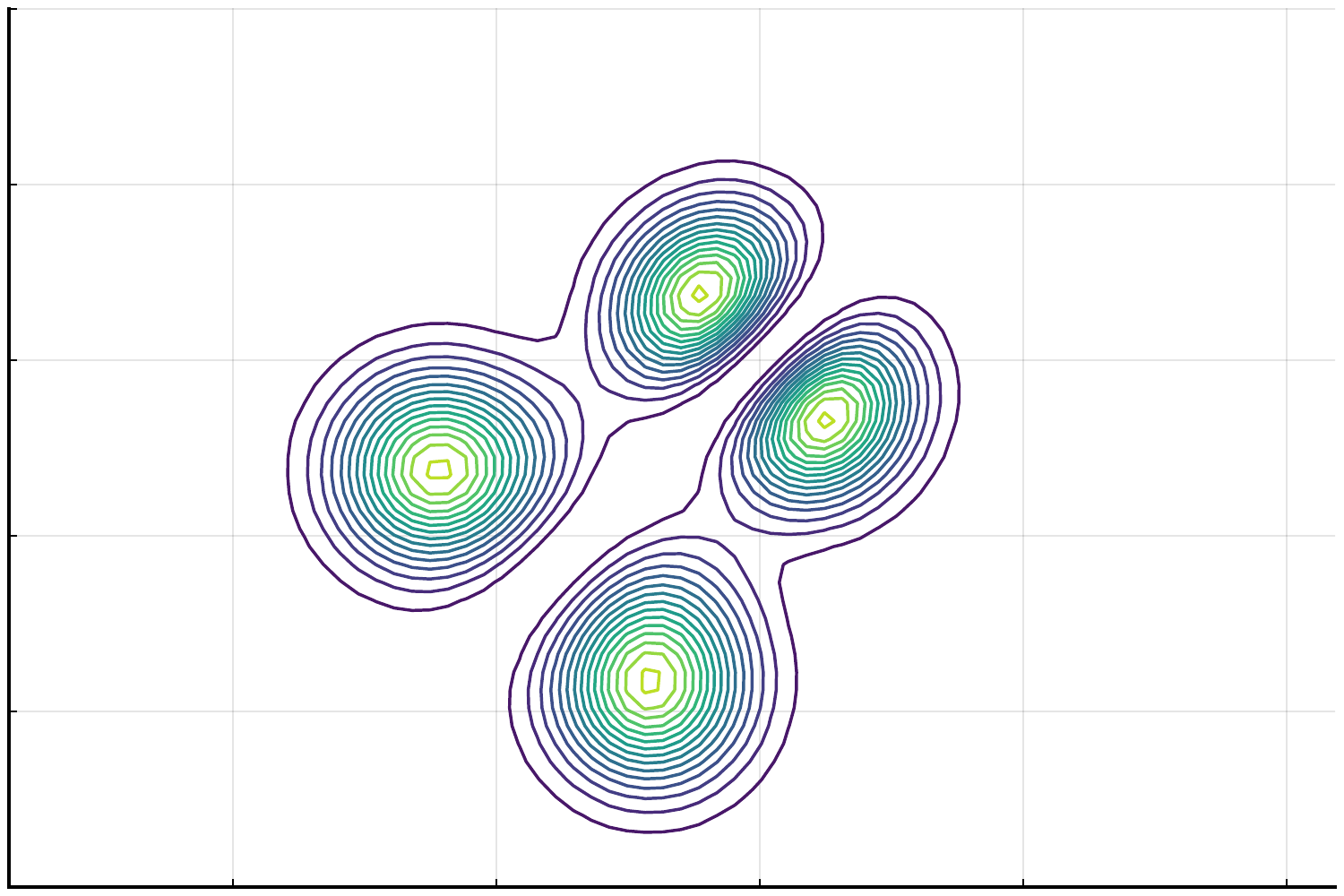} \\
         \includegraphics[width=0.18\textwidth]{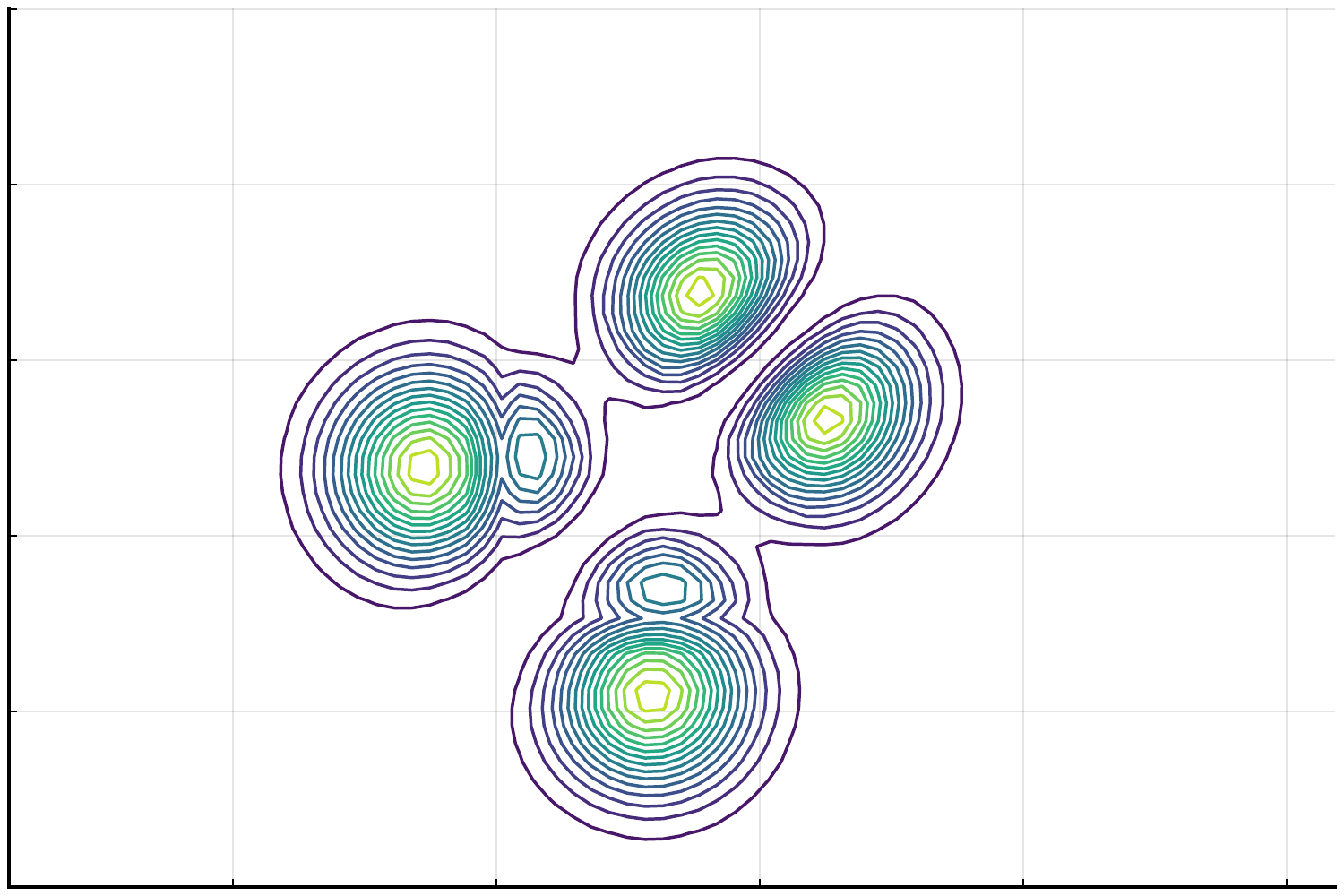}
    \end{tabular}} 
    \subfigure[$t=1$]{\begin{tabular}{@{}c@{}}
         \includegraphics[width=0.18\textwidth]{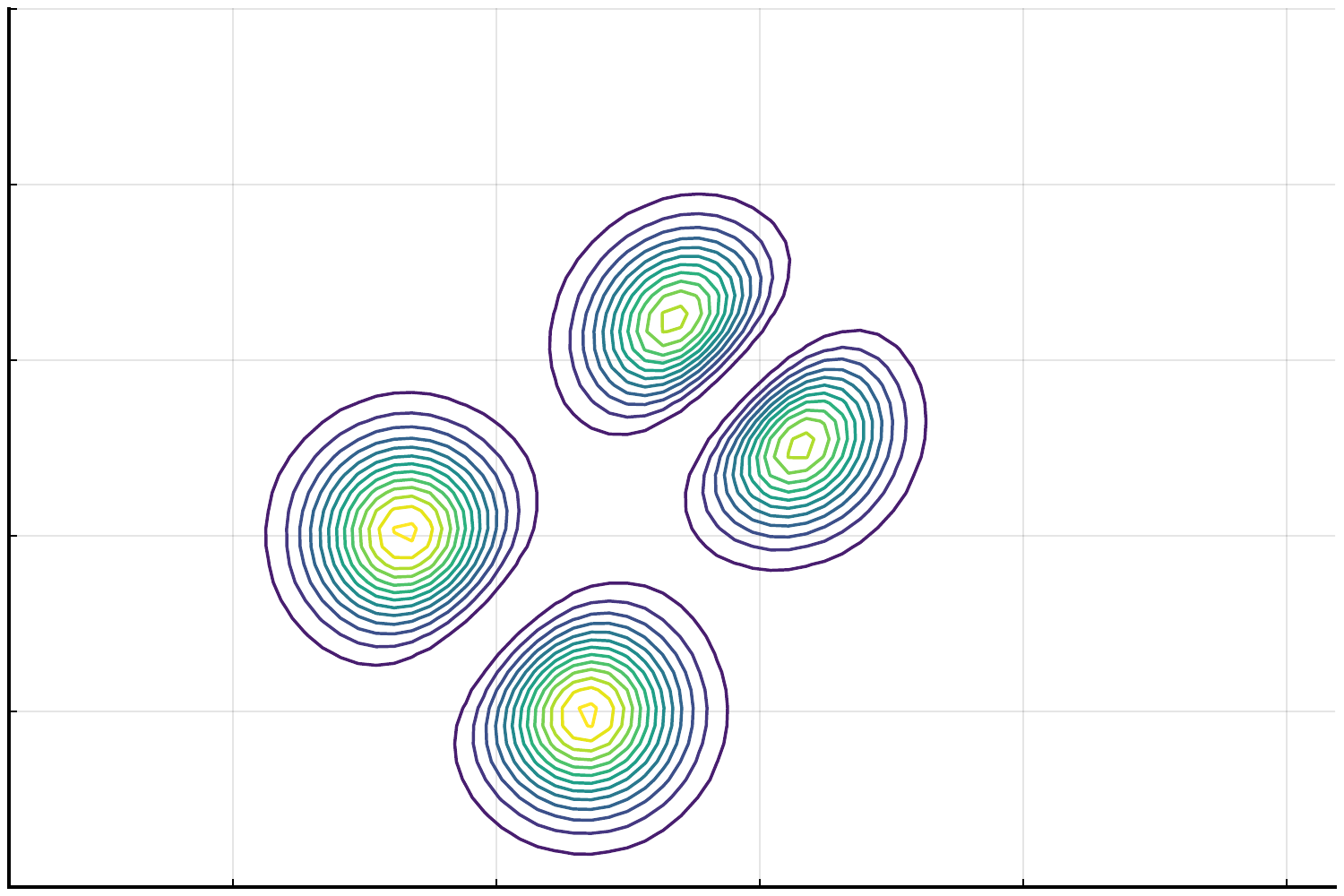} \\
         \includegraphics[width=0.18\textwidth]{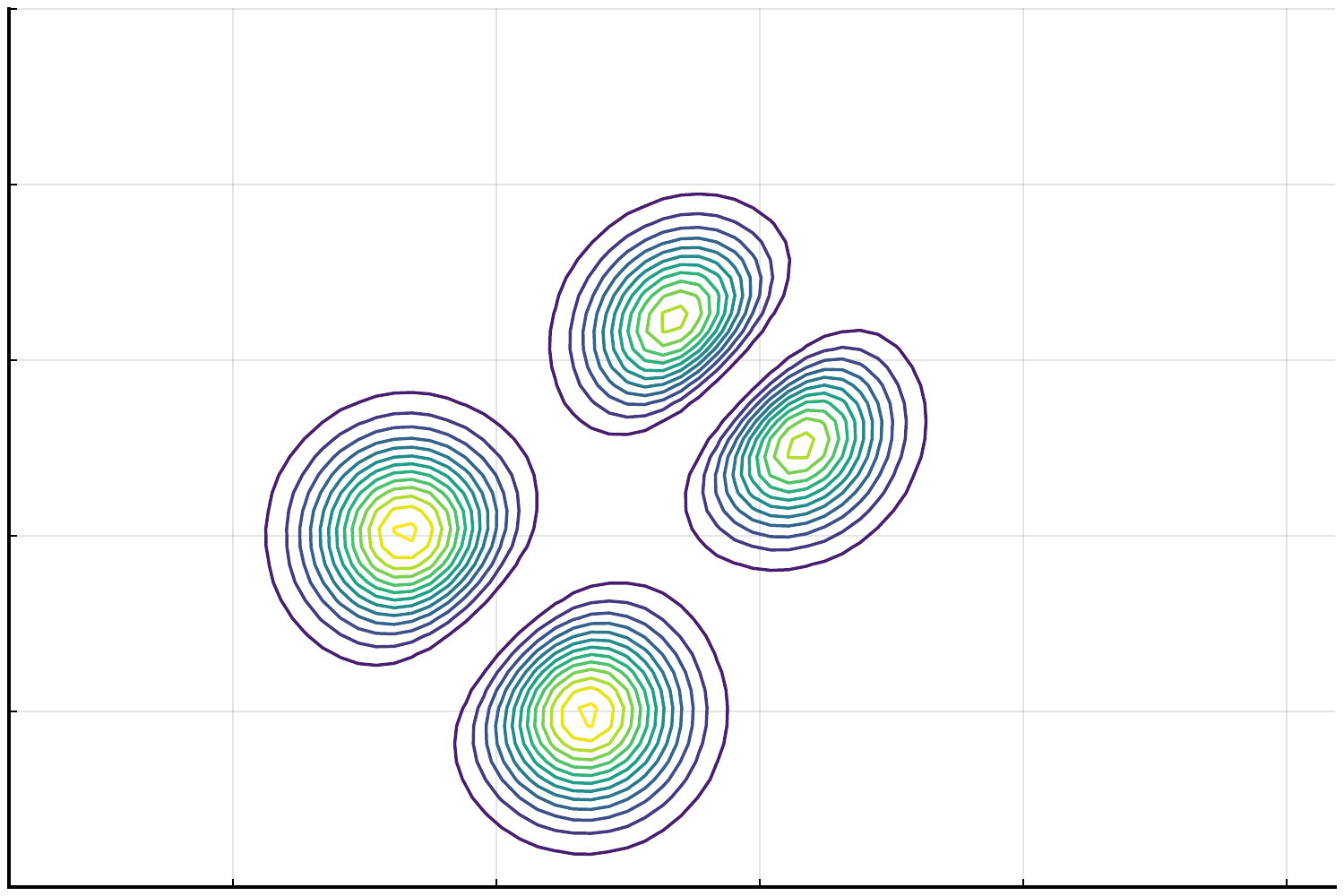}
    \end{tabular}} 
    \caption{Contour plots of $W_{2,\mathcal M}$ (top) and $W_2$ (bottom) barycenters between two mixtures of squared Slater determinants based on gaussian functions.}
    \label{fig:permantisym2dmw2_contour}
\end{figure}

\subsubsection{Rotation group $SO(2)$}

Finally, we consider the case where $\D = \mathbb{R}^2$ and the two-dimensional rotation group $SO(2)$, for which we can, as for the permutation group, define a natural group action 
$\cdot$ via
\[
   \forall a\in\A, \quad \forall Q \in SO(2),\quad 
   (Q \cdot a) = a\# T_Q ,
\]
with $T_Q: \mathbb{R}^2 \ni {\bf x}\mapsto Q{\bf x}$.
Note that such a group action can similarly be defined for $SO(n)$.
From definition~\eqref{eq:Asym} the symmetric atoms $\bar a$ of $\A_{\rm sym}$ write 
\[
\bar a = 
\int  Q\cdot a \; H(dQ).
\]
We now provide a numerical example, taking $\A$ as the set of two-dimensional slater-type distributions. The main difficulty and difference with the previous examples is that the group is not finite. 
Therefore, problem~\eqref{eq:symm_distance} which has to be solved to compute the distance between elements in $\A_{\rm sym}$ as well as barycenters is now a continuous optimization problem. 
In this two-dimensional case, there is in fact only one parameter to optimize which is the angle of the rotation, and we chose to perform this optimization numerically. This increases the computational cost of the distance and barycenters, however the computational cost stays way below the cost of computing the $W_2$ Wasserstein barycenters. 
In Figure~\ref{fig:rotsym2dmw2_contour}
we provide an example of barycenters between two mixtures of symmetric distributions, themselves based on one slater function each. 
We observe that the $W_2$ Wasserstein barycenter computed with the Sinkhorn algorithm is 
way less regular than the modified Wasserstein barycenter.

\begin{figure}
    \centering
    \subfigure[$t=0$]{
    \begin{tabular}{@{}c@{}}
         \includegraphics[width=0.19\textwidth]{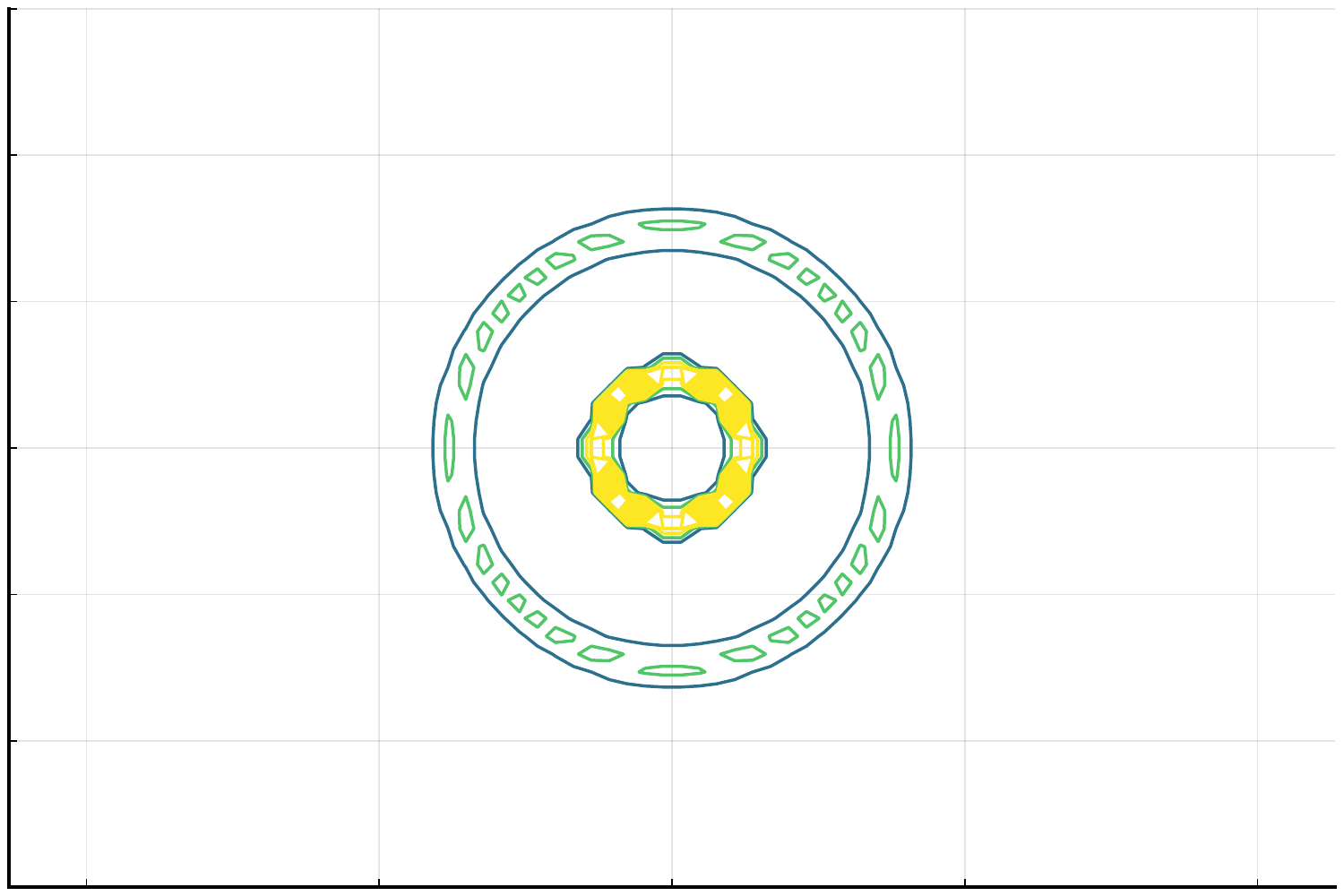} \\
         \includegraphics[width=0.19\textwidth]{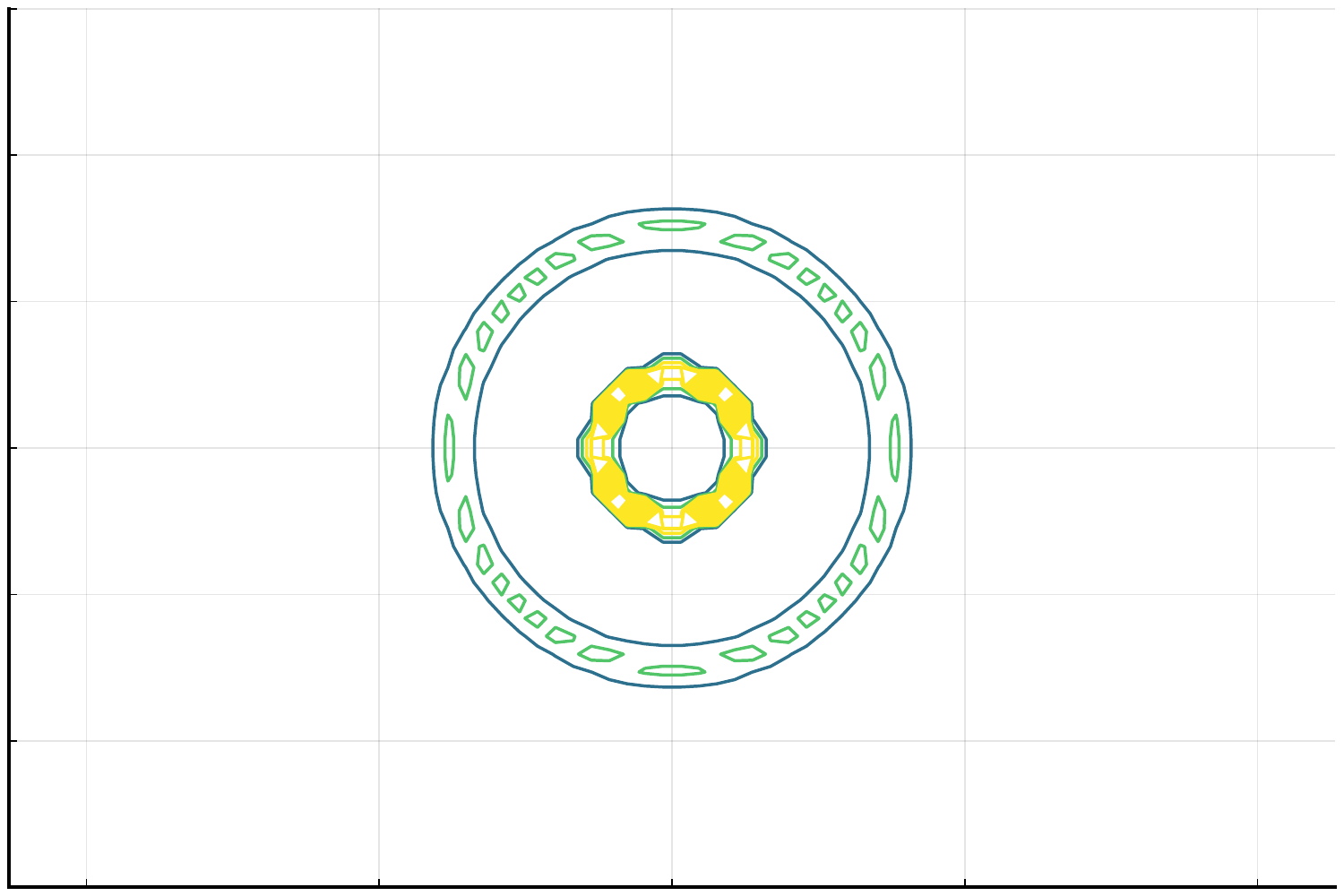}
    \end{tabular}} 
    \subfigure[$t=0.25$]{\begin{tabular}{@{}c@{}}
         \includegraphics[width=0.19\textwidth]{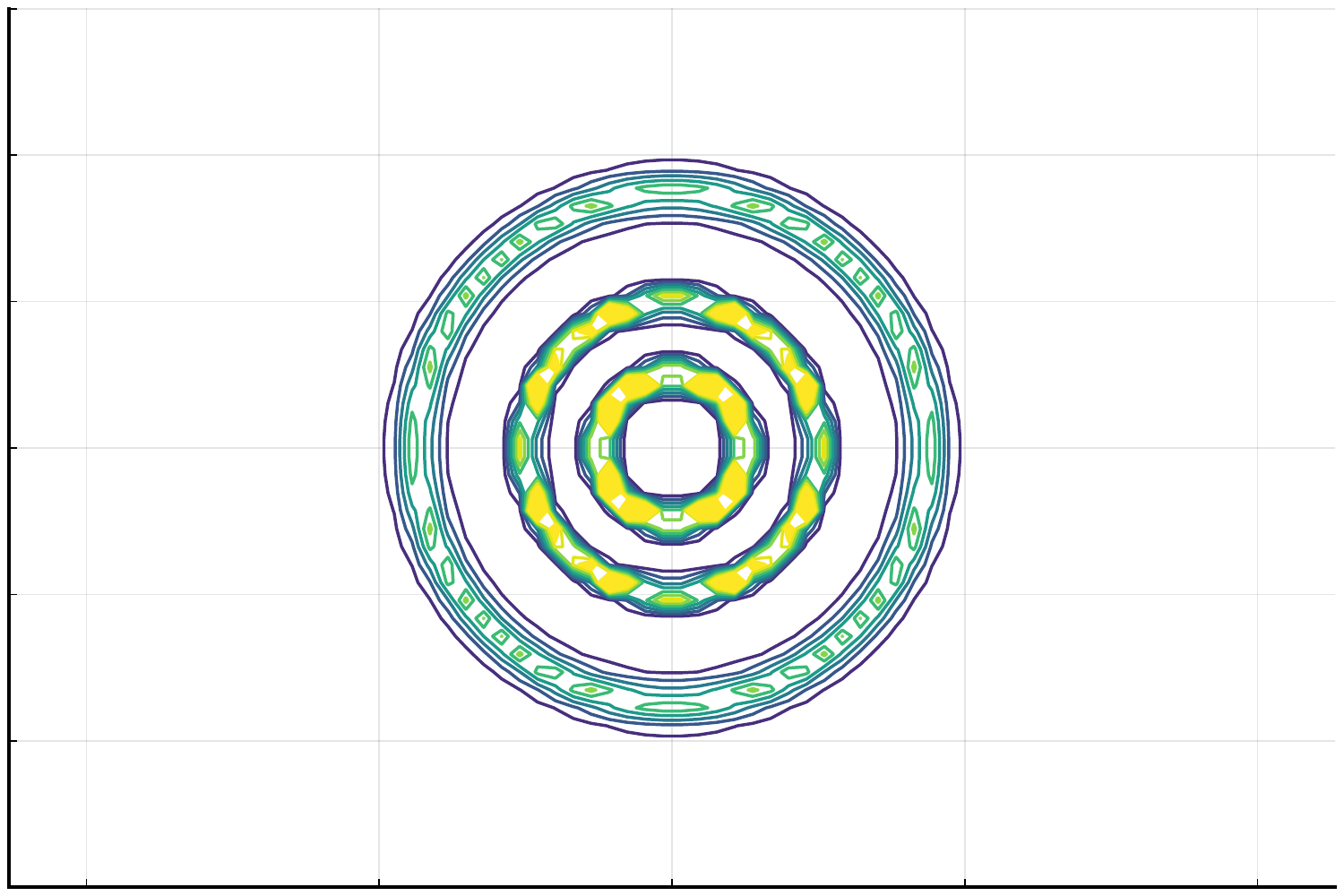} \\
         \includegraphics[width=0.19\textwidth]{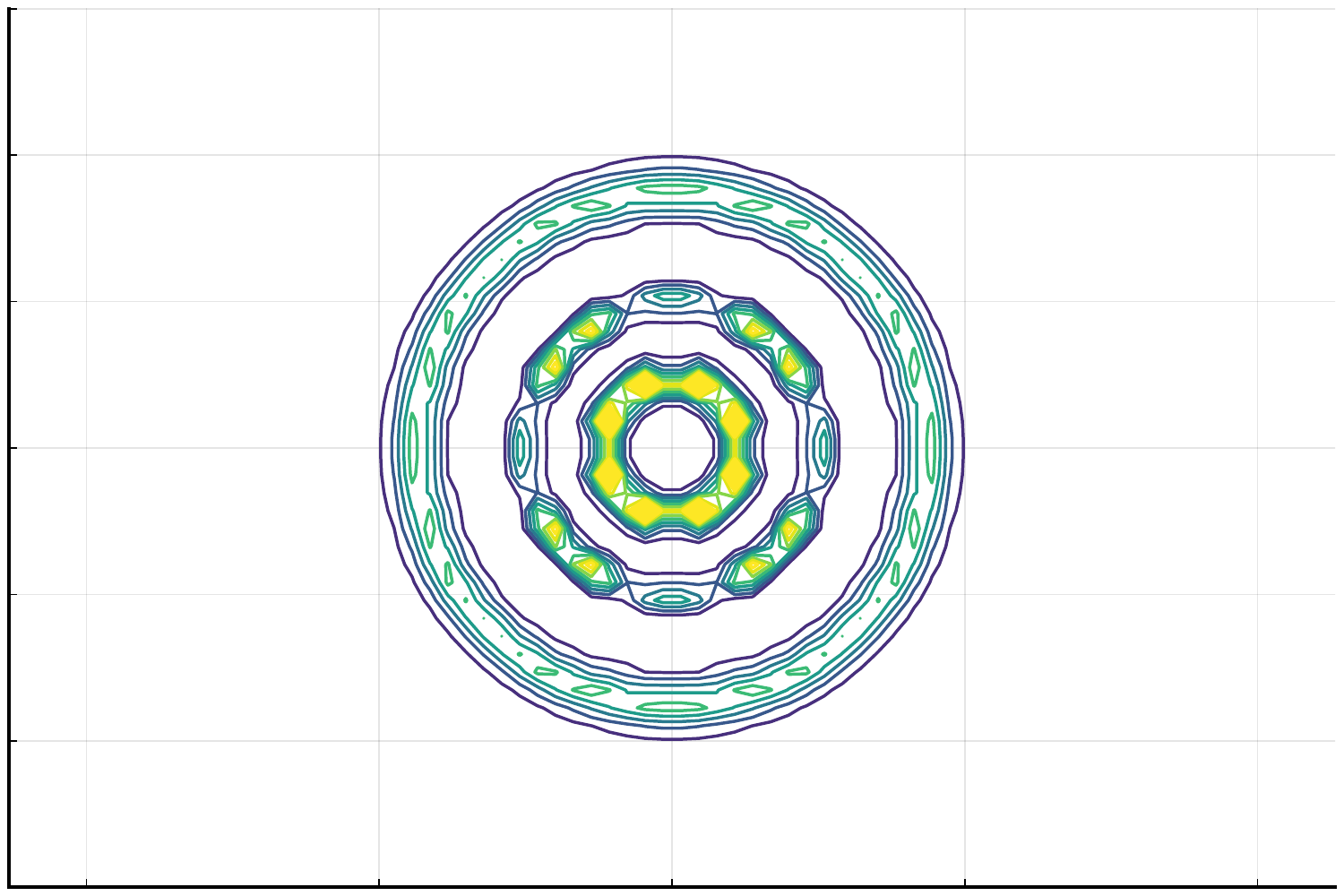}
    \end{tabular}} 
    \subfigure[$t=0.5$]{\begin{tabular}{@{}c@{}}
         \includegraphics[width=0.19\textwidth]{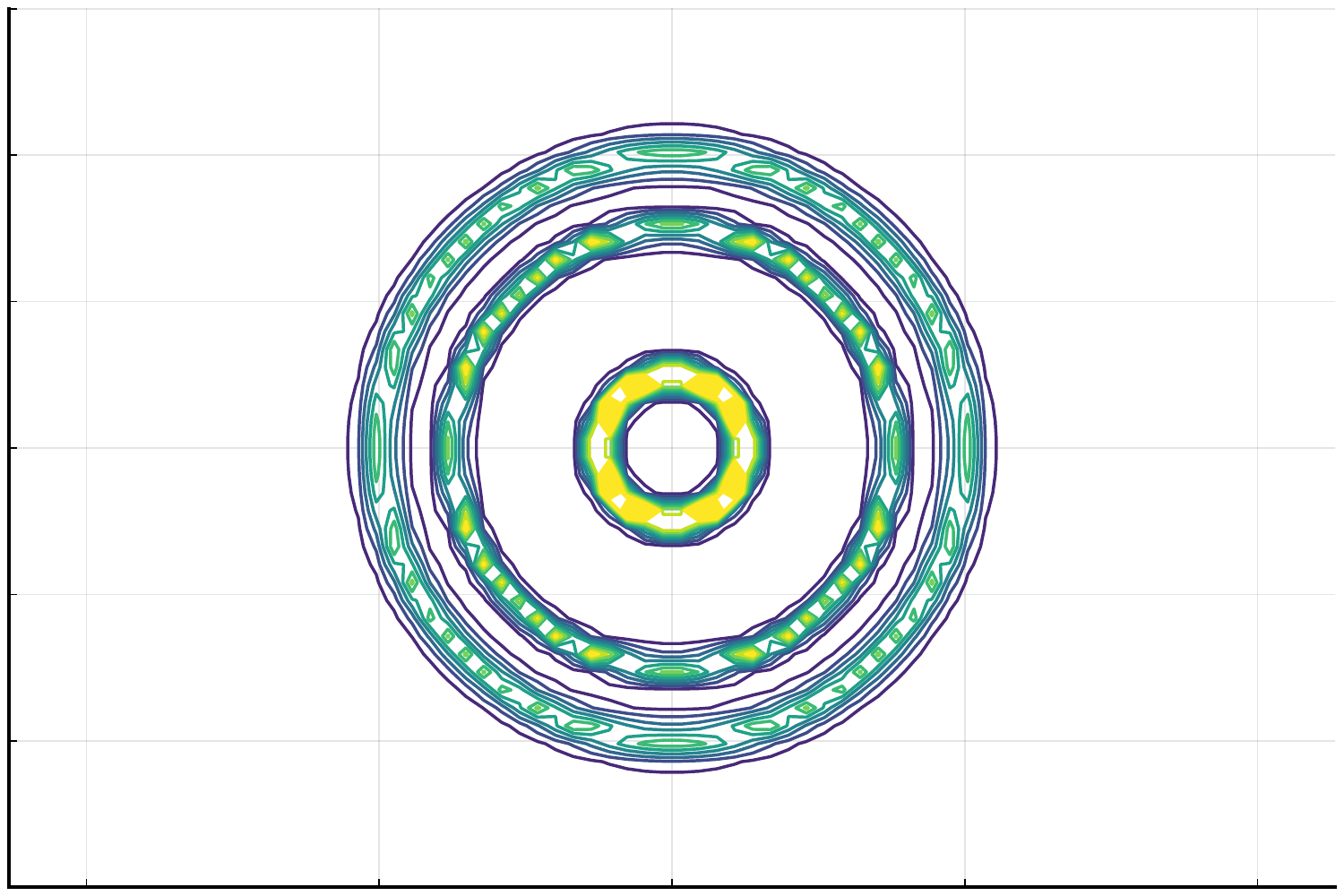} \\
         \includegraphics[width=0.19\textwidth]{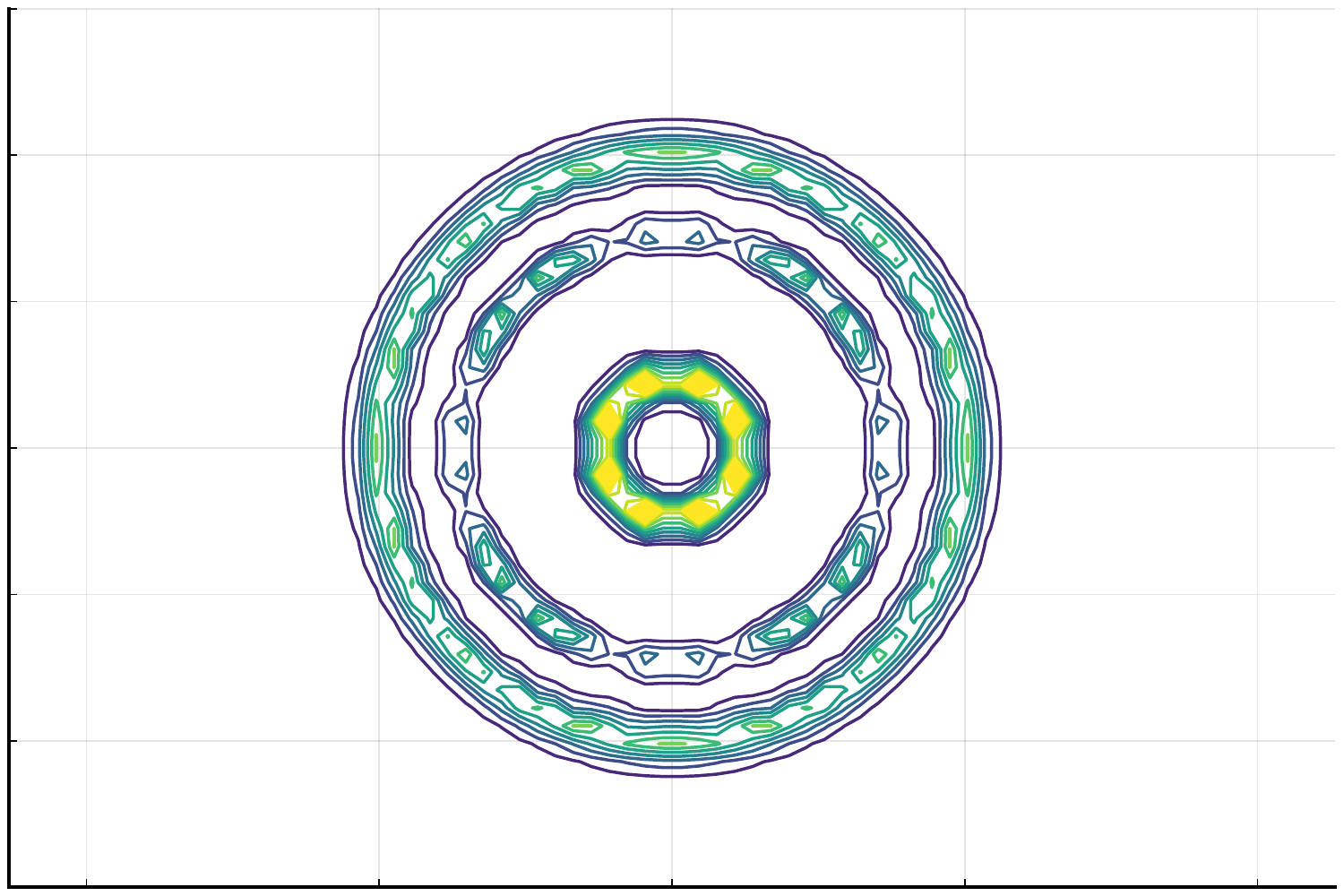}
    \end{tabular}} 
    \subfigure[$t=0.75$]{\begin{tabular}{@{}c@{}}
         \includegraphics[width=0.19\textwidth]{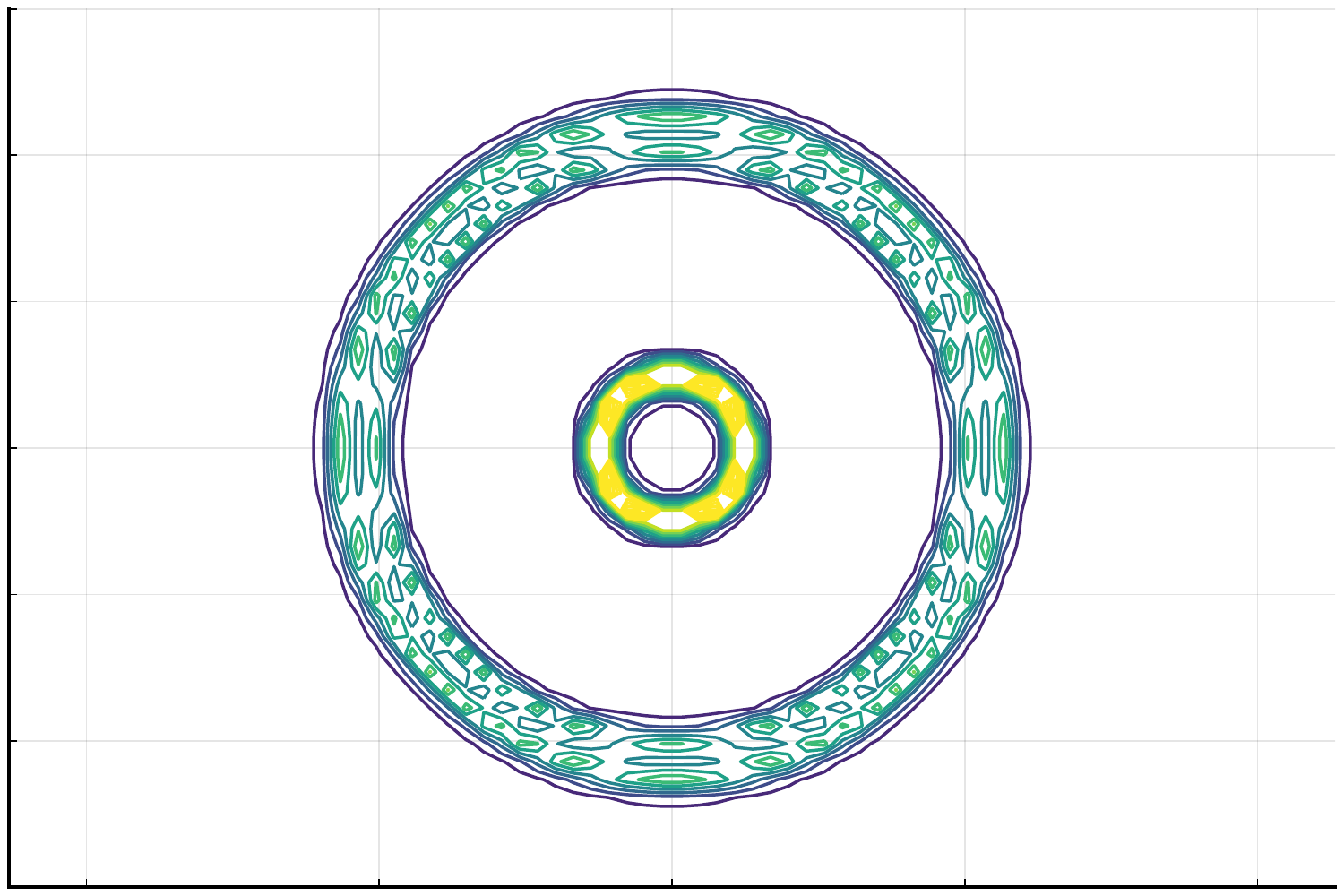} \\
         \includegraphics[width=0.19\textwidth]{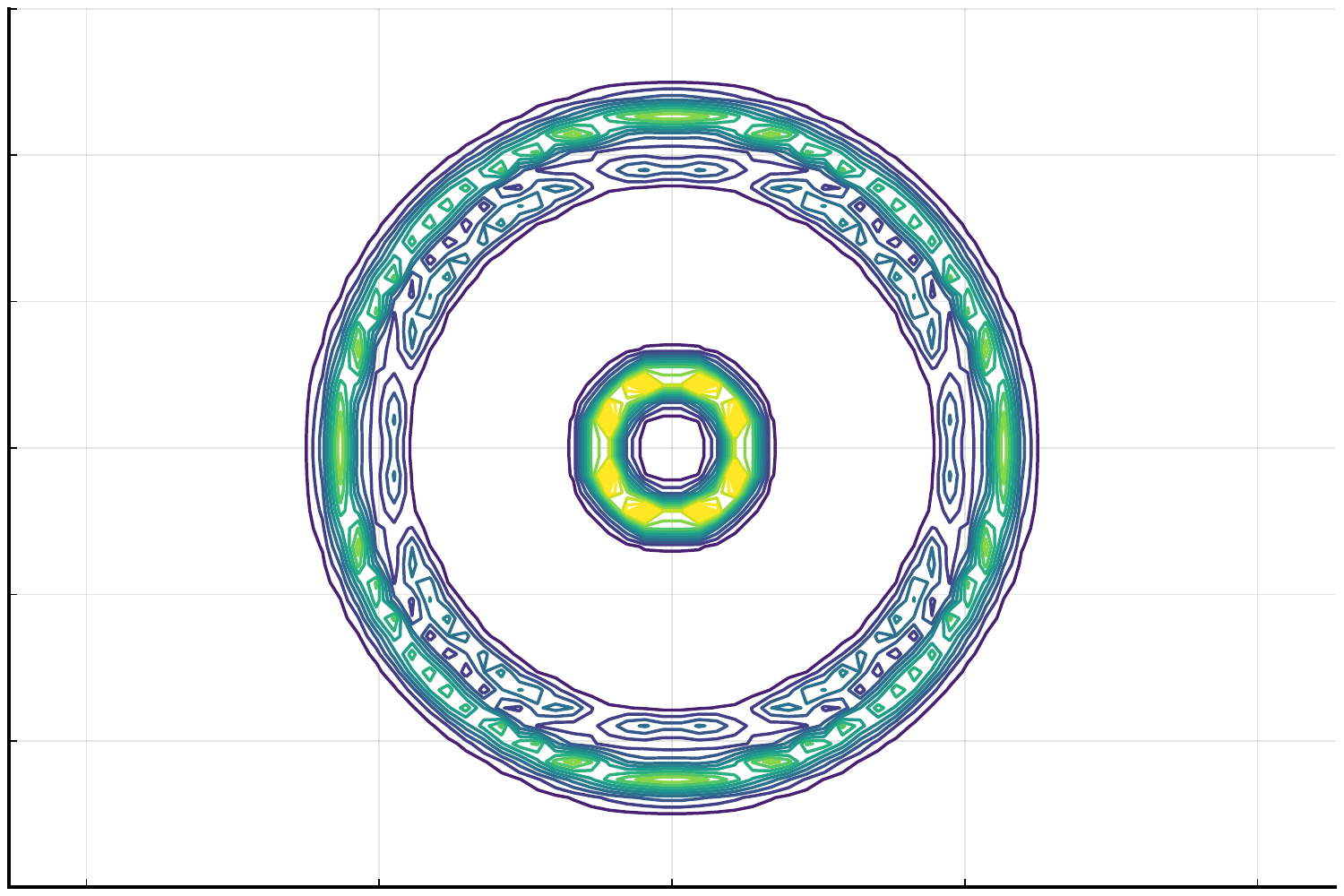}
    \end{tabular}} 
    \subfigure[$t=1$]{\begin{tabular}{@{}c@{}}
         \includegraphics[width=0.19\textwidth]{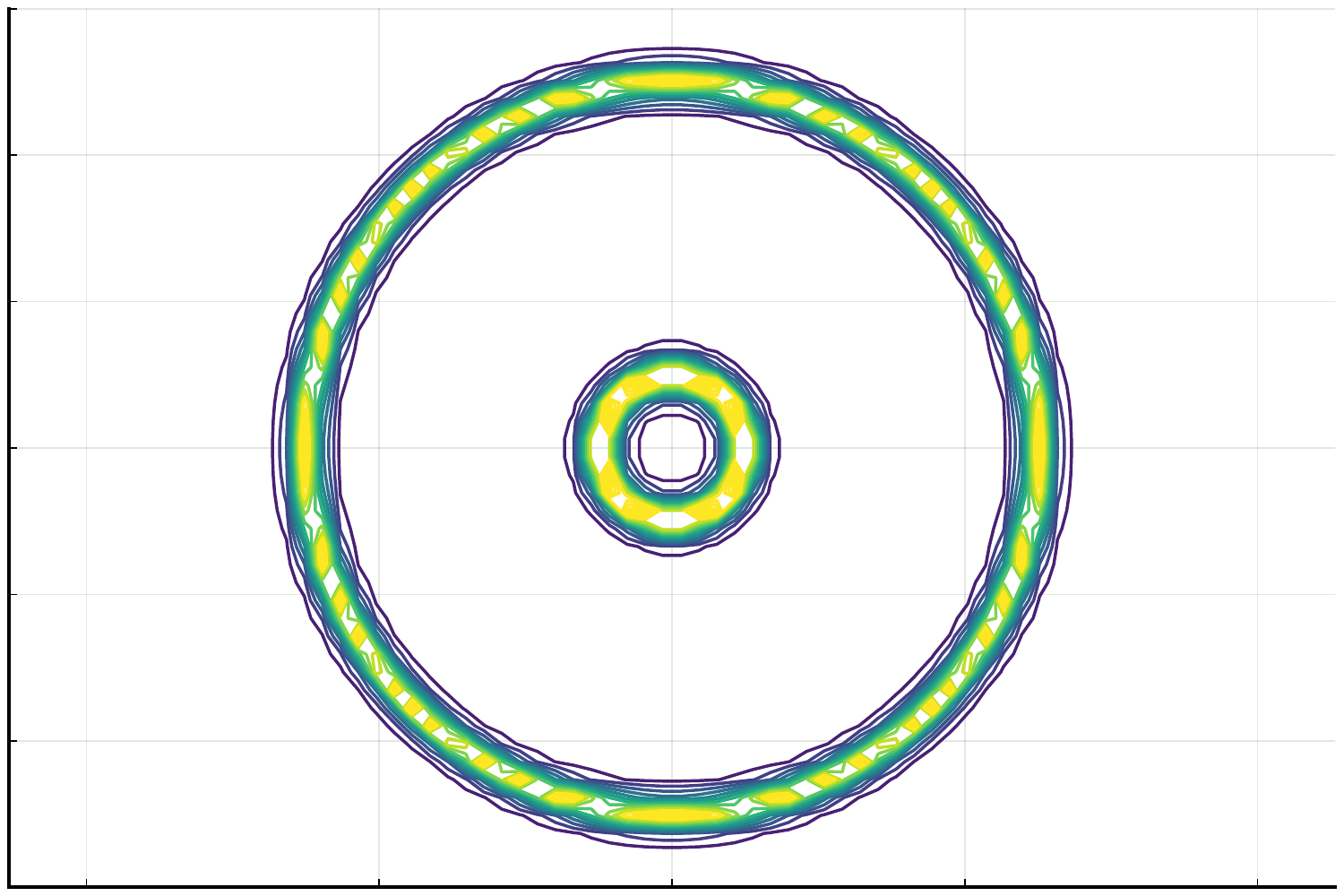} \\
         \includegraphics[width=0.19\textwidth]{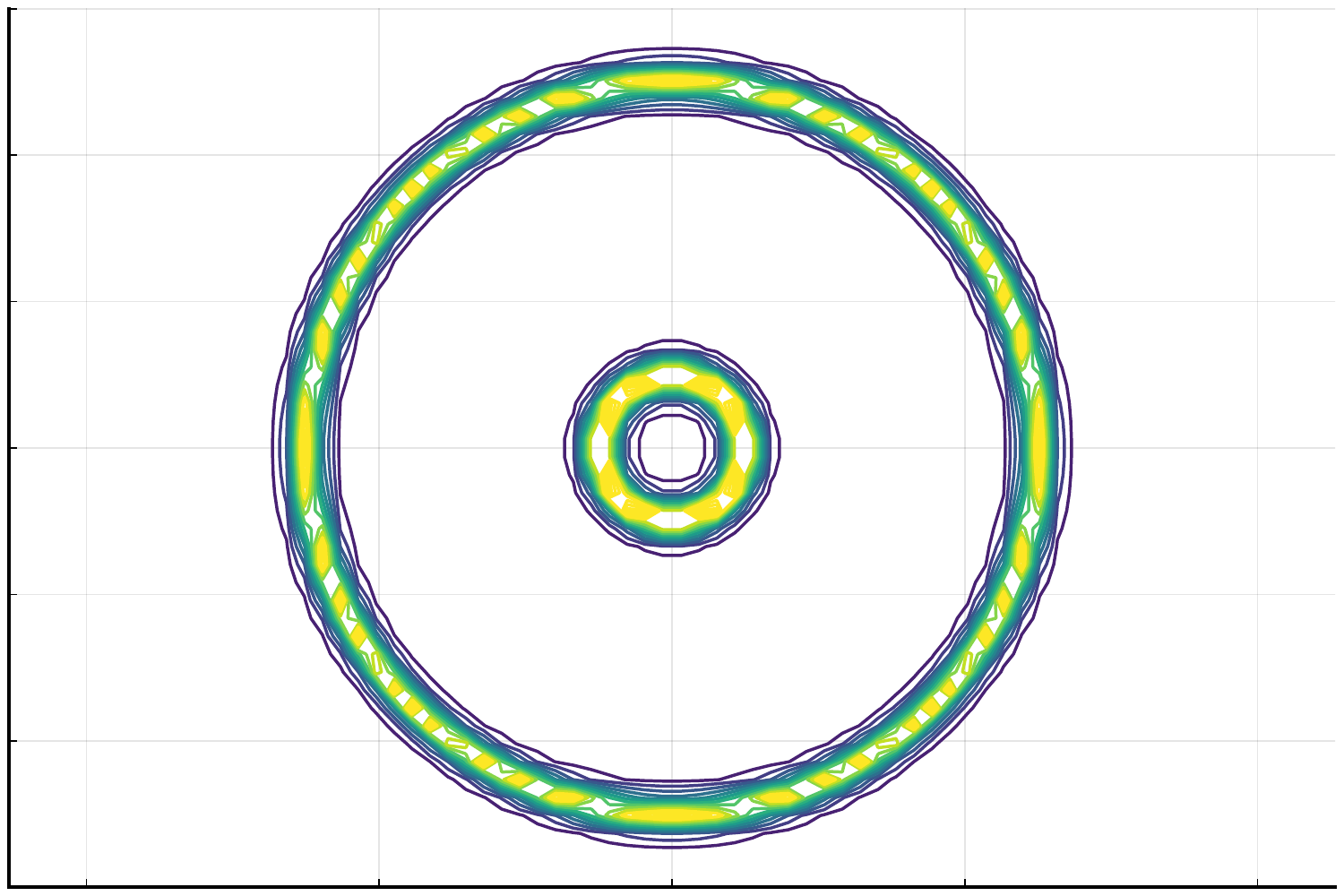}
    \end{tabular}} 
    \caption{Contour plots of $W_{2,\mathcal M}$ (top) and $W_2$ (bottom) barycenters between two mixtures of $SO(2)$-symmetric measures based on Slater-type distributions}
    \label{fig:rotsym2dmw2_contour}
\end{figure}

\section*{Acknowledgements}

This project has received funding from the
European Research Council (ERC) under the European Union's Horizon 2020
research and innovation programme (grant agreement EMC2 No 810367). This work
was supported by the French ‘Investissements d’Avenir’ program, project Agence
Nationale de la Recherche (ISITE-BFC) (contract ANR-15-IDEX-0003). GD was also supported by the Ecole des Ponts-ParisTech. VE acknowledges support from the ANR project COMODO (ANR-19-CE46-0002).

\bibliographystyle{siamplain}
\bibliography{biblio}

\end{document}